\documentclass[reqno,12pt]{amsart}
\usepackage{amsmath,amsfonts,amscd,amsthm,graphicx,amssymb,txfonts,epic,eepic,color}
\newtheorem{theorem}{Theorem}

\newtheorem{lemma}{Lemma}
\numberwithin{equation}{section}

\newcommand{\eps}{\varepsilon}
\newcommand{\Z}{{\mathbb{Z}}}

\newcommand{\Q}{{\mathbb{Q}}}
\newcommand{\PP}{{\mathbb{P}}}
\newcommand{\R}{{\mathbb{R}}}

\newcommand{\T}{{\mathbb{T}}}
\newcommand{\N}{{\mathbb{N}}}

\newcommand{\FF}{{\mathcal F}}

\newcommand{\II}{{\mathcal{I}}}
\newcommand{\JJ}{{\mathcal{J}}}
\renewcommand{\SS}{{\mathfrak{S}}}

\newcommand{\G}{\mathbb{G}}

\renewcommand{\AA}{\mathcal A}
\newcommand{\BB}{\mathcal B}
\newcommand{\CC}{\mathcal C}

\newcommand{\EE}{\mathcal E}

\newcommand{\hex}{\mathrm{hex}}
\newcommand{\sq}{\square}

\title[The linear flow in a honeycomb]{Distribution of the linear flow length in a honeycomb \\ in the
small-scatterer limit}

\author[Florin P. Boca]{Florin P. Boca }

\address{Department of Mathematics, University of Illinois, 1409 W. Green Street,
Urbana, IL 61801, USA}
\address{Institute of Mathematics of the Romanian Academy, P.O. Box 1-764,
RO-014700 Bucharest, Romania}

\email{fboca@illinois.edu}
\date{May 31, 2010}

\begin{document}

\begin{abstract}
We study the statistics of the linear flow in a punctured
honeycomb lattice, or equivalently the free motion of a particle
on a regular hexagonal billiard table with holes of equal size at
the corners and obeying the customary reflection rules. In the small-scatterer limit we prove the
existence of the limiting distribution of the free path length with randomly chosen origin of
the trajectory and explicitly compute it.
\end{abstract}

\maketitle

\section{Introduction}\label{Sect1}
From the regular hexagon of unit size remove circular
holes of small radius $\eps >0$ centered at the vertices,
obtaining the billiard table $H_\eps$ of area $\vert H_\eps\vert=\frac{3\sqrt{3}}{2}-2\pi
\eps^2$. For each pair $({\mathbf x},\omega)\in H_\eps \times
[0,2\pi]$ consider a point particle moving at unit speed on a
linear trajectory, with specular reflections when reaching the
boundary. The time $\tau_\eps^\hex ({\mathbf x},\omega)$
it takes the particle to reach one of the holes is
called the \emph{free path length} (or \emph{first exit time}).
Equivalently, one can consider the unit honeycomb tessellation of the
euclidean plane, with ``fat points" (obstacles or
scatterers) of radius $\eps$ centered at the vertices $m\textbf{e}_1+n\textbf{j}$, $m\nequiv n\pmod{3}$,
of the lattice $\Lambda_6 =\Z \textbf{e}_1+\Z \textbf{j}=\Z^2 \left( \begin{smallmatrix}
1 & 0 \\ 1/2 & \sqrt{3}/ 2 \end{smallmatrix} \right)$,
$\textbf{j} =\big(\frac{1}{2},\frac{\sqrt{3}}{2}\big)$,
and a particle moving at unit speed and velocity
$\omega$ on a linear trajectory until it hits one of the
obstacles (see Figure \ref{Figure1}). If the initial position ${\mathbf x}$ is always chosen
in a fundamental domain, the first hitting time coincides with
$\tau_{\eps}^\hex ({\mathbf x}, \omega)$. In this paper we are
interested in estimating the probability
\begin{equation}\label{1.1}
\PP_\eps^\hex (\xi)=\frac{1}{2\pi \vert H_\eps\vert}  \left|
\big\{ ({\mathbf x},\omega)\in H_\eps \times [0,2\pi]: \eps
\tau_\eps^\hex ({\mathbf x},\omega)>\xi \big\}\right|,\quad
\xi\in [0,\infty),
\end{equation}
that $\eps \tau_\eps^\hex ({\mathbf x},\omega) > \xi$
as $\eps\rightarrow 0^+$. We will prove that $\Phi^\hex ( \xi)
= \lim_{\eps\rightarrow 0^+} \PP_\eps^\hex (\xi)$ exists for
all $\xi\geqslant 0$ and show how to explicitly compute
$\Phi^\hex (\xi)$.

The version of this problem where the initial point is chosen to
be the center of the hexagon has been solved in \cite{BG}. The
square lattice analog of estimating \eqref{1.1} has a longer
history originating in the work of H. A. Lorentz \cite{Lor} and G.
P\' olya \cite{Po}. A complete solution was given in \cite{BZ2}. A
detailed history and presentation of various ideas and tools
involved in this and related problems, as well as a description of
recent developments in the study of the periodic Lorentz gas,
including \cite{CG2,Dah,Go1,MS1,MS2,MS3}, is provided in \cite{Go2,Mar}.

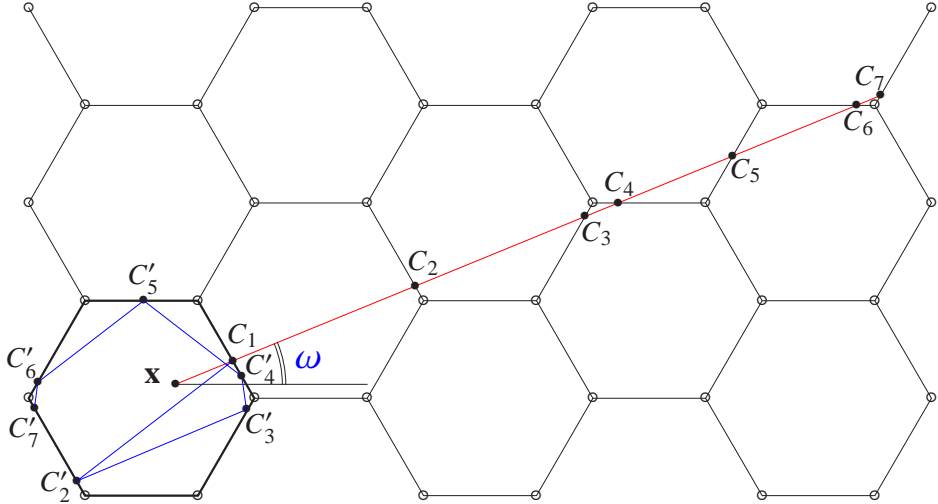
\begin{figure}[ht]
\centering
\unitlength 0.3mm
\begin{picture}(420,230)(-10,-43)

{\color{red} \path(65,6.1499)(377.53,134.29)}

\path(0,0)(25,-43.3)(75,-43.3)(100,0)(150,0)(175,-43.3)(225,-43.3)(250,0)(300,0)(325,-43.3)(375,-43.3)(400,0)
\path(0,0)(25,43.3)(75,43.3)(100,0)
\path(150,0)(175,43.3)(225,43.3)(250,0)
\path(300,0)(325,43.3)(375,43.3)(400,0)

\path(0,86.6)(25,43.3)
\path(75,43.3)(100,86.6)(150,86.6)(175,43.3)
\path(225,43.3)(250,86.6)(300,86.6)(325,43.3)
\path(375,43.3)(400,86.6)

\path(0,86.6)(25,129.9)(75,129.9)(100,86.6)
\path(150,86.6)(175,129.9)(225,129.9)(250,86.6)
\path(300,86.6)(325,129.9)(375,129.9)(400,86.6)

\path(0,173.2)(25,129.9)
\path(75,129.9)(100,173.2)(150,173.2)(175,129.9)
\path(225,129.9)(250,173.2)(300,173.2)(325,129.9)
\path(375,129.9)(400,173.2)

\put(65,6.1499){\makebox(0,0){\tiny $\bullet$}}
\put(90.43,16.576){\makebox(0,0){\tiny $\bullet$}}

\put(95,28){\makebox(0,0){\small $C_1$}}
\put(171.29,49.73){\makebox(0,0){{\tiny $\bullet$}}}
\put(175,60){\makebox(0,0){\small $C_2$}}
\put(21.29,-36.87){\makebox(0,0){{\tiny $\bullet$}}}
\put(12,-39){\makebox(0,0){\small $C_2^\prime$}}

\put(246.52,80.573){\makebox(0,0){{\tiny $\bullet$}}}
\put(252,73){\makebox(0,0){\small $C_3$}}
\put(96.52,-5.1){\makebox(0,0){{\tiny $\bullet$}}}
\put(103,-13){\makebox(0,0){\small $C_3^\prime$}}
\put(261.22,86.6){\makebox(0,0){{\tiny $\bullet$}}}
\put(262,95){\makebox(0,0){\small $C_4$}}
\put(94.39,9.717){\makebox(0,0){{\tiny $\bullet$}}}
\put(103,13){\makebox(0,0){\small $C_4^\prime$}}

\put(312.03,107.43){\makebox(0,0){{\tiny $\bullet$}}}
\put(318,101){\makebox(0,0){\small $C_5$}}
\put(50.94,43.3){\makebox(0,0){{\tiny $\bullet$}}}
\put(51,53){\makebox(0,0){\small $C_5^\prime$}}

\put(366.83,129.9){\makebox(0,0){{\tiny $\bullet$}}}
\put(368,122){\makebox(0,0){\small $C_6$}}
\put(4.1,7.074){\makebox(0,0){{\tiny $\bullet$}}}
\put(-3,16){\makebox(0,0){\small $C_6^\prime$}}

\put(377.53,134.29){\makebox(0,0){{\tiny $\bullet$}}}
\put(373,142){\makebox(0,0){\small $C_7$}}
\put(2.53,-4.39){\makebox(0,0){{\tiny $\bullet$}}}
\put(-2,-15){\makebox(0,0){\small $C_7^\prime$}}

\put(25,-43.3){\makebox(0,0){{\footnotesize $\circ$}}}
\put(75,-43.3){\makebox(0,0){{\footnotesize $\circ$}}}
\put(175,-43.3){\makebox(0,0){{\footnotesize $\circ$}}}
\put(225,-43.3){\makebox(0,0){{\footnotesize $\circ$}}}
\put(325,-43.3){\makebox(0,0){{\footnotesize $\circ$}}}
\put(375,-43.3){\makebox(0,0){{\footnotesize $\circ$}}}

\put(0,0){\makebox(0,0){{\footnotesize $\circ$}}}
\put(100,0){\makebox(0,0){{\footnotesize $\circ$}}}
\put(150,0){\makebox(0,0){{\footnotesize $\circ$}}}
\put(250,0){\makebox(0,0){{\footnotesize $\circ$}}}
\put(300,0){\makebox(0,0){{\footnotesize $\circ$}}}
\put(400,0){\makebox(0,0){{\footnotesize $\circ$}}}

\put(25,43.3){\makebox(0,0){{\footnotesize $\circ$}}}
\put(75,43.3){\makebox(0,0){{\footnotesize $\circ$}}}
\put(175,43.3){\makebox(0,0){{\footnotesize $\circ$}}}
\put(225,43.3){\makebox(0,0){{\footnotesize $\circ$}}}
\put(325,43.3){\makebox(0,0){{\footnotesize $\circ$}}}
\put(375,43.3){\makebox(0,0){{\footnotesize $\circ$}}}

\put(0,86.6){\makebox(0,0){{\footnotesize $\circ$}}}
\put(100,86.6){\makebox(0,0){{\footnotesize $\circ$}}}
\put(150,86.6){\makebox(0,0){{\footnotesize $\circ$}}}
\put(250,86.6){\makebox(0,0){{\footnotesize $\circ$}}}
\put(300,86.6){\makebox(0,0){{\footnotesize $\circ$}}}
\put(400,86.6){\makebox(0,0){{\footnotesize $\circ$}}}

\put(25,129.9){\makebox(0,0){{\footnotesize $\circ$}}}
\put(75,129.9){\makebox(0,0){{\footnotesize $\circ$}}}
\put(175,129.9){\makebox(0,0){{\footnotesize $\circ$}}}
\put(225,129.9){\makebox(0,0){{\footnotesize $\circ$}}}
\put(325,129.9){\makebox(0,0){{\footnotesize $\circ$}}}
\put(375,129.9){\makebox(0,0){{\footnotesize $\circ$}}}

\put(0,173.2){\makebox(0,0){{\footnotesize $\circ$}}}
\put(100,173.2){\makebox(0,0){{\footnotesize $\circ$}}}
\put(150,173.2){\makebox(0,0){{\footnotesize $\circ$}}}
\put(250,173.2){\makebox(0,0){{\footnotesize $\circ$}}}
\put(300,173.2){\makebox(0,0){{\footnotesize $\circ$}}}
\put(400,173.2){\makebox(0,0){{\footnotesize $\circ$}}}

{\color{blue} \put(124,17){\makebox(0,0){\large $\omega$}}}
\put(65,6.15){\arc{95}{-0.39}{0}}
\put(65,6.15){\arc{98}{-0.393}{0}}
\put(55,10){\makebox(0,0){${\mathbf x}$}}

\thinlines {\color{blue}
\path(90.43,16.576)(21.29,-36.87)(96.52,-5.1)(94.39,9.717)(50.94,43.3)(4.1,7.074)(2.53,-4.39)}

\path(65,6.14199)(150,6.14199)

\thicklines
\path(0,0)(25,43.3)(75,43.3)(100,0)(75,-43.3)(25,-43,3)(0,0)

\end{picture}

\caption{The free path in a hexagonal billiard and respectively in a
hexagonal lattice}\label{Figure1}
\end{figure}

One additional difficulty encountered here is the absence of a
theory of continued fractions in the case of the hexagonal
tessellation. To bypass this obstacle we shall deform this tessellation, as in
\cite{BG}, into $\Z^2_{(3)}=\{ (m,n)\in\Z^2, m\nequiv n
\pmod{3}\}$. The three-strip partition of the unit square
employed \cite{CG1,BZ2} in the situation of the square lattice, or
equivalently the corresponding tiling of $\R^2$ shown in Figure
\ref{Figure6}, will be useful here. However, the presence of
certain $\pmod{3}$ constraints translates here in the existence of
a positive proportion of angles with very long trajectories. This
leads to a large number of (non-redundant) cases that have to be
analyzed individually. The main result is

\begin{theorem}\label{THM}
There exists a decreasing continuous function $\Phi^\hex :[0,\infty)\rightarrow (0,\infty)$,
$\Phi^\hex (0)=1$, $\Phi^\hex (\infty)=0$, such that
for any $\delta > 0$, as $\eps\rightarrow 0^+$,
\begin{equation*}
{\mathbb P}_\eps^\hex (\xi)=\Phi^\hex (\xi)
+O_{\delta} \big( \eps^{\frac{1}{8}-\delta} \big)
,\quad \forall \xi \geqslant 0,
\end{equation*}
uniformly for $\xi$ in compact subsets of $[0,\infty)$. Moreover, there exist constants
$C_1,C_2 >0$ such that
\begin{equation}\label{1.2}
\frac{C_1}{\xi} \leqslant \Phi^\hex (\xi) \leqslant \frac{C_2}{\xi},\quad \forall \xi \in [1,\infty).
\end{equation}
\end{theorem}

Estimate \eqref{1.2} is discussed in Remark 2 of Section 5.
The repartition function $\Phi^\hex$ can be explicitly computed as
\begin{equation}\label{1.3}
\Phi^\hex (\xi)=\frac{4}{\pi^2}\ G\left( \frac{2\xi}{\sqrt{3}}\right),
\end{equation}
with $G(\xi)$ obtained by adding all terms
$\frac{\zeta(2)}{c_I} G_{I,Q}^{(*)}$ from formulas \eqref{5.22},
\eqref{6.6}, \eqref{6.7}, \eqref{6.8}, \eqref{6.9}, \eqref{6.10}, \eqref{6.11},
\eqref{7.1}, \eqref{7.2}, \eqref{7.5}, \eqref{7.6}, \eqref{7.8}, \eqref{7.11}, \eqref{7.12}, \eqref{7.13},
\eqref{8.1}, \eqref{8.2}, \eqref{8.5}, \eqref{8.6}, \eqref{8.8}, \eqref{8.9},
\eqref{8.10}, \eqref{8.11}, \eqref{8.12}, \eqref{8.13}, \eqref{8.14},
\eqref{9.1}, \eqref{9.2}, \eqref{9.4},
\eqref{9.5}, \eqref{9.6}, \eqref{9.7}, \eqref{9.9},
\eqref{9.10}, \eqref{9.11}, \eqref{9.12}, \eqref{9.13},
\eqref{9.14}, \eqref{9.15} and \eqref{9.16}.

In the case of the square lattice only the term from \eqref{5.22}
arises, with a different constant and no $\frac{2}{\sqrt{3}}$ scaling for $\xi$.
The limiting distribution $\Phi^\square$ also satisfies \eqref{1.2}.

\begin{figure}[ht]
\centering{\includegraphics*[scale=1.1, bb=0 0 250 160]{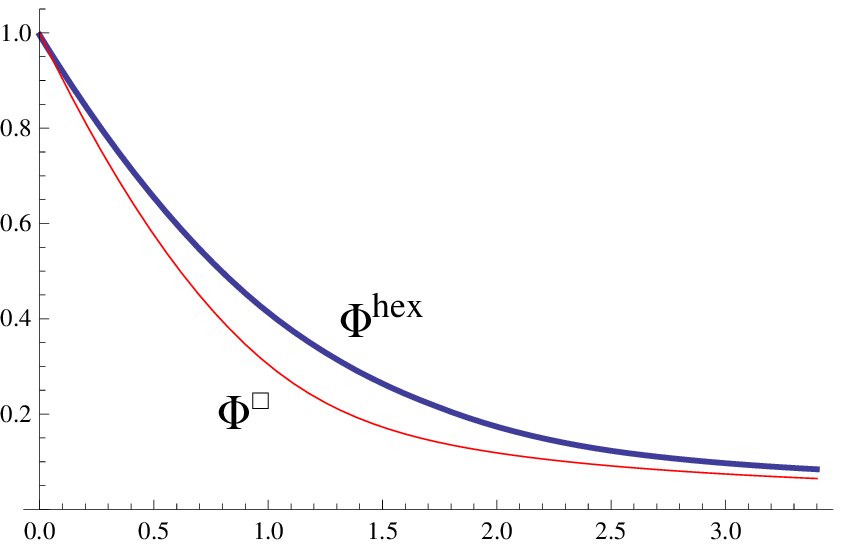}}
\caption{The limiting repartition functions $\Phi^\hex$ and $\Phi^\square$ }
\label{Figure2}
\end{figure}

In the case of a lattice it was actually proved in \cite{MS1} that for
every ${\mathbf x} \in \R^2 \setminus \Q^2$,
$\lim_{\eps\rightarrow 0^+} \frac{1}{2\pi} \big| \big\{ \omega \in [0,2\pi):
\eps \tau_\eps^\square ({\mathbf x} ,\omega) >\xi \big\} \big| =\Phi^\square (\xi)$.
It would be interesting to know whether a similar result holds true for a generic choice of ${\mathbf x}$
in the case of the honeycomb.

The analog problem about the free path length in a regular polygon with
$n$ sides ($n\neq 3,4,6$) seems to be out of reach at this time, due to lack of
a tractable coding for the linear flow. In the case of the regular octogon the recent
results in \cite{SU} may prove helpful.

\section{Translating the problem to the square lattice with $\hspace{-6pt}\mod{3}$
constraints}\label{Sect2}

For manifest symmetry reasons it suffices to consider ${\mathbf x}\in H_\eps$ and
$\omega \in \big[ 0,\frac{\pi}{6}\big]$, or
equivalently $t=\tan\omega \in \big[ 0,\frac{1}{\sqrt{3}}\big]$.
We will simply write $\tau_\eps^\hex ({\mathbf x } ,
\omega)=\tau^\hex_\eps ({\mathbf x},t)$. As in \cite{BG}
consider the lattice $\Z^2 M_0$, $M_0=\left( \begin{smallmatrix} 1 & 0 \\ 1/2 & \sqrt{3}/2 \end{smallmatrix} \right)$,
and the linear transformation $T{\mathbf x} ={\mathbf x} M_0^{-1}$ on $\R^2$:
\begin{equation*}
T(x,y)=\left( x-\frac{y}{\sqrt{3}},\frac{2y}{\sqrt{3}} \right)
=(x^\prime,y^\prime) .
\end{equation*}
This maps the vertices $\big(
q+\frac{a}{2},\frac{a\sqrt{3}}{2}\big)$ of the grid of equilateral
triangles of unit side onto the vertices $(q,a)$ of the square
lattice $\Z^2$. The vertices of the honeycomb are mapped exactly into
$\Z^2_{(3)}$, the subset of elements of $\Z^2$
with $q \nequiv a\pmod{3}$ (see Figure \ref{Figure3}). The
points of the $x$-axis are fixed by $T$. The circular scatterers
$S_{q,a,\eps}=(x_0,y_0)+\eps (\cos\theta,\sin\theta)$ with
$(x_0,y_0)=\big( q+\frac{a}{2},\frac{a\sqrt{3}}{2}\big)$ are
mapped onto ellipsoidal scatterers $(x_0^\prime,y_0^\prime)+\eps
\big( \cos \theta - \frac{\sin\theta}{\sqrt{3}} ,
\frac{2\sin\theta}{\sqrt{3}} \big)$ centered at
$(x_0^\prime,y_0^\prime)=(q,a)=T(x_0,y_0)$. The channel of width
$w=2\eps$, bounded by the two lines of slope $t=\tan\omega$ and
tangent to the circle $S_{q,a,\eps}$, is mapped (see Figure \ref{Figure4}) onto the channel of
width $w^\prime=2\eps^\prime \cos\omega^\prime$, bounded by the two
lines tangent to the ellipse $T(S_{q,a,\eps})$, of slope
$t^\prime=\tan\omega^\prime=\Psi (t)$, where
\begin{equation*}
\Psi :\left[ 0,\frac{1}{\sqrt{3}}\right] \rightarrow [0,1],\quad
t^\prime =\Psi (t)=\frac{2t}{\sqrt{3}-t},\quad
t=\Psi^{-1}(t^\prime)=\frac{t^\prime\sqrt{3}}{t^\prime+2} .
\end{equation*}
The intersection of these two channels and the $x$-axis is the horizontal segment
centered at the origin, of length
\begin{equation*}
\frac{2\eps}{\sin\omega} = \frac{w}{\sin\omega} =
\frac{w^\prime}{\sin\omega^\prime} =
\frac{2\eps^\prime}{\tan\omega^\prime} .
\end{equation*}
In particular
\begin{equation}\label{2.1}
\eps^\prime =\eps^\prime (\omega,\eps)=\frac{\eps
\tan\omega^\prime}{\sin \omega}=\frac{\eps}{\cos (\pi/6+\omega)} .
\end{equation}

\begin{center}
\begin{figure}[ht]
\centering{\includegraphics*[scale=1.1, bb=95 -10 337 102]{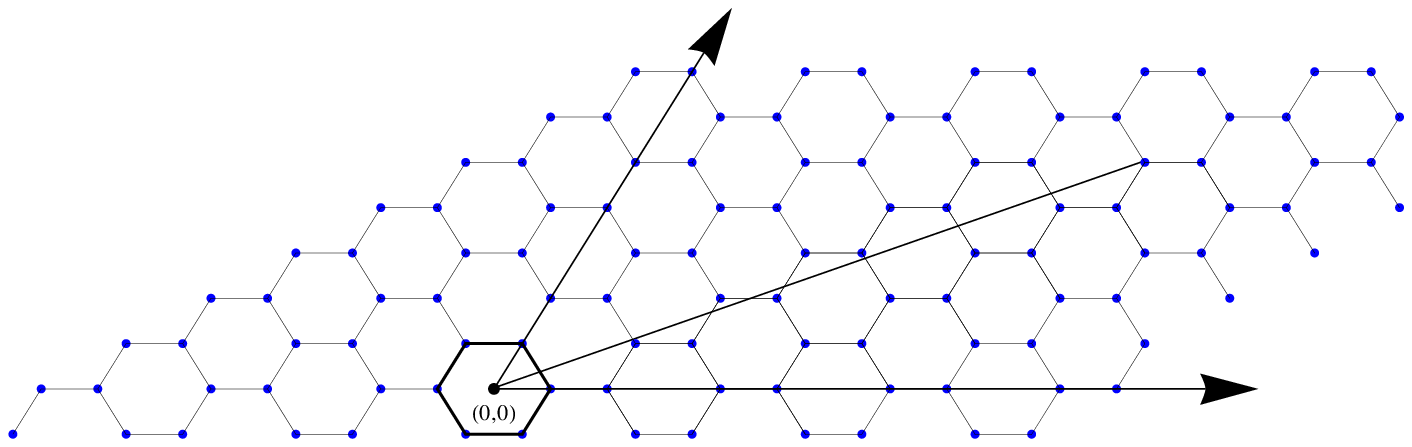}}
\centering{\includegraphics*[scale=0.5, bb=-50 0 520 270]{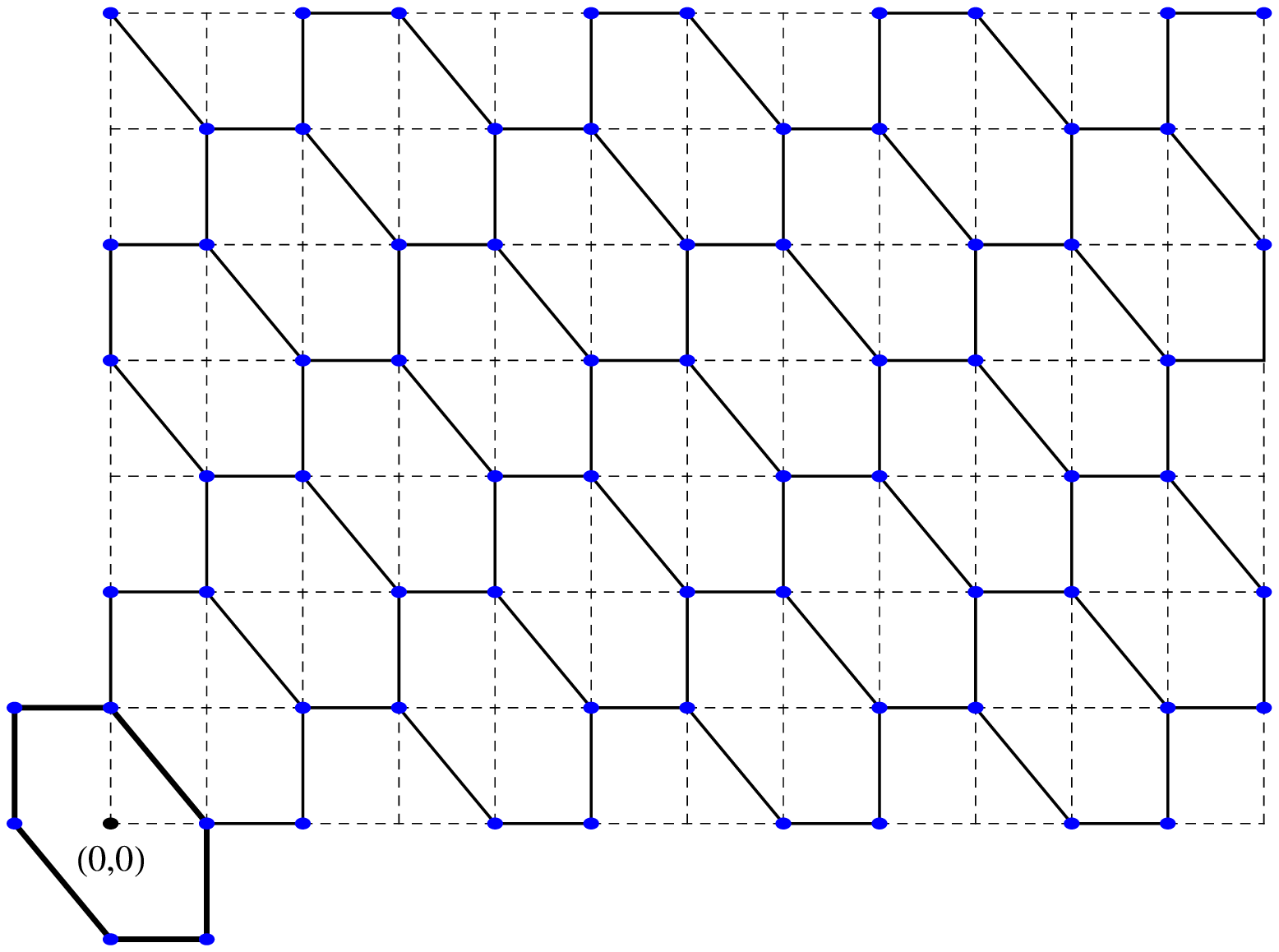}}
\caption{The free path length in the honeycomb and in
the deformed honeycomb } \label{Figure3}
\end{figure}
\end{center}

\begin{figure}[ht]
\centering{\includegraphics*[scale=1, bb=-40 0 250 200]{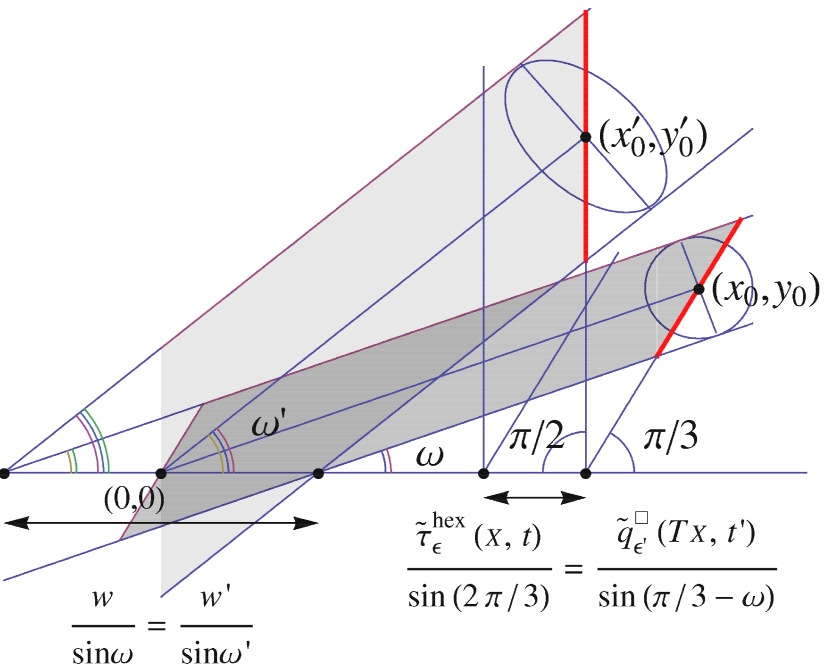}}
\caption{Change of scatterers under the linear transformation $T$ } \label{Figure4}
\end{figure}

We can first replace each circular scatterer $S_{q,a,\eps}$ by the
segment $\tilde{S}_{q,a,\eps}$ centered at $(x_0,y_0)$, of slope
$\frac{\pi}{3}$ and length
$\frac{2\eps}{\cos(\pi/6+\omega)}=2\eps^\prime$ (see Figures
\ref{Figure3} and \ref{Figure4}). Indeed, this change will result
in altering, for each $\omega$, the free path length
$\tau_\eps^\hex ({\mathbf x},t)$ to the free path length
$\tilde{\tau}_\eps^\hex ({\mathbf x},t)$ corresponding to the
later model by a quantity lesser than $4\sqrt{3} \eps^2$, which is
insignificant for the final result.

Next we apply $T$ to transport the problem from the honeycomb to
the square lattice with congruence $\hspace{-6pt}\pmod{3}$
constraints (or in the opposite direction through $T^{-1}$). The
unit regular hexagon $H$ centered at the origin is mapped to the hexagon
$T(H)$ which contains $(0,0)$ in Figure \ref{Figure4}. Actually it will be
more convenient to replace $T(H)$ by the fundamental domain $F$
consisting of the union of the square $[0,1)^2$ and of its
translates $[-1,0) \times [0,1)$ and $[0,1) \times [-1,0)$. Let
$\eps^\prime =\eps^\prime (\omega,\eps)$ be as in \eqref{2.1},
$t^\prime =\tan\omega^\prime$ as above, and consider the vertical segment
$V_{\eps^\prime}=\{ 0\} \times [-\eps^\prime,\eps^\prime]$. Consider
\begin{equation*}
\tilde{q}_{\eps^\prime}^\square ({\mathbf x}^\prime,t^\prime)=
\inf\{ n\in \N_0: {\mathbf x}^\prime +(n,nt^\prime)\in \Z^2_{(3)}+V_{\eps^\prime}\} ,
\end{equation*}
the \emph{horizontal free path length} in the square lattice with vertical
scatterers of (nonconstant) length $2\eps^\prime$ centered at
points $(x_0^\prime,y_0^\prime)=(q,a)\in \Z^2_{(3)}$. Consider also
$q_{\eps_0}^\square ({\mathbf x}^\prime,t^\prime)$, the
horizontal free path in the square lattice with vertical
scatterers of \emph{constant} length $2\eps_0$ centered at points
$(x_0^\prime,y_0^\prime)=(q,a)\in \Z^2_{(3)}$. Clearly
$q^\sq_{\eps_+}({\mathbf x}^\prime,t^\prime) \leqslant
\tilde{q}^\sq_{\eps^\prime} ({\mathbf x}^\prime,t^\prime)
\leqslant q^\sq_{\eps_-} ({\mathbf x}^\prime,t^\prime)$ when
$t^\prime$ belongs to an interval $I^\prime$ and $\eps_- \leqslant
\eps^\prime=\eps^\prime ( t^\prime) \leqslant \eps^+$, $\forall
t^\prime \in I^\prime$.

For each angle $\omega^\prime$ the transformation $T$ maps $H$ onto $F$ and
$T^{-1}$ preserves the structure of channels in the corresponding three-strip
partition from \cite{BK,BZ2,CG1} (see also the expository paper \cite{Go2}). Removal of vertical scatterers
$V_{q,a,\eps^\prime}=T(\tilde{S}_{q,a,\eps})$ with $q\equiv a
\pmod{3}$ in the $\Z^2_{(3)}$ picture results in dividing the
corresponding channel of the three-strip partition from the square
lattice model into several sub-channels, and in the occurrence of longer trajectories
associated with them. This is transported by $T^{-1}$ back to the
honeycomb model. The key observation here is that, by the Rule of
Sines,
\begin{equation*}
\frac{\tilde{\tau}^\hex_{\eps} ({\mathbf
x},t)}{\sin(2\pi/3)}=\frac{\tilde{q}^\square_{\eps^\prime}
(T{\mathbf x},t^\prime)}{\sin(\pi/3-\omega)}
=\frac{\tilde{q}^\square_{\eps^\prime} (T{\mathbf
x},t^\prime)}{\cos(\pi/6+\omega)}.
\end{equation*}
This shows that
\begin{equation}\label{2.2}
\tilde{\tau}^\hex_{\eps} ({\mathbf x},t) >\frac{\xi}{\eps}
\quad \Longleftrightarrow \quad \tilde{q}^\square_{\eps^\prime}
(T{\mathbf x},t^\prime) > \frac{2\xi\cos(\pi/6+\omega)}{\eps\sqrt{3}}
=\frac{\xi^\prime}{\eps^\prime} ,\quad
\xi^\prime:=\frac{2\xi}{\sqrt{3}},
\end{equation}
leading to
\begin{equation*}
\chi_{\big(\frac{\xi}{\eps},\infty\big)} \Big(
\tilde{\tau}_\eps^\hex ({\mathbf x},t)\Big) =\chi_{\big(
\frac{\xi^\prime}{\eps^\prime},\infty\big)} \Big(
\tilde{q}^\square_{\eps^\prime} (T{\mathbf x} , t^\prime) \Big) ,
\quad \forall {\mathbf x}\in H,\forall t\in \left[
0,\frac{1}{\sqrt{3}}\right].
\end{equation*}

For each interval $I=[\tan\omega_0,\tan\omega_1]\subseteq \big[
0,\frac{1}{\sqrt{3}}\big]$ one has
\begin{equation*}
\begin{split}
\eps^-_I:= \frac{\eps}{\cos(\pi/6+\omega_0)} \leqslant \eps^\prime &
=\frac{\eps}{\cos(\pi/6+\omega)} \leqslant
\eps^+_I:=\frac{\eps}{\cos(\pi/6+\omega_1)} ,\\ \eps^\pm_I
& = \big( 1+O\vert I\vert)\big) \eps .
\end{split}
\end{equation*}
Employing now \eqref{2.2} and the fact that $\eps \mapsto
\tilde{q}_\eps^\sq (T{\mathbf x},t^\prime)$ is non-decreasing, and
taking
\begin{equation*}
\xi^-_I:=\frac{\xi^\prime\eps^-_I}{\eps^+_I} \leqslant
\xi^\prime \leqslant \xi^+_I:=\frac{\xi^\prime
\eps^+_I}{\eps^-_I} ,\quad \xi^\pm_I = \big(
1+O(\vert I\vert)\big) \xi^\prime ,
\end{equation*}
we infer
\begin{equation*}
\begin{split}
\chi_{\big(\frac{\xi^+_I}{\eps^+_I},\infty\big)} \Big(
q_{\eps^+_I}^\sq (T{\mathbf x},t^\prime)\Big) & = \chi_{\big(
\frac{\xi^\prime}{\eps^-_I},\infty \big)} \Big(
q_{\eps^+_I}^\sq (T{\mathbf x},t^\prime)\Big) \leqslant
\chi_{\big( \frac{\xi^\prime}{\eps^\prime},\infty \big)} \Big(
q_{\eps^+_I}^\sq (T{\mathbf x},t^\prime)\Big)
\\ & \leqslant \chi_{\big(
\frac{\xi^\prime}{\eps^\prime},\infty\big)} \Big(
\tilde{q}^\sq_{\eps^\prime} (T{\mathbf x},t^\prime)\Big)
=\chi_{\big( \frac{\xi}{\eps},\infty\big)} \Big(
\tilde{\tau}^\hex_\eps ({\mathbf x},t)\Big)
\leqslant \chi_{\big( \frac{\xi^\prime}{\eps^\prime},\infty
\big)} \Big( q_{\eps^-_I}^\sq (T{\mathbf x} ,t^\prime ) \Big) \\ &
\leqslant \chi_{\big( \frac{\xi^\prime}{\eps^+_I},\infty
\big)} \Big( q_{\eps^-_I}^\sq (T{\mathbf x},t^\prime)\Big)
=\chi_{\big(\frac{\xi^-_I}{\eps^-_I},\infty\big)} \Big(
q_{\eps^-_I}^\sq (T{\mathbf x},t^\prime)\Big) .
\end{split}
\end{equation*}

Consider
\begin{equation}\label{2.3}
\G_{I,\eps}(\xi):=\int_{\Psi(I)} \frac{dt^\prime}{t^{\prime
2}+t^\prime +1} \int_F d{\mathbf x}^\prime \,  \chi_{\big(
\frac{\xi}{\eps},\infty\big)} \Big( q^\sq_{\eps} ({\mathbf
x}^\prime , t^\prime ) \Big).
\end{equation}
Applying the change of variable $({\mathbf x }^\prime , t^\prime)
= \big(T{\mathbf x},\Psi (t)\big)$ and employing \eqref{2.3},
$d{\mathbf x}=\frac{\sqrt{3}}{2} d{\mathbf x}^\prime$ and
$\frac{dt}{t^2+1}=\frac{\sqrt{3}}{2} \cdot
\frac{dt^\prime}{t^{\prime 2}+t^\prime+1}$ we infer
\begin{equation*}
\begin{split}
\frac{3}{4}\, \G_{I,\eps^+_I} (\xi^+_I) & \leqslant \frac{3}{4}
\int_{\Psi(I)} \frac{dt^\prime}{t^{\prime 2}+t^\prime +1} \int_F
d{\mathbf x}^\prime \,  \chi_{\big(
\frac{\xi^+_I}{\eps^+_I},\infty\big)} \Big( q^\sq_{\eps^+_I}
({\mathbf x}^\prime,t^\prime) \Big) \\ &  \leqslant
\widetilde{\mathbb P}_{I,\eps} (\xi):= \iint_{H\times I}
\chi_{\big( \frac{\xi}{\eps},\infty\big)} \Big(
\tilde{\tau}_\eps^\hex ({\mathbf x},\omega)\Big)\, d{\mathbf x}
d\omega \\ & =\int_I \frac{dt}{t^2+1} \int_H d{\mathbf x}\, \chi_{\big(
\frac{\xi}{\eps},\infty\big)}
\Big( \tilde{\tau}^\hex_{\eps} ({\mathbf x},t)\Big) \\
&  \leqslant \frac{3}{4} \int_{\Psi(I)}
\frac{dt^\prime}{t^{\prime 2}+t^\prime +1} \int_F d{\mathbf
x}^\prime \,  \chi_{\big( \frac{\xi^-_I}{\eps^-_I},\infty\big)}
\Big( q^\sq_{\eps^-_I} ({\mathbf x}^\prime,t^\prime)\Big) =
\frac{3}{4}\, \G_{I,\eps^-_I} (\xi^-_I).
\end{split}
\end{equation*}

To simplify notation we simply denote $\G_{I,1/(2Q)}$ by $\G_{I,Q}$ throughout.
We shall employ the following result, whose proof occupies the remaining part of the paper.

\begin{theorem}\label{T21}
Let $c,c^\prime >0$ such that $c+c^\prime <1$. For every interval
$I\subseteq [0,1]$ of length $\vert I\vert \asymp Q^c$, every
$\xi\geqslant 0$ and $\delta >0$, uniformly for $\xi$ in
compact subsets $K$ of $[0,\infty)$,
\begin{equation}\label{2.4}
\G_{I,Q}(\xi)=\frac{2 c_I}{\zeta(2)}\,
G(\xi)+O_{\delta,K}\big( E_{c,c^\prime,\delta}(Q)\big),
\end{equation}
where $G(\xi)$ is the 40 term sum described below \eqref{1.3} and
\begin{equation*}
c_I=\int_I \frac{dt}{t^2+t+1},\quad
c_{[0,1]}=\frac{\pi}{3\sqrt{3}},
\end{equation*}
\begin{equation}\label{2.5}
E_{c,c^\prime,\delta}(Q)=Q^{\max\big\{ 2 c^\prime - \frac{1}{2} ,
-c-c^\prime \big\}+\delta}.
\end{equation}
\end{theorem}

\begin{proof}{Proof of Theorem \ref{THM}.}
Let $Q^-=Q^-_I :=\Big\lfloor \frac{1}{2\eps_I^-}\Big\rfloor +1$
and $\frac{\eps_I^-}{1+2\eps_I^-} \leqslant \eps^-
:=\frac{1}{2Q^-}\leqslant \eps_I^-$. Then $q^\sq_{\eps_I^-}
({\mathbf x}^\prime,t^\prime) \leqslant q^\sq_{\eps^-} ({\mathbf
x}^\prime,t^\prime)$, so $\chi_{( \xi_I^-/\eps_I^-,\infty)}
\big(q^\sq_{\eps_I^-} ({\mathbf x}^\prime,t^\prime)\big) \leqslant \chi_{(
\xi_I^-/\eps_I^- ,\infty)} \big(q^\sq_{\eps_-} ({\mathbf x
}^\prime ,t^\prime)\big) $. Similarly, taking $Q^+=Q^+_I
:=\Big\lfloor \frac{1}{2\eps_I^+}\Big\rfloor$, $\eps_I^+ \leqslant
\eps^+:=\frac{1}{2Q^+} \leqslant \frac{\eps_I^+}{1-2\eps_I^+}$, we
have $\chi_{( \xi_I^+/\eps_I^+,\infty)} \big(
q^\sq_{\eps_I^+} ({\mathbf x}^\prime,t^\prime) \big) \geqslant
\chi_{( \xi_I^+/\eps_I^+,\infty)} \big( q^\sq_{\eps^+}
({\mathbf x}^\prime,t^\prime ) \big)$. On the other hand
$\frac{\eps^\pm}{\eps^\pm_I}=1+O(\eps)$, hence
\begin{equation*}
\frac{\xi_I^\pm}{\eps_I^\pm}
=\frac{\xi_I^\pm}{\eps^\pm}\cdot\frac{\eps^\pm}{\eps^\pm_I} =
\big( 1+O(\vert I\vert)\big)\,\frac{\xi^\prime}{\eps^\pm} .
\end{equation*}
We now infer
\begin{equation}\label{2.6}
\begin{split}
\G_{I,Q^+} \Big( \big( 1+O(\vert I\vert)\big)\xi^\prime \Big)
& \leqslant \frac{3}{4} \G_{I,\eps_I^+}(\xi_I^+) \leqslant
\widetilde{\mathbb P}_{I,\eps}(\xi) \\ & \leqslant \frac{3}{4}\,
\G_{I,\eps_I^-}(\xi_I^-) \leqslant \frac{3}{4}
\G_{I,Q^-} \Big( \big( 1+O(\vert I\vert)\big) \xi^\prime \Big).
\end{split}
\end{equation}

Partition now the interval $\big[ 0,\frac{1}{\sqrt{3}}\big]$ as a
union of $N=\big[\eps^{-c}\big]$ intervals
$I_j=[\tan\omega_j,\tan\omega_{j+1}]$ of equal length $\vert
I_j\vert=\frac{1}{N} \asymp \eps^c$, with $0<c<1$ to be chosen
later. As above, consider $Q_j^\pm=Q_{I_j}^\pm$. The intervals
$\Psi(I_j)$ partition $[0,1]$ and $\vert \Psi (I_j)\vert \asymp
\eps^c$. Applying \eqref{2.6}, Theorem \ref{T21} and the property
of $G$ of being Lipschitz on the compact $K$ we infer
\begin{equation*}
\begin{split}
{\mathbb P}_\eps^\hex (\xi) & =\frac{6}{\pi \vert H\vert}
\sum_{j=1}^N \widetilde{\mathbb P}_{I_j,\eps} (\xi)
=\frac{\sqrt{3}}{\pi} \sum_{j=1}^N \G_{I_j,Q_j^\pm} (\xi) \\
& = \frac{2\sqrt{3}}{\pi \zeta(2)} \sum_{j=1}^N c_{I_j} G\left(
\Big( 1+O(\vert I\vert )\Big)\frac{2\xi}{\sqrt{3}} \right)
+ O_{\delta,K} \big( N \eps^{-\max\{
2c_1-\frac{1}{2},-c-c_1\}} \big) \\ & =
\frac{2c_{[0,1]}\sqrt{3}}{\pi\zeta(2)}\, G\left(
\frac{2\xi}{\sqrt{3}}\right) +O_{\delta,K} \big(
\eps^{-c} + \eps^{-\max\{ c + 2 c_1 - \frac{1}{2} , - c_1
\}-\delta}\big),
\end{split}
\end{equation*}
uniformly for $\xi$ in compacts of $[0,\infty)$.
Taking $c=c_1=\frac{1}{8}$ we find
\begin{equation*}
{\mathbb P}_\eps^\hex (\xi) =\frac{4}{\pi^2}\, G\left(
\frac{2\xi}{\sqrt{3}}\right)+O_{\delta,K}
\big(\eps^{\frac{1}{8}-\delta}\big) ,
\end{equation*}
as stated in Theorem \ref{THM} and in \eqref{1.3}.
\end{proof}

\section{Some number theoretical estimates}\label{Sect3}
In this section we review and prove some number theoretical estimates that will be further
used to estimate certain sums over integer lattice points with congruence constraints.
The principal Dirichlet character $\hspace{-3pt} \pmod{\ell}$ will be denoted by $\chi_0$.
The number of divisors of $N$ is denoted by $\sigma_0 (N)$.

\begin{lemma}[Lemma 2.2 of \cite{BCZ}]\label{L3.1}
For each function $f\in C^1 [0,N]$ with total variation $T_0^N f$,
\begin{equation*}
\sum\limits_{\substack{1\leqslant q\leqslant N \\ (q,\ell)=1}} f(q) =
\frac{\varphi (\ell)}{\ell} \int_0^N f(x)\, dx +O \Big( (\| f\|_\infty +T_0^N f) \sigma_0 (N)\Big) .
\end{equation*}
\end{lemma}

\begin{lemma}[Lemma 2.1 of \cite{BG}]\label{L3.2}
For each function $V\in C^1 [0,N]$,
\begin{equation*}
\sum\limits_{\substack{1\leqslant q\leqslant N \\ (q,\ell)=1}}
\frac{\varphi(q)}{q} V(q) = C(\ell) \int_0^N V(x)\, dx+O_\ell
\Big( (\| V\|_\infty +T_0^N V)\log N\Big),
\end{equation*}
where
\begin{equation*}
C(\ell)=\frac{\varphi(\ell)}{\zeta(2)\ell} \prod\limits_{\substack{p\in {\mathcal P} \\ p\mid \ell}}
\left( 1-\frac{1}{p^2}\right)^{-1} =\frac{1}{\zeta(2)}
\prod\limits_{\substack{p\in {\mathcal P} \\ p\mid \ell}} \left( 1+\frac{1}{p}\right)^{-1} =
\frac{\varphi (\ell)}{\ell L(2,\chi_0)} .
\end{equation*}
\end{lemma}

We need a more precise form of Lemma \ref{L3.1}, as follows:

\begin{lemma}\label{L3.3}
Suppose that $(r,\ell)=1$. For each function $V\in C^1[0,N]$,
\begin{equation*}
\sum\limits_{\substack{1\leqslant q\leqslant N \\ q\equiv r
\hspace{-6pt} \pmod{\ell}}} \frac{\varphi (q)}{q} V(q) =
\frac{C(\ell)}{\varphi(\ell)} \int_0^N V(x)\, dx +O_\ell\Big( (\|
V\|_\infty +T_0^N V)\log N\Big).
\end{equation*}
In particular
\begin{equation*}
\sum\limits_{\substack{1\leqslant q\leqslant N \\ q\equiv \pm 1
\hspace{-6pt} \pmod{3}}} \frac{\varphi(q)}{q} V(q) \sim
\frac{3}{8\zeta(2)} \int_0^N
V(x)\, dx ,\quad \sum\limits_{\substack{1\leqslant q\leqslant N \\
q\equiv 0 \hspace{-6pt} \pmod{3}}} \frac{\varphi(q)}{q} V(q) \sim \frac{1}{4\zeta(2)}
\int_0^N V(x)\, dx.
\end{equation*}
\end{lemma}

\begin{proof}
When $(k,\ell)=1$ denote by $\bar{k}$ the multiplicative inverse
of $k\hspace{-3pt} \pmod{\ell}$. Let $G=U(\Z / \ell \Z)$ denote the multiplicative group of
units of $\Z /\ell \Z$ and $\widehat{G}$ be the group of characters $\chi:G \rightarrow \T$,
extended as multiplicative functions on $\N$.
Set $V_d(x)=V(xd)$. By Schur's
orthogonality relations for characters, for every $x,s\in\N$ with
$(s,\ell)=1$,
\begin{equation}\label{3.1}
\frac{1}{\varphi (\ell)} \sum_{\chi \in \widehat{G}} \chi (x)
\chi (\bar{s})=\frac{1}{\varphi (\ell)} \sum_{\chi \in
\widehat{G}} \chi (x)\, \overline{\chi (s)} =
\begin{cases} 1 & \mbox{\rm if $x\equiv s \pmod{\ell},$} \\
0 & \mbox{\rm if $x\nequiv s \pmod{\ell}.$} \end{cases}
\end{equation}
Taking $s=r$ and summing over $x=q\leqslant N$, we infer by M\"obius
summation, with $q=md$,
\begin{equation}\label{3.2}
\begin{split}
\sum\limits_{\substack{q\leqslant N\\ q\equiv r\hspace{-6pt}
\pmod{\ell}}} \frac{\varphi(q)}{q} V(q) & =
\frac{1}{\varphi(\ell)}  \sum_{q\leqslant N }
\frac{\varphi(q)}{q} V(q) \sum_{\chi\in \widehat{G}} \chi (q)
\chi(\bar{r}) \\ & = \frac{1}{\varphi(\ell)} \sum_{q\leqslant N}
\sum_{d\mid q} \frac{\mu(d)}{d}  V(q)\sum_{\chi\in \widehat{G}}
\chi (q)\chi(\bar{r}) \\ & =\frac{1}{\varphi(\ell)} \sum_{d\leqslant N}
\sum_{m=1}^{\lfloor N/d\rfloor} \sum_{\chi\in \widehat{G}}
\frac{\mu (d)}{d} V_d (m) \chi (md) \chi (\bar{r}) \\
& = \frac{1}{\varphi (\ell)} \sum_{\chi\in \widehat{G}} \chi
(\bar{r}) \sum_{d\leqslant N} \frac{\mu(d)\chi (d)}{d}
\sum_{m=1}^{\lfloor N/d\rfloor} V_d (m) \chi (m).
\end{split}
\end{equation}
We split the inner sum above according to whether $\chi=\chi_0$ or
$\chi \neq \chi_0$. Employing Lemma \ref{L3.1} for the function
$V_d$ we find that the contribution of the former is
\begin{equation*}
\begin{split}
\frac{1}{\varphi(\ell)} \sum_{d\leqslant N} \frac{\mu (d) \chi_0
(d)}{d} & \left( \frac{\varphi(\ell)}{\ell} \int_0^{\lfloor
N/d\rfloor} V_d +O\Big( (\|
V_d\|_\infty +T_0^{N/d} V_d)\log N\Big) \right)
\\ & =\frac{1}{\ell} \left(\ \sum_{d\leqslant N}
\frac{\mu(d)\chi_0(d)}{d^2}\right) \int_0^N V  +O_\ell \Big( (\|
V\|_\infty+T_0^N V)\log N\Big) \\ & =\left( \frac{1}{\ell L(2,\chi_0)}+O_\ell \Big( \frac{1}{N}\Big)\right)
\int_0^N V +O_\ell \Big( (\| V\|_\infty+T_0^N V)\log N\Big) \\ &
=\frac{C(\ell)}{\varphi (\ell)} \int_0^N V +O_\ell \Big( (\|
V\|_\infty+T_0^N V)\log N\Big).
\end{split}
\end{equation*}
When $\chi\neq \chi_0$ we find by partial summation and P\'
olya-Vinogradov (or a weaker inequality) that the innermost sum
in \eqref{3.2} is $\ll_\ell T_0^N V+\| V\|_\infty$, so the total
contribution of non-principal characters in \eqref{3.2} is
$\ll_\ell \| V\|_\infty \log N$. The proof is complete.
\end{proof}

\begin{lemma}[Proposition A4 of \cite{BZ1}]\label{L3.4}
Assume that $q\geqslant 1$ and $h$ are integers, $\II$ and $\JJ$
are intervals of length less than $q$, and $f\in C^1 (\II \times
\JJ)$. For any integer $T>1$ and any $\delta > 0$
\begin{equation*}
\sum\limits_{\substack{a\in\II,b\in\JJ \\ (a,q)=1 \\ ab \equiv
h\hspace{-5pt}\pmod{q}}} f(a,b)=\frac{\varphi (q)}{q^2}
\iint_{\II\times \JJ} f(x,y)\, dx\, dy +\EE,
\end{equation*}
with
\begin{equation*}
\EE \ll_\delta T^2 \| f\|_\infty q^{\frac{1}{2}+\delta}
(h,q)^{\frac{1}{2}} +T \| \nabla f\|_\infty q^{\frac{3}{2}+\delta}
(h,q)^{\frac{1}{2}} +\frac{\| \nabla f\|_\infty\vert \II \vert\,
\vert \JJ\vert}{T},
\end{equation*}
where $\| f\|_\infty$ and $\| \nabla f \|_\infty$ denote the
sup-norm of $f$ and respectively $\big| \frac{\partial f}{\partial
x}\big| +\big| \frac{\partial f}{\partial y}\big|$ on $\II
\times\JJ$.
\end{lemma}

For $q$, $\ell$ positive integers denote
\begin{equation*}
A(q,\ell)=\begin{cases} 1 & \mbox{\rm if $(q,\ell)=1,$} \\
\prod\limits_{\substack{p\in {\mathcal P} \\ p\mid (q,\ell)}}
\left( 1-\frac{1}{p}\right)^{-1} & \mbox{\rm if $(q,\ell)>1.$}
\end{cases}
\end{equation*}

\begin{lemma}\label{L3.5}
Suppose that $(r,\ell)=1$. For any interval $\II$, uniformly in
$\vert \II\vert$,
\begin{equation*}
\sum\limits_{\substack{x\in \II,\, (x,q)=1 \\ x\equiv r\hspace{-6pt}
\pmod{\ell}}} 1 = \frac{A(q,\ell)}{\ell}\cdot\frac{\varphi (q)}{q}
\, \vert \II\vert +O\left( \sigma_0 (q)\right) .
\end{equation*}
\end{lemma}

\begin{proof}
Without loss of generality assume $\II=[1,N]$. As in the
proof of Lemma \ref{L3.3} take $s=r$ and sum in \eqref{3.1} over $x\in \II$ with
$(x,q)=1$. By M\"obius inversion with $x=md$ we infer
\begin{equation}\label{3.3}
\begin{split}
\sum\limits_{\substack{x\in \II,\, (x,q)=1 \\ x\equiv r\hspace{-6pt}
\pmod{\ell}}} 1 & = \frac{1}{\varphi (\ell)} \sum\limits_{\substack{x\in \II \\
(x,q)=1}} \sum_{\chi \in \widehat{G}} \chi (x) \chi (\bar{r})
=\frac{1}{\varphi (\ell)} \sum_{\chi\in \widehat{G}}
\sum_{x\in \II} \chi (x)\chi (\bar{r}) \sum\limits_{\substack{d\mid q \\
d\mid x}}\mu (d) \\
& = \frac{1}{\varphi (\ell)} \sum_{d\mid q} \mu (d)
\sum\limits_{\substack{x\in \II \\ d\mid x}} \sum_{\chi\in
\widehat{G}} \chi (x)\chi (\bar{r}) =\frac{1}{\varphi(\ell)}
\sum_{d\mid q} \mu (d) \sum_{m=1}^{\lfloor N/d\rfloor}
\sum_{\chi\in \widehat{G}} \chi (md) \chi (\bar{r}).
\end{split}
\end{equation}
The contribution of $\chi=\chi_0$ to \eqref{3.3} is
\begin{equation}\label{3.4}
\begin{split}
\frac{1}{\varphi(\ell)} \sum_{d\mid q} \mu (d) \chi_0 (d)
\sum\limits_{\substack{1\leqslant m\leqslant \lfloor N/d\rfloor \\
(m,\ell)=1}} 1 & =\frac{1}{\varphi (\ell)} \sum_{d\mid q} \mu (d)
\chi_0 (d) \left( \frac{\varphi(\ell)}{\ell} \bigg(
\frac{N}{d}+O(1)\bigg) \right) \\ &  =\frac{\vert \II\vert}{\ell}
\sum_{d\mid q} \frac{\mu (d) \chi_0 (d)}{d}+O\left( \sigma_0
(q)\right) .
\end{split}
\end{equation}
In the contribution of non-principal characters to \eqref{3.3},
\begin{equation}\label{3.5}
\frac{1}{\varphi(\ell)} \sum_{d\mid q} \mu (d)
\sum\limits_{\substack{\chi\in \widehat{G} \\ \chi \neq \chi_0}}
\chi (d) \chi (\bar{r}) \sum_{m=1}^{\lfloor N/d\rfloor} \chi (m),
\end{equation}
the innermost sum is $\ll_\ell 1$ (by P\' olya-Vinogradov or a weaker
inequality), showing that the quantity in \eqref{3.5} is
$\ll_\ell \sigma_0 (q)$. The statement follows now because the sum
in \eqref{3.4} is equal to $\sum_{d\mid q}
\frac{\mu(d)}{d}=\frac{\varphi(q)}{q}$ when $(q,\ell)=1$, while
when $(q,\ell)>1$, writing $q=p_1^{\alpha_1} \cdots p_r^{\alpha_r} \tilde{q}$ with
$p_1,\ldots,p_r$ prime divisors of $\ell$ and $(\tilde{q},\ell)=1$, this sum is equal to
\begin{equation*}
\sum_{d\mid \tilde{q}} \frac{\mu (d)}{d} =
\frac{\varphi(\tilde{q})}{\tilde{q}} =\frac{\varphi (q)}{q}
\prod\limits_{i=1}^r
\left( 1-\frac{1}{p_i}\right)^{-1} = A(q,\ell)\,
\frac{\varphi(q)}{q} ,
\end{equation*}
as desired.
\end{proof}

We also need a slight extension of Lemma \ref{L3.4}.
Suppose that $(r,\ell)=1$ and denote by $\bar{x}$ the multiplicative
inverse of $x\pmod{\ell q}$ when $(x,\ell q)=1$. The Kloosterman type sums
\begin{equation*}
\begin{split}
K(m,n;\ell q) & :=\sum\limits_{\substack{x\hspace{-6pt}\pmod{\ell q} \\
(x,\ell q)=1}} e\left( \frac{mx+n\bar{x}}{\ell q} \right) ,\quad
\widetilde{K}_r (m,n;\ell q):= \sum\limits_{\substack{x\hspace{-6pt}\pmod{\ell q} \\
(x,q)=1,\, x\equiv r\hspace{-6pt}\pmod{\ell}}} e\left( \frac{mx+n\bar{x}}{\ell q}\right), \\
K_{\II,r}(m,n;\ell q) & :=\sum\limits_{\substack{x\in \II,\, (x,q)=1 \\
x\equiv r \hspace{-6pt} \pmod{\ell q}}} e\left(
\frac{mx+n\bar{x}}{\ell q}\right),
\end{split}
\end{equation*}
will be used to estimate
\begin{equation*}
\widetilde{N}_{q,\ell,r,h}(\II_1,\II_2)  :=\# \big\{ (x,y)\in \II_1
\times \II_2 : (x,q)=1, x\equiv r\hspace{-6pt}\pmod{\ell}, xy\equiv
h \hspace{-6pt} \pmod{\ell q}\big\}
\end{equation*}

\begin{lemma}\label{L3.6}
When $(r,\ell)=1$, for any interval $\II$ of length less than $q$,
\begin{equation*}
\vert K_{\II,r} (0,n;\ell q)\vert
\ll_{\ell,\delta} (n,q)^{\frac{1}{2}} q^{\frac{1}{2}+\delta}.
\end{equation*}
\end{lemma}

\begin{proof}
We write
\begin{equation*}
\begin{split}
K_{\II,r}(0,n;\ell q) & =\sum\limits_{\substack{x\in \II,\, (x,q)=1 \\
x\equiv r \hspace{-6pt} \pmod{\ell}}} e\left(
\frac{n\bar{x}}{\ell q} \right)
= \sum\limits_{\substack{x\hspace{-6pt}\pmod{\ell q}\\ (x,q)=1,\,
x \equiv r \hspace{-6pt} \pmod{\ell}}} e\left(
\frac{n\bar{x}}{\ell q}\right)\sum_{y\in \II} \frac{1}{\ell q}
\sum_{k\hspace{-6pt}\pmod{\ell q}} e\left( \frac{k(y-x)}{\ell
q}\right) \\ & =\frac{1}{\ell q} \sum_{k\hspace{-6pt} \pmod{\ell
q}} \widetilde{K}_r (-k,n;\ell q) \sum_{y\in \II} e\left( \frac{ky}{\ell
q}\right),
\end{split}
\end{equation*}
and
\begin{equation*}
\begin{split}
\vert \widetilde{K}_r (m,n;\ell q) \vert &
= \Bigg| \sum\limits_{\substack{x\hspace{-6pt}\pmod{\ell q} \\
(x,\ell q)=1}} e \left( \frac{mx+n\bar{x}}{\ell q}\right)
\frac{1}{\ell} \sum_{j\hspace{-6pt}\pmod{\ell}} e\left(
\frac{j(x-r)}{\ell}\right) \Bigg|  \\ & =\frac{1}{\ell} \Bigg|
\sum_{j\hspace{-6pt} \pmod{\ell}} e\left( -\frac{jr}{\ell}\right)
K(m+jq,n;\ell q) \Bigg| \leqslant \max_j \left| K(m+jq,n;\ell q)\right| .
\end{split}
\end{equation*}
Employing\footnote{Here $\| x\|=\operatorname{dist} (x,\Z)$, $x\in\R$, and $c_q(n)=K(n,0;q)=K(0,n;q)$ is the Ramanujan sum.}
\begin{equation}\label{3.6}
\Bigg| \sum_{y\in \II} e\left(\frac{ky}{\ell q}\right)\Bigg|
\leqslant  \min \Bigg\{ \vert \II\vert +1 , \frac{1}{2\big\| \frac{k}{\ell q} \big\|} \Bigg\},
\end{equation}
\begin{equation*}
\vert K(0,n;\ell q)\vert =\vert c_{\ell q} (n)\vert \leqslant
(n,\ell q)\leqslant (n,q) \ell \ll_\ell (n,q)^{\frac{1}{2}} q^{\frac{1}{2}},
\end{equation*}
and the Weil estimate
\begin{equation*}
\vert K(m+jq,n;\ell q)\vert \leqslant \sigma_0 (\ell q)
(m+jq,n,\ell q)^{\frac{1}{2}} (\ell q)^{\frac{1}{2}}
\ll_{\ell,\delta} (n,q)^{\frac{1}{2}} q^{\frac{1+\delta}{2}},
\end{equation*}
we infer
\begin{equation*}
\begin{split}
\vert K_{\II,r} (0,n;\ell q)\vert & \leqslant \frac{\vert \II\vert +1}{\ell q}
\, \vert \widetilde{K}_r (0,n;\ell q)\vert + \frac{1}{\ell q}
\sum\limits_{\substack{k\hspace{-6pt}\pmod{\ell q} \\ k\neq 0}}
 \frac{\vert \widetilde{K}_r (-k,n;\ell q)\vert}{\big\|
\frac{k}{\ell q}\big\|} \\ & \ll_{\ell,\delta}
(n,q)^{\frac{1}{2}} q^{\frac{1+\delta}{2}} +\frac{1}{2\ell q}\cdot (n,q)^{\frac{1}{2}}
q^{\frac{1+\delta}{2}} \ell q \log (\ell q) \ll_\ell
(n,q)^{\frac{1}{2}} q^{\frac{1}{2}+\delta},
\end{split}
\end{equation*}
as desired.
\end{proof}

\begin{lemma}\label{L3.7}
Suppose that $(r,\ell)=1$. For any intervals $\II_1$ and $\II_2$ of
length less than $q$, any integer $h$ and any $\delta >0$,
\begin{equation*}
\widetilde{N}_{q,\ell,r,h} (\II_1,\II_2) = \frac{A(q,\ell)}{\ell^2}
\cdot \frac{\varphi(q)}{q^2} \, \vert \II_1\vert\, \vert \II_2\vert
+O_{\ell,\delta} \Big( (h,q)^{\frac{1}{2}}
q^{\frac{1}{2}+\delta}\Big) .
\end{equation*}
\end{lemma}

\begin{proof}
We write
\begin{equation*}
\widetilde{N}_{q,\ell,r,h} (\II_1,\II_2) =\sum\limits_{\substack{x\in
\II_1,\, y\in \II_2 \\ (x,q)=1,\, x\equiv r \hspace{-6pt}\pmod{\ell}}}
\frac{1}{\ell q} \sum_{k\hspace{-6pt}\pmod{\ell q}} e\left(
\frac{k(y-h\bar{x})}{\ell q}\right) =M+E,
\end{equation*}
where $\bar{x}$ denotes the multiplicative inverse of
$x\pmod{\ell q}$ and
\begin{equation}\label{3.7}
\begin{split}
E = & \frac{1}{\ell q} \sum\limits_{\substack{k\hspace{-6pt}\pmod{\ell q} \\
k\neq 0}} \left( \,\sum_{y\in \II_2} e\left( \frac{ky}{\ell q} \right)
\right) K_{\II_1,r} (0,-hk;\ell q),\\ M = & \frac{1}{\ell q}
\sum\limits_{\substack{x\in \II_1,\,y\in \II_2 \\ (x,q)=1,\, x\equiv
r\hspace{-6pt}\pmod{\ell}}} 1 = \frac{1}{\ell q} \left(
\frac{A(q,\ell)}{\ell} \cdot \frac{\varphi(q)}{q}\, \vert \II_1
\vert +O\left( \sigma_0 (q)\right) \right) \Big( \vert \II_2\vert
+O(1)\Big) \\
= & \frac{A(q,\ell)}{\ell^2} \cdot \frac{\varphi(q)}{q^2} \, \vert
\II_1\vert\,\vert \II_2\vert +O_\ell \left( \sigma_0 (q)\right) .
\end{split}
\end{equation}
From \eqref{3.6}, Lemma \ref{L3.6} and $(hk,\ell q) \leqslant ( h
, q) (k,q)\ell $, we infer as in the proof of \cite[Proposition
A.3]{BZ1}
\begin{equation}\label{3.8}
\begin{split}
\vert E\vert & \ll_{\ell,\delta} q^{-\frac{1}{2}+\frac{\delta}{2}}
\sum\limits_{\substack{k\hspace{-6pt} \pmod{\ell q} \\ k\neq 0}}
\frac{(hk,\ell q)^{1/2}}{\big\| \frac{k}{\ell q}\big\|}
\ll_{\ell,\delta} q^{-\frac{1}{2}+\frac{\delta}{2}}
(h,q)^{\frac{1}{2}}
\sum\limits_{\substack{k\hspace{-6pt}\pmod{\ell q} \\ k\neq 0}}
\frac{(k,q)^{1/2}}{\big\| \frac{k}{\ell q}\big\|} \\
& \ll q^{\frac{1+\delta}{2}} (h,q)^{\frac{1}{2}} \sum_{1\leqslant
k\leqslant (q-1)/2} \frac{(k,q)^{1/2}}{k} \ll_\delta
(h,q)^{\frac{1}{2}} q^{\frac{1}{2}+\delta}.
\end{split}
\end{equation}
The statement follows now from \eqref{3.7} and \eqref{3.8}.
\end{proof}

\begin{lemma}\label{L3.8}
Suppose that $(r,\ell)=1$. Let $\II_1$, $\II_1$ be intervals of length
less than $q$, $f\in C^1(\II_1 \times \II_2)$, and $h\in\Z$. For
any integer $T>1$ and any $\delta > 0$,
\begin{equation*}
\sum\limits_{\substack{x\in \II_1, y\in \II_2 \\ (x,q)=1,\, x\equiv
r \hspace{-6pt}\pmod{\ell} \\
xy\equiv h \hspace{-5pt}\pmod{\ell q}}} f(x,y) =
\frac{A(q,\ell)}{\ell^2} \cdot \frac{\varphi (q)}{q^2} \iint_{\II_1
\times \II_2} f(u,v)\, du\, dv +\EE,
\end{equation*}
with
\begin{equation*}
\EE \ll_{\ell,\delta} T^2 \| f\|_\infty q^{\frac{1}{2}+\delta}
(h,q)^{\frac{1}{2}} +T \| \nabla f\|_\infty q^{\frac{3}{2}+\delta}
(h,q)^{\frac{1}{2}} +\frac{\| \nabla f\|_\infty\vert \II_1 \vert
\vert \II_2\vert}{T}.
\end{equation*}
\end{lemma}

\begin{proof}
This plainly follows from Lemma \ref{L3.7} as in the proof of \cite[Lemma 2.2]{BGZ0}.
\end{proof}

Only the case $\ell=3$ and $x\equiv r\nequiv 0\hspace{-3pt}\pmod{3}$
is needed here, with main term given by
\begin{equation*}
\sum\limits_{\substack{x\in \II_1 ,\, y\in \II_2 \\ (x,q)=1,\, x\equiv r
\hspace{-6pt} \pmod{3} \\ xy\equiv h \hspace{-6pt}\pmod{3q}}} f(x,y)
\sim \frac{\varphi(q)}{q^2}\left(
\iint_{\II_1\times \II_2} f(x,y)\, dx\, dy\right) \cdot \begin{cases} \frac{1}{9} &
\mbox{\rm if $3\nmid q,$} \\ \frac{1}{6} & \mbox{\rm if $3\mid
q.$} \end{cases}
\end{equation*}

\section{Coding the linear flow in $\Z^2_{(3)}$ and the three-strip partition}\label{Sect4}
To keep notation short denote
\begin{equation*}
x\vee y=\max\{ x,y\},\quad x\wedge y =\min\{ x,y\},\quad x_+
=\max\{ x,0\}.
\end{equation*}
Consider $c,c^\prime,\delta,\xi$ and $I$ as in the statement of
Theorem \ref{T21} and $E_{c,c^\prime,\delta} (Q)$ as in
\eqref{2.5}. For $(A_Q)$, $(B_Q)$ sequences of real numbers we
write $A_Q\approxeq B_Q$ when $A_Q = B_Q +O_{\delta,\xi} \big(
E_{c,c^\prime,\delta}(Q)\big)$, uniformly for $\xi$ in compact subsets of $[0,\infty)$.
Our primary aim is to estimate the quantity $\G_{I,Q} (\xi)$ from
Theorem \ref{T21}, associated with lattice points from $\Z^2_{(3)}$ with corresponding vertical
scatterers of width $2\eps=\frac{1}{Q}$, as $Q\rightarrow \infty$.

It is useful to recall first the approach and notation from
\cite{BZ2}. $\FF(Q)$ denotes the set of Farey fractions of order $Q$,
consisting of rational numbers
$\gamma=\frac{a}{q}$, $0<a\leqslant q\leqslant Q$, with $(a,q)=1$.
The interval $I$ will be first partitioned into
intervals $I_\gamma =(\gamma,\gamma^\prime)$ with
$\gamma,\gamma^\prime$ consecutive in $\FF_I(Q):=\FF (Q) \cap I$.
Each interval $I_\gamma$ is further partitioned into subintervals
$I_{\gamma,k}$, $k\in \Z$, defined as
\begin{equation*}
I_{\gamma,k}=(t_k,t_{k-1}],\quad I_{\gamma,0}=(t_0,u_0], \quad
I_{\gamma,-k} =(u_{k-1},u_k],\quad k\in \N,
\end{equation*}
where
\begin{equation*}
t_k=\frac{a_k-2\eps}{q_k},\quad u_k=\frac{a^\prime_k + 2
\eps}{q_k^\prime},\quad k\in \N_0,
\end{equation*}
and
\begin{equation*}
q_k=q^\prime +kq,\quad a_k=a^\prime+ka,\quad q_k^\prime
=q+kq^\prime,\quad a_k^\prime =a+ka^\prime,\quad k\in \Z,
\end{equation*}
satisfy the fundamental relations
\begin{equation*}
\begin{cases}
a_{k-1} q_k -a_k q_{k-1} =1=a_{k-1}q-aq_{k-1} ,\\
a_k^\prime q_{k-1}^\prime -a_{k-1}^\prime q_k^\prime =1=a^\prime
q^\prime_{k-1} -a^\prime_{k-1} q^\prime ,\quad k\in\Z,\\
2\eps q_k \wedge 2\eps q_k^\prime \geqslant 2\eps (q+q^\prime)
> 1,\quad k\geqslant 1.
\end{cases}
\end{equation*}
Consider also
\begin{equation*}
t:=\tan\omega \quad \mbox{\rm and} \quad \gamma_{k} =\frac{a_k}{q_k},\quad k\in \N .
\end{equation*}

As it will be seen shortly, the coding of the linear flow is
considerably more involved than in the case of the square lattice.
As a result our attempt of providing asymptotic results for the
repartition of the free path length will require additional
partitioning for each of the interval $I_{\gamma,k}$. For symmetry
reasons\footnote{Which are not geometrically obvious but become
apparent after translating $\G_{I,Q}(\xi)$ into sums involving
(sub)intervals (of) $I_{\gamma,k}$ and Farey fractions from
$\FF_I(Q)$.} the mediant intervals $(\gamma,\gamma_1]$ and
$(\gamma_1,\gamma^\prime]$ will contribute by the same amount to the
main term, so we shall only consider $t\in (\gamma,\gamma_1]$ and redefine
\begin{equation*}
I_{\gamma,0}:=(t_0,t_{-1}],\qquad t_{-1}:=\gamma_1 =\frac{a^\prime +a}{q^\prime +q} .
\end{equation*}
This explains the appearance of the factor $2$ in formula
\eqref{2.4}.

As in \cite[Section 3]{BZ2} we shall consider\footnote{Here we
use $t=\tan\omega$ as variable and use $\frac{dt}{t^2+t+1}$
instead of $d\omega$.}, when $t=\tan\omega\in I_{\gamma,k}$,
$k\in\N_0$,
\begin{equation*}
\begin{split}
& w_{\AA_0}(t) = a +2\eps -qt =q(u_0-t) ,\qquad w_{\CC_k}(t)=
a_{k-1}-2\eps-q_{k-1}t =q_{k-1} (t_{k-1}-t) ,\\ &
w_{\BB_k}(t) =q_k t -a_k+2\eps =q_k (t-t_k) \in [0,2\eps] ,
\end{split}
\end{equation*}
representing the widths of the bottom, center, and respectively
top channels $\AA_0$, $\CC_k$, $\BB_k$, of the three-strip
partition of $[0,1)^2$ (see Figures \ref{Figure5} and
\ref{Figure6}). Clearly
\begin{equation*}
\begin{cases}
2\eps = w_{\AA_0}(t) +w_{\BB_k}(t) +w_{\CC_k} (t), \\
1 = qw_{\AA_0}(t) +q_{k+1} w_{\CC_k}(t) + q_k w_{\BB_k}(t),
\end{cases} \qquad
\forall t \in I_{\gamma,k} .
\end{equation*}
Recall \cite{BZ2,CG1} that in the case of the square lattice the three
weights corresponding to $\omega$ are given by
\begin{equation}\label{4.1}
\begin{split}
& W_{\AA_0}(t)=(q-\xi Q)_+ w_{\AA_0}(t),\quad
W_{\BB_k}(t)=(q_k-\xi Q)_+ w_{\BB_k}(t),\\ &
W_{\CC_k}(t)=(q_{k+1}-\xi Q)_+ w_{\CC_k}(t).
\end{split}
\end{equation}
They reflect the area of the parallelogram of height given
respectively by $w_{\AA_0}$, $w_{\CC_k}$ or $w_{\BB_k}$, and
length given by the distance from $\xi Q$ to the bottom of the
corresponding sub-channel (if $\xi Q$ is lesser than the total
length of the sub-channel).

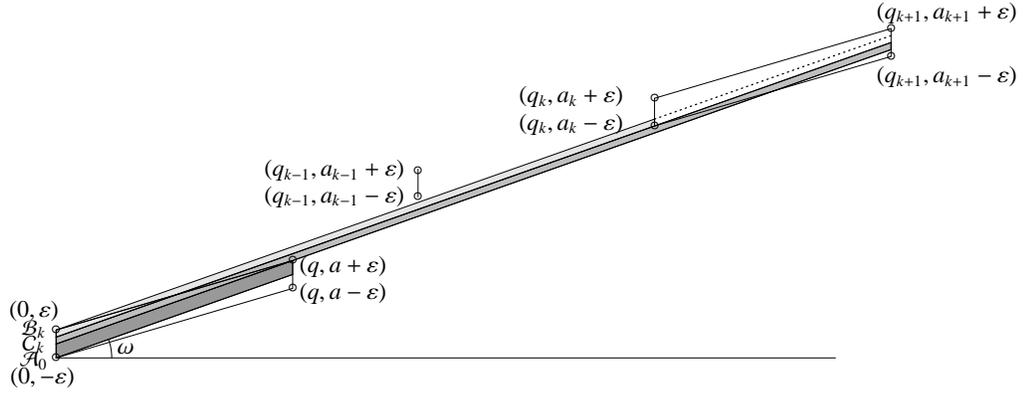
\begin{figure}[ht]
\centering
\unitlength 0.37mm
\begin{picture}(300,125)(-10,0)
\texture{cccc 0} \shade\path(0,0)(85,30)(85,35)(0,5)(0,0)

\texture{cccc 0000}
\shade\path(0,5)(300,110.882)(300,113.382)(0,7.5)(0,5)

\texture{c 0000}
\shade\path(0,7.5)(215,83.382)(215,85.8823)(0,10)(0,7.5)


\path(0,0)(280,0) \path(0,0)(85,25)(85,35)(0,10)(0,0)
\path(0,0)(85,30) \path(85,35)(0,10)
\path(215,83.3823)(215,93.3823)(300,118.382)(300,108.382)(215,83.3823)
\dottedline{2}(215,85.882)(300,115.8823)
\path(130,57.8823)(130,67.8823) \path(0,7.5)(300,113.3823)
\path(0,5)(300,110.88236)

\put(-5,-7){\makebox(0,0){{\scriptsize $(0,-\eps)$}}}
\put(-8,17){\makebox(0,0){{\scriptsize $(0,\eps)$}}}
\put(-8,-1){\makebox(0,0){{\tiny ${\mathcal A_0}$}}}
\put(-8,5){\makebox(0,0){{\tiny ${\mathcal C_k}$}}}
\put(-8,10.5){\makebox(0,0){{\tiny ${\mathcal B_k}$}}}
\put(0,0){\makebox(0,0){{\tiny $\circ$}}}
\put(0,10){\makebox(0,0){{\tiny $\circ$}}}
\put(85,25){\makebox(0,0){{\tiny $\circ$}}}
\put(85,35){\makebox(0,0){{\tiny $\circ$}}}

\put(300,108.382){\makebox(0,0){{\tiny $\circ$}}}
\put(300,118.382){\makebox(0,0){{\tiny $\circ$}}}
\put(215,83.3823){\makebox(0,0){{\tiny $\circ$}}}
\put(215,93.3823){\makebox(0,0){{\tiny $\circ$}}}

\put(130,58.05){\makebox(0,0){{\tiny $\circ$}}}
\put(130,67.3823){\makebox(0,0){{\tiny $\circ$}}}

\put(100,58.05){\makebox(0,0){{\scriptsize $(q_{k-1},a_{k-1}-\eps)$}}}
\put(100,67.3823){\makebox(0,0){{\scriptsize $(q_{k-1},a_{k-1}+\eps)$}}}

\put(103,33){\makebox(0,0){{\scriptsize $(q,a+\eps)$}}}
\put(103,23){\makebox(0,0){{\scriptsize $(q,a-\eps)$}}}

\put(320,101){\makebox(0,0){{\scriptsize $(q_{k+1},a_{k+1}-\eps)$}}}
\put(320,124){\makebox(0,0){{\scriptsize $(q_{k+1},a_{k+1}+\eps)$}}}

\put(185,94){\makebox(0,0){{\scriptsize $(q_k,a_k+\eps)$}}}
\put(185,84){\makebox(0,0){{\scriptsize $(q_k,a_k-\eps)$}}}
\put(25,3.5){\makebox(0,0){{\scriptsize $\omega$}}}
\put(0,0){\arc{40}{-0.34}{0}}
\end{picture}

\caption{The three-strip partition of $\R^2/\Z^2$ when
$t\in I_{\gamma,k}$} \label{Figure5}
\end{figure}

\begin{figure}[ht]
\centering
\unitlength 0.27mm
\begin{picture}(200,140)(20,0)
\texture{cccc 0000} \shade\path(0,14)(0,5)(141,52)(141,61)(0,14)

\texture{c 0000} \shade\path(0,14)(0,20)(90,50)(90,44)(0,14)

\texture{cccc 0} \shade\path(0,0)(0,5)(51,22)(51,17)(0,0)

\path(51,2)(51,22)(141,52) \path(141,46)(192,63)
\path(192,54)(102,24) \path(102,19)(51,2)

\path(90,44)(90,64)(180,94) \path(180,88)(231,105)
\path(231,96)(141,66) \path(141,61)(90,44)

\path(141,46)(141,66)(231,96) \path(231,90)(282,107)
\path(282,98)(192,68) \path(192,63)(141,46)

\path(-51,-2)(-51,18)(39,48) \path(39,42)(90,59)
\path(90,50)(0,20) \path(0,15)(-51,-2)

\path(39,42)(39,62)(129,92) \path(129,86)(180,103)

\path(51,22)(51,2) \path(90,44)(90,64) \path(141,46)(141,66)

\path(141,66)(141,46) \path(180,88)(180,108)\path(231,90)(231,110)

\path(102,24)(102,4) \path(141,46)(141,66) \path(192,48)(192,68)

\path(192,68)(192,48) \path(231,90)(231,110)\path(282,92)(282,112)

\path(39,42)(39,62) \path(90,44)(90,64)

\path(129,86)(129,106)

\Thicklines \path(0,0)(0,20)(90,50) \path(90,44)(141,61)
\path(141,52)(51,22) \path(51,17)(0,0)

\end{picture}

\caption{The tiling $\SS_\omega$ of the plane (shaded region represents $\R^2/\Z^2$)} \label{Figure6}
\end{figure}
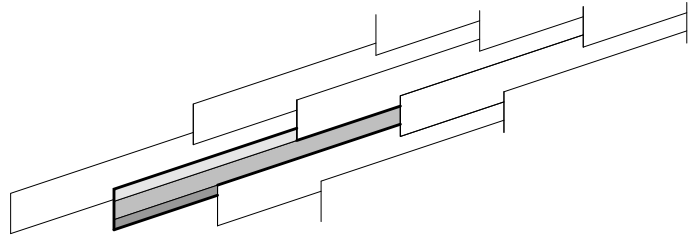

The range for $q_k=q^\prime + kq$, respectively $q_k^\prime$, will be
\begin{equation*}
q_k \in \II_{q,k}:=\big(Q+(k-1)q,Q+kq\big],\quad \mbox{\rm
respectively} \quad q^\prime_k \in \II_{q^\prime ,k}.
\end{equation*}

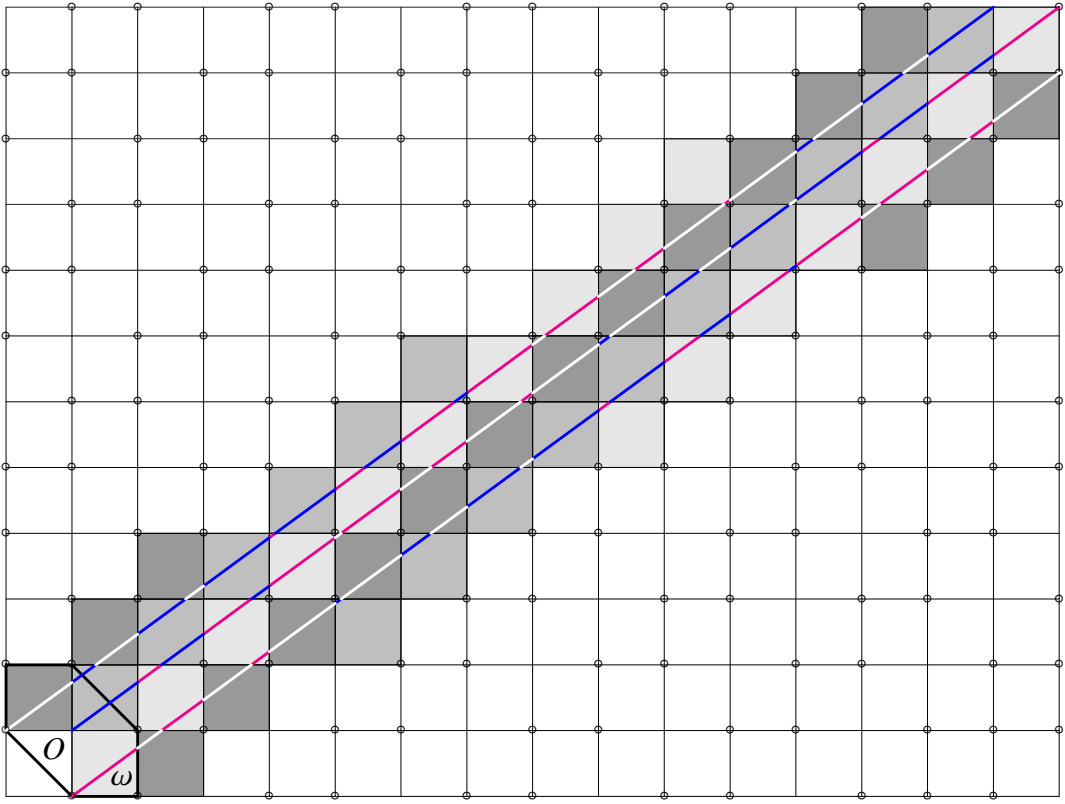
\begin{figure}[ht]
\center
\unitlength 0.35mm
\begin{picture}(550,310)(-40,0)

\texture{cccc 0000}

\shade\path(25,25)(50,25)(50,50)(25,50)(25,25)
\shade\path(50,50)(75,50)(75,75)(50,75)(50,50)
\shade\path(75,75)(100,75)(100,100)(75,100)(75,75)
\shade\path(100,100)(125,100)(125,125)(100,125)(100,100)
\shade\path(125,125)(150,125)(150,150)(125,150)(125,125)
\shade\path(150,150)(175,150)(175,175)(150,175)(150,150)

\shade\path(125,50)(125,75)(150,75)(150,50)(125,50)
\shade\path(150,75)(150,100)(175,100)(175,75)(150,75)
\shade\path(175,100)(175,125)(200,125)(200,100)(175,100)
\shade\path(200,125)(200,150)(225,150)(225,125)(200,125)
\shade\path(225,150)(225,175)(250,175)(250,150)(225,150)
\shade\path(250,175)(250,200)(275,200)(275,175)(250,175)
\shade\path(275,200)(275,225)(300,225)(300,200)(275,200)
\shade\path(300,225)(300,250)(325,250)(325,225)(300,225)
\shade\path(325,250)(325,275)(350,275)(350,250)(325,250)
\shade\path(350,275)(350,300)(375,300)(375,275)(350,275)

\texture{c 0000} \shade\path(50,25)(75,25)(75,50)(50,50)(50,25)
\shade\path(25,0)(25,25)(50,25)(50,0)(25,0)
\shade\path(75,50)(75,75)(100,75)(100,50)(75,50)
\shade\path(100,75)(100,100)(125,100)(125,75)(100,75)
\shade\path(125,100)(125,125)(150,125)(150,100)(125,100)
\shade\path(150,125)(150,150)(175,150)(175,125)(150,125)
\shade\path(175,150)(175,175)(200,175)(200,150)(175,150)
\shade\path(200,175)(200,200)(225,200)(225,175)(200,175)
\shade\path(225,200)(225,225)(250,225)(250,200)(225,200)
\shade\path(250,225)(250,250)(275,250)(275,225)(250,225)

\shade\path(225,125)(225,150)(250,150)(250,125)(225,125)
\shade\path(250,150)(250,175)(275,175)(275,150)(250,150)
\shade\path(275,175)(275,200)(300,200)(300,175)(275,175)
\shade\path(300,200)(300,225)(325,225)(325,200)(300,200)
\shade\path(325,225)(325,250)(350,250)(350,225)(325,225)
\shade\path(350,250)(350,275)(375,275)(375,250)(350,250)
\shade\path(375,275)(375,300)(400,300)(400,275)(375,275)

\texture{cccc 0}

\shade\path(0,25)(0,50)(25,50)(25,25)(0,25)
\shade\path(25,50)(25,75)(50,75)(50,50)(25,50)
\shade\path(50,75)(50,100)(75,100)(75,75)(50,75)

\shade\path(50,0)(50,25)(75,25)(75,0)(50,0)
\shade\path(75,25)(75,50)(100,50)(100,25)(75,25)
\shade\path(100,50)(100,75)(125,75)(125,50)(100,50)
\shade\path(125,75)(125,100)(150,100)(150,75)(125,75)
\shade\path(150,100)(150,125)(175,125)(175,100)(150,100)
\shade\path(175,125)(175,150)(200,150)(200,125)(175,125)
\shade\path(200,150)(200,175)(225,175)(225,150)(200,150)
\shade\path(225,175)(225,200)(250,200)(250,175)(225,175)
\shade\path(250,200)(250,225)(275,225)(275,200)(250,200)
\shade\path(275,225)(275,250)(300,250)(300,225)(275,225)
\shade\path(300,250)(300,275)(325,275)(325,250)(300,250)
\shade\path(325,275)(325,300)(350,300)(350,275)(325,275)

\shade\path(375,250)(375,275)(400,275)(400,250)(375,250)
\shade\path(350,225)(350,250)(375,250)(375,225)(350,225)
\shade\path(325,200)(325,225)(350,225)(350,200)(325,200)

\thinlines \path(0,200)(400,200) \path(0,175)(400,175)
\path(0,150)(400,150) \path(0,125)(400,125) \path(0,100)(400,100)
\path(0,75)(400,75) \path(0,50)(400,50) \path(0,25)(400,25)
\path(0,0)(400,0)

\path(0,0)(0,300) \path(400,0)(400,300) \path(25,0)(25,300)
\path(50,0)(50,300) \path(75,0)(75,300) \path(100,0)(100,300)
\path(125,0)(125,300) \path(150,0)(150,300) \path(175,0)(175,300)
\path(200,0)(200,300) \path(225,0)(225,300) \path(250,0)(250,300)
\path(275,0)(275,300) \path(300,0)(300,300) \path(325,0)(325,300)
\path(350,0)(350,300) \path(375,0)(375,300)


{\Thicklines
\path(0,25)(0,50)(25,50)(50,25)(50,0)(25,0)(0,25)}

\put(25,50){\makebox(0,0){{\tiny $\circ$}}}
\put(50,75){\makebox(0,0){{\tiny $\circ$}}}
\put(75,100){\makebox(0,0){{\tiny $\circ$}}}
\put(100,125){\makebox(0,0){{\tiny $\circ$}}}
\put(125,150){\makebox(0,0){{\tiny $\circ$}}}
\put(150,175){\makebox(0,0){{\tiny $\circ$}}}
\put(175,200){\makebox(0,0){{\tiny $\circ$}}}
\put(200,225){\makebox(0,0){{\tiny $\circ$}}}
\put(225,250){\makebox(0,0){{\tiny $\circ$}}}
\put(250,275){\makebox(0,0){{\tiny $\circ$}}}
\put(275,300){\makebox(0,0){{\tiny $\circ$}}}

\put(0,50){\makebox(0,0){{\tiny $\circ$}}}
\put(25,75){\makebox(0,0){{\tiny $\circ$}}}
\put(50,100){\makebox(0,0){{\tiny $\circ$}}}
\put(75,125){\makebox(0,0){{\tiny $\circ$}}}
\put(100,150){\makebox(0,0){{\tiny $\circ$}}}
\put(125,175){\makebox(0,0){{\tiny $\circ$}}}
\put(150,200){\makebox(0,0){{\tiny $\circ$}}}
\put(175,225){\makebox(0,0){{\tiny $\circ$}}}
\put(200,250){\makebox(0,0){{\tiny $\circ$}}}
\put(225,275){\makebox(0,0){{\tiny $\circ$}}}
\put(250,300){\makebox(0,0){{\tiny $\circ$}}}

\put(50,25){\makebox(0,0){{\tiny $\circ$}}}
\put(75,50){\makebox(0,0){{\tiny $\circ$}}}
\put(100,75){\makebox(0,0){{\tiny $\circ$}}}
\put(125,100){\makebox(0,0){{\tiny $\circ$}}}
\put(150,125){\makebox(0,0){{\tiny $\circ$}}}
\put(175,150){\makebox(0,0){{\tiny $\circ$}}}
\put(200,175){\makebox(0,0){{\tiny $\circ$}}}
\put(225,200){\makebox(0,0){{\tiny $\circ$}}}
\put(250,225){\makebox(0,0){{\tiny $\circ$}}}
\put(275,250){\makebox(0,0){{\tiny $\circ$}}}
\put(300,275){\makebox(0,0){{\tiny $\circ$}}}
\put(325,300){\makebox(0,0){{\tiny $\circ$}}}

\put(50,0){\makebox(0,0){{\tiny $\circ$}}}
\put(75,25){\makebox(0,0){{\tiny $\circ$}}}
\put(100,50){\makebox(0,0){{\tiny $\circ$}}}
\put(125,75){\makebox(0,0){{\tiny $\circ$}}}
\put(150,100){\makebox(0,0){{\tiny $\circ$}}}
\put(175,125){\makebox(0,0){{\tiny $\circ$}}}
\put(200,150){\makebox(0,0){{\tiny $\circ$}}}
\put(225,175){\makebox(0,0){{\tiny $\circ$}}}
\put(250,200){\makebox(0,0){{\tiny $\circ$}}}
\put(275,225){\makebox(0,0){{\tiny $\circ$}}}
\put(300,250){\makebox(0,0){{\tiny $\circ$}}}
\put(325,275){\makebox(0,0){{\tiny $\circ$}}}
\put(350,300){\makebox(0,0){{\tiny $\circ$}}}

\put(100,0){\makebox(0,0){{\tiny $\circ$}}}
\put(125,25){\makebox(0,0){{\tiny $\circ$}}}
\put(150,50){\makebox(0,0){{\tiny $\circ$}}}
\put(175,75){\makebox(0,0){{\tiny $\circ$}}}
\put(200,100){\makebox(0,0){{\tiny $\circ$}}}
\put(225,125){\makebox(0,0){{\tiny $\circ$}}}
\put(250,150){\makebox(0,0){{\tiny $\circ$}}}
\put(275,175){\makebox(0,0){{\tiny $\circ$}}}
\put(300,200){\makebox(0,0){{\tiny $\circ$}}}
\put(325,225){\makebox(0,0){{\tiny $\circ$}}}
\put(350,250){\makebox(0,0){{\tiny $\circ$}}}
\put(375,275){\makebox(0,0){{\tiny $\circ$}}}
\put(400,300){\makebox(0,0){{\tiny $\circ$}}}

\put(125,0){\makebox(0,0){{\tiny $\circ$}}}
\put(150,25){\makebox(0,0){{\tiny $\circ$}}}
\put(175,50){\makebox(0,0){{\tiny $\circ$}}}
\put(200,75){\makebox(0,0){{\tiny $\circ$}}}
\put(225,100){\makebox(0,0){{\tiny $\circ$}}}
\put(250,125){\makebox(0,0){{\tiny $\circ$}}}
\put(275,150){\makebox(0,0){{\tiny $\circ$}}}
\put(300,175){\makebox(0,0){{\tiny $\circ$}}}
\put(325,200){\makebox(0,0){{\tiny $\circ$}}}
\put(350,225){\makebox(0,0){{\tiny $\circ$}}}
\put(375,250){\makebox(0,0){{\tiny $\circ$}}}
\put(400,275){\makebox(0,0){{\tiny $\circ$}}}

\put(175,0){\makebox(0,0){{\tiny $\circ$}}}
\put(200,25){\makebox(0,0){{\tiny $\circ$}}}
\put(225,50){\makebox(0,0){{\tiny $\circ$}}}
\put(250,75){\makebox(0,0){{\tiny $\circ$}}}
\put(275,100){\makebox(0,0){{\tiny $\circ$}}}
\put(300,125){\makebox(0,0){{\tiny $\circ$}}}
\put(325,150){\makebox(0,0){{\tiny $\circ$}}}
\put(350,175){\makebox(0,0){{\tiny $\circ$}}}
\put(375,200){\makebox(0,0){{\tiny $\circ$}}}
\put(400,225){\makebox(0,0){{\tiny $\circ$}}}

\put(200,0){\makebox(0,0){{\tiny $\circ$}}}
\put(225,25){\makebox(0,0){{\tiny $\circ$}}}
\put(250,50){\makebox(0,0){{\tiny $\circ$}}}
\put(275,75){\makebox(0,0){{\tiny $\circ$}}}
\put(300,100){\makebox(0,0){{\tiny $\circ$}}}
\put(325,125){\makebox(0,0){{\tiny $\circ$}}}
\put(350,150){\makebox(0,0){{\tiny $\circ$}}}
\put(375,175){\makebox(0,0){{\tiny $\circ$}}}
\put(400,200){\makebox(0,0){{\tiny $\circ$}}}

\put(250,0){\makebox(0,0){{\tiny $\circ$}}}
\put(275,25){\makebox(0,0){{\tiny $\circ$}}}
\put(300,50){\makebox(0,0){{\tiny $\circ$}}}
\put(325,75){\makebox(0,0){{\tiny $\circ$}}}
\put(350,100){\makebox(0,0){{\tiny $\circ$}}}
\put(375,125){\makebox(0,0){{\tiny $\circ$}}}
\put(400,150){\makebox(0,0){{\tiny $\circ$}}}

\put(275,0){\makebox(0,0){{\tiny $\circ$}}}
\put(300,25){\makebox(0,0){{\tiny $\circ$}}}
\put(325,50){\makebox(0,0){{\tiny $\circ$}}}
\put(350,75){\makebox(0,0){{\tiny $\circ$}}}
\put(375,100){\makebox(0,0){{\tiny $\circ$}}}
\put(400,125){\makebox(0,0){{\tiny $\circ$}}}

\put(325,0){\makebox(0,0){{\tiny $\circ$}}}
\put(350,25){\makebox(0,0){{\tiny $\circ$}}}
\put(375,50){\makebox(0,0){{\tiny $\circ$}}}
\put(400,75){\makebox(0,0){{\tiny $\circ$}}}

\put(350,0){\makebox(0,0){{\tiny $\circ$}}}
\put(375,25){\makebox(0,0){{\tiny $\circ$}}}
\put(400,50){\makebox(0,0){{\tiny $\circ$}}}

\put(375,0){\makebox(0,0){{\tiny $\circ$}}}
\put(400,25){\makebox(0,0){{\tiny $\circ$}}}

\put(0,100){\makebox(0,0){{\tiny $\circ$}}}
\put(25,125){\makebox(0,0){{\tiny $\circ$}}}
\put(50,150){\makebox(0,0){{\tiny $\circ$}}}
\put(75,175){\makebox(0,0){{\tiny $\circ$}}}
\put(100,200){\makebox(0,0){{\tiny $\circ$}}}
\put(125,225){\makebox(0,0){{\tiny $\circ$}}}
\put(150,250){\makebox(0,0){{\tiny $\circ$}}}
\put(175,275){\makebox(0,0){{\tiny $\circ$}}}
\put(200,300){\makebox(0,0){{\tiny $\circ$}}}

\put(0,125){\makebox(0,0){{\tiny $\circ$}}}
\put(25,150){\makebox(0,0){{\tiny $\circ$}}}
\put(50,175){\makebox(0,0){{\tiny $\circ$}}}
\put(75,200){\makebox(0,0){{\tiny $\circ$}}}
\put(100,225){\makebox(0,0){{\tiny $\circ$}}}
\put(125,250){\makebox(0,0){{\tiny $\circ$}}}
\put(150,275){\makebox(0,0){{\tiny $\circ$}}}
\put(175,300){\makebox(0,0){{\tiny $\circ$}}}

\put(0,175){\makebox(0,0){{\tiny $\circ$}}}
\put(25,200){\makebox(0,0){{\tiny $\circ$}}}
\put(50,225){\makebox(0,0){{\tiny $\circ$}}}
\put(75,250){\makebox(0,0){{\tiny $\circ$}}}
\put(100,275){\makebox(0,0){{\tiny $\circ$}}}
\put(125,300){\makebox(0,0){{\tiny $\circ$}}}

\put(0,200){\makebox(0,0){{\tiny $\circ$}}}
\put(25,225){\makebox(0,0){{\tiny $\circ$}}}
\put(50,250){\makebox(0,0){{\tiny $\circ$}}}
\put(75,275){\makebox(0,0){{\tiny $\circ$}}}
\put(100,300){\makebox(0,0){{\tiny $\circ$}}}

\put(0,250){\makebox(0,0){{\tiny $\circ$}}}
\put(25,275){\makebox(0,0){{\tiny $\circ$}}}
\put(50,300){\makebox(0,0){{\tiny $\circ$}}}

\put(0,275){\makebox(0,0){{\tiny $\circ$}}}
\put(25,300){\makebox(0,0){{\tiny $\circ$}}}

\put(44,7){\makebox(0,0){{$\omega$}}}

\put(18,18){\makebox(0,0){$O$}}

{\color{magenta} \Thicklines \path(50,43.3325)(59.0925,50)
\path(75,61.665)(93.1849,75) \path(100,79.9975)(125,98.33)
\path(127.277,100)(150,116.6625) \path(161.3698,125)(175,134.995)
\path(195.4622,150)(200,153.3275) \path(325,245)(331.8321,250)
\path(350,263.3225)(365.9245,275) \path(375,281.655)(400,300)

\Thicklines \path(100,98.33)(102.277,100)
\path(125,116.6625)(136.3698,125) \path(150,134.995)(170.4622,150)
\path(175,153.3275)(200,171.66) \path(204.5547,175)(225,189.9925)
\path(238.6472,200)(250,208.325) \path(272.7396,225)(275,226.6575)

\Thicklines \path(25,0)(50,18.3325)  \path(59.0925,25)(75,36.665)
\path(93.1849,50)(100,54.9975) \path(225,146.66)(229.5547,150)
\path(250,164.9925)(263.6472,175) \path(275,183.325)(297.7396,200)
\path(300,201.6575)(325,220) \path(331.8321,225)(350,238.3225)
\path(365.9245,250)(375,256.655) }

\thinlines \path(0,300)(400,300) \path(0,275)(400,275)

{\color{blue} \Thicklines \path(25,25)(50,43.3325)
\path(59.0925,50)(75,61.665) \path(93.1849,75)(100,79.9975)
\path(225,171.66)(229.5547,175) \path(250,189.9925)(263.6472,200)
\path(275,208.325)(297.7396,225) \path(300,226.6575)(325,245)
\path(331.8321,250)(350,263.3225) \path(365.9245,275)(375,281.655)

\Thicklines \path(25,43.3325)(34.0925,50)
\path(50,61.665)(68.1849,75) \path(75,79.9975)(100,98.33)
\path(102.277,100)(125,116.6625) \path(136.3698,125)(150,134.995)
\path(170.4622,150)(175,153.3275) \path(300,245)(306.8321,250)
\path(325,263.3225)(340.9245,275) \path(350,281.655)(375,300)

\Thicklines \path(125,73.33)(127.277,75)
\path(150,91.6625)(161.3698,100) \path(175,109.995)(195.4622,125)
\path(200,128.3275)(225,146.66) \path(229.5547,150)(250,164.9925)
\path(263.6472,175)(275,183.325)
\path(297.7396,200)(300,201.6575)}

\thinlines \path(0,250)(400,250) \path(0,225)(400,225)




{\color{white} \Thicklines \path(125,98.33)(127.277,100)
\path(150,116.6625)(161.3698,125) \path(175,134.995)(195.4622,150)
\path(200,153.3275)(225,171.66) \path(229.5547,175)(250,189.9925)
\path(263.6472,200)(275,208.325) \path(297.7396,225)(300,226.6575)

\Thicklines \path(0,25)(25,43.3325) \path(34.0925,50)(50,61.665)
\path(68.1849,75)(75,79.9975) \path(200,171.66)(204.5547,175)
\path(225,189.9925)(238.6472,200) \path(250,208.325)(272.7396,225)
\path(275,226.6575)(300,245) \path(306.8321,250)(325,263.3225)
\path(340.9245,275)(350,281.655)

\Thicklines \path(50,18.3325)(59.0925,25)
\path(75,36.665)(93.1849,50) \path(100,54.9975)(125,73.33)
\path(127.277,75)(150,91.6625) \path(161.3698,100)(175,109.995)
\path(195.4622,125)(200,128.3275) \path(325,220)(331.8321,225)
\path(350,238.3225)(365.9245,250) \path(375,256.655)(400,275)}

\put(25,0){\makebox(0,0){{\tiny $\circ$}}}
\put(0,25){\makebox(0,0){{\tiny $\circ$}}}
\end{picture}

\caption{The channels $\CC_0$, $\CC_\leftarrow$ and
$\CC_\downarrow$}\label{Figure7}

\end{figure}

Denote by $r, r_k, r^\prime, r_k^\prime \in \Z / 3\Z$ the remainders
$\hspace{-3pt} \pmod{3}$ of $q-a$, $q_k-a_k$, $q^\prime
-a^\prime$, $q_k^\prime -a_k^\prime$ respectively. The equality
$a^\prime q-aq^\prime =1$ shows that at most one element of the
triple $(r,r_k,r_{k+1}) \in (\Z / 3\Z)^3$ can be equal to zero.
Similarly, at most one element of the triple
$(r^\prime,r_k^\prime,r_{k+1}^\prime)$ can be zero.

To ascertain the contribution of the slope $t=\tan\omega \in
I_{\gamma,k}$ to $\G_{I,Q}(\xi)$, we should look at the tiling
$\SS_\omega$ defined by the three-strip partition of $\R^2$ (shown
in Figure \ref{Figure6}), but also at $\SS_\omega^\leftarrow$, its
left-horizontal translate by $(1,0)$, and at
$\SS_\omega^\downarrow$, its down-vertical translate by $(0,1)$.
Since the slits $(m,n)\in\Z^2$ with $m\equiv n\pmod{3}$ are being
removed, sinks are going to arise in the channels. This phenomenon
will lead to frequent occurrence of trajectories much longer then
in the case of the square lattice. A careful analysis of
the bottom of the channels $\AA_0$, $\CC_k$ and $\BB_k$ is required when the
corresponding slit where trajectory ends in the case of the square
lattice has been removed. Besides, there is a manifest difference
between the three situations where the channel originates at
$O=(0,0)$ (this will be referred to as $\CC_O$
contribution), at $(-1,0)$ ($\CC_\leftarrow$ contribution), or at
$(-1,0)$ ($\CC_\downarrow$ contribution), resulting from the
different congruence conditions $\hspace{-3pt}\pmod{3}$ satisfied by the
centers of removed slits. This is shown in Figure \ref{Figure6}
where the small circles centered at lattice points $(m,n)$ with
$m\nequiv n \pmod{3}$ represent the vertical slits of width
$2\eps=\frac{1}{Q}$. We were not able to spot any symmetry that
would reduce the analysis to only one of these three types of
channels. The contributions to $\Phi^\hex$ of the five types
of situations that we analyze seem to be quite different (see Figure \ref{Figure19}).

To attain a better visualization of the structure of channels we
shall represent the slope $\tan\omega$ horizontally. The possible
situations are shown in Figure \ref{Figure8}, where dotted lines
indicate that the corresponding slit has been removed from
$\SS_\omega$ in the case of $\CC_O$, from $\SS_\omega^\leftarrow$
in the case of $\CC_\leftarrow$ and respectively from
$\SS_\omega^\downarrow$ in the case of $\CC_\downarrow$.

\begin{figure}[ht]
\centering
\unitlength 0.31mm
\begin{picture}(320,335)(-20,-10)

\put(-20,7.5){\makebox(0,0){$(-1,1,0)$}}

\path(60,0)(30,0)(30,15)(80,15) \path(30,11)(110,11)
\path(30,6)(110,6) \path(60,6)(60,-9) \path(80,11)(80,26)
\dottedline{2}(110,2)(110,17)

\path(160,0)(130,0)(130,15)(180,15) \path(130,11)(210,11)
\path(130,6)(210,6) \path(160,6)(160,-9)
\dottedline{2}(180,11)(180,26) \path(210,2)(210,17)

\path(260,0)(230,0)(230,15)(280,15) \path(230,11)(310,11)
\path(230,6)(310,6) \dottedline{2}(260,6)(260,-9)
\path(280,11)(280,26) \path(310,2)(310,17)

\put(-20,47.5){\makebox(0,0){$(1,-1,0)$}}

\path(60,40)(30,40)(30,55)(80,55) \path(30,51)(110,51)
\path(30,46)(110,46) \path(60,46)(60,31) \path(80,51)(80,66)
\dottedline{2}(110,42)(110,57)

\path(160,40)(130,40)(130,55)(180,55) \path(130,51)(210,51)
\path(130,46)(210,46) \dottedline{2}(160,46)(160,31)
\path(180,51)(180,66) \path(210,42)(210,57)

\path(260,40)(230,40)(230,55)(280,55) \path(230,51)(310,51)
\path(230,46)(310,46) \path(260,46)(260,31)
\dottedline{2}(280,51)(280,66) \path(310,42)(310,57)

\put(-20,87.5){\makebox(0,0){$(-1,0,-1)$}}

\path(60,80)(30,80)(30,95)(80,95) \path(30,91)(110,91)
\path(30,86)(110,86) \path(60,86)(60,71)
\dottedline{2}(80,91)(80,106) \path(110,82)(110,97)

\path(160,80)(130,80)(130,95)(180,95) \path(130,91)(210,91)
\path(130,86)(210,86) \path(160,86)(160,71) \path(180,91)(180,106)
\path(210,82)(210,97)

\path(260,80)(230,80)(230,95)(280,95) \path(230,91)(310,91)
\path(230,86)(310,86) \dottedline{2}(260,86)(260,71)
\path(280,91)(280,106) \dottedline{2}(310,82)(310,97)

\put(-20,127.5){\makebox(0,0){$(1,0,1)$}}

\path(60,120)(30,120)(30,135)(80,135) \path(30,131)(110,131)
\path(30,126)(110,126) \path(60,126)(60,111)
\dottedline{2}(80,131)(80,146) \path(110,122)(110,137)

\path(160,120)(130,120)(130,135)(180,135) \path(130,131)(210,131)
\path(130,126)(210,126) \dottedline{2}(160,126)(160,111)
\path(180,131)(180,146) \dottedline{2}(210,122)(210,137)

\path(260,120)(230,120)(230,135)(280,135) \path(230,131)(310,131)
\path(230,126)(310,126) \path(260,126)(260,111)
\path(280,131)(280,146) \path(310,122)(310,137)

\put(-20,167.5){\makebox(0,0){$(0,-1,-1)$}}

\path(60,160)(30,160)(30,175)(80,175) \path(30,171)(110,171)
\path(30,166)(110,166) \dottedline{2}(60,166)(60,151)
\path(80,171)(80,186) \path(110,162)(110,177)

\path(160,160)(130,160)(130,175)(180,175) \path(130,171)(210,171)
\path(130,166)(210,166) \path(160,166)(160,151)
\path(180,171)(180,186) \path(210,162)(210,177)

\path(260,160)(230,160)(230,175)(280,175) \path(230,171)(310,171)
\path(230,166)(310,166) \path(260,166)(260,151)
\dottedline{2}(280,171)(280,186) \dottedline{2}(310,162)(310,177)

\put(-20,207.5){\makebox(0,0){$(0,1,1)$}}

\path(60,200)(30,200)(30,215)(80,215) \path(30,211)(110,211)
\path(30,206)(110,206) \dottedline{2}(60,206)(60,191)
\path(80,211)(80,226) \path(110,202)(110,217)

\path(160,200)(130,200)(130,215)(180,215) \path(130,211)(210,211)
\path(130,206)(210,206) \path(160,206)(160,191)
\dottedline{2}(180,211)(180,226) \dottedline{2}(210,202)(210,217)

\path(260,200)(230,200)(230,215)(280,215) \path(230,211)(310,211)
\path(230,206)(310,206) \path(260,206)(260,191)
\path(280,211)(280,226) \path(310,202)(310,217)

\put(-20,247.5){\makebox(0,0){$(-1,-1,1)$}}

\path(60,240)(30,240)(30,255)(80,255) \path(30,251)(110,251)
\path(30,246)(110,246) \path(60,246)(60,231) \path(80,251)(80,266)
\path(110,242)(110,257)

\path(160,240)(130,240)(130,255)(180,255) \path(130,251)(210,251)
\path(130,246)(210,246) \path(160,246)(160,231)
\path(180,251)(180,266) \dottedline{2}(210,242)(210,257)

\path(260,240)(230,240)(230,255)(280,255) \path(230,251)(310,251)
\path(230,246)(310,246) \dottedline{2}(260,246)(260,231)
\dottedline{2}(280,251)(280,266) \path(310,242)(310,257)

\put(-20,287.5){\makebox(0,0){$(1,1,-1)$}}

\path(60,280)(30,280)(30,295)(80,295) \path(30,291)(110,291)
\path(30,286)(110,286) \path(60,286)(60,271) \path(80,291)(80,306)
\path(110,282)(110,297)

\path(160,280)(130,280)(130,295)(180,295) \path(130,291)(210,291)
\path(130,286)(210,286) \dottedline{2}(160,286)(160,271)
\dottedline{2}(180,291)(180,306) \path(210,282)(210,297)

\path(260,280)(230,280)(230,295)(280,295) \path(230,291)(310,291)
\path(230,286)(310,286) \path(260,286)(260,271)
\path(280,291)(280,306) \dottedline{2}(310,282)(310,297)

\put(-20,320){\makebox(0,0){$(r,r_k,r_{k+1})$}}
\put(70,320){\makebox(0,0){$\CC_O$}}
\put(170,320){\makebox(0,0){$\CC_\leftarrow$}}
\put(270,320){\makebox(0,0){$\CC_\downarrow$}}

\Thicklines \path(-50,-12)(330,-12) \path(-50,28)(330,28)
\path(-50,68)(330,68) \path(-50,108)(330,108)
\path(-50,148)(330,148) \path(-50,188)(330,188)
\path(-50,228)(330,228) \path(-50,268)(330,268)
\path(-50,308)(330,308) \path(-50,331)(330,331)
\path(-50,331)(-50,-12) \path(330,331)(330,-12)
\path(15,331)(15,-12) \path(120,331)(120,-12)
\path(220,331)(220,-12)
\end{picture}
\caption{The removed slits for $\CC_O$, $\CC_\leftarrow$ and
$\CC_\downarrow$} \label{Figure8}
\end{figure}

To clarify the terminology, by ``slit $q_k$" etc. we will mean
\underline{the} slit which is centered at some lattice point
$(q_k,m)$ and which intersects the channel that is analyzed (there
is at most one such point for given $q_k$).

\section{The contribution of channels whose slits are not
removed}\label{Sect5}

This resembles the situation of the square lattice and will be
discussed in this section. The more intricate situation of the channels where
bottom slits are removed will be analyzed
in Sections 6-9. When the ``first" slit (i.e. the one corresponding to $q$ for
$\AA_0$, $q_{k+1}$ for $\CC_k$, respectively $q_k$ for $\BB_k$) is
not removed, the table from Figure \ref{Figure8} shows that the corresponding
weight is described in the table from Figure \ref{Figure9}.

\begin{figure}[ht]
\centering
\unitlength 0.33mm
\begin{picture}(320,185)(-40,0)

\put(-30,10){\makebox(0,0){$(-1,1,0)$}}
\put(50,10){\makebox(0,0){$W_{\AA_0}+W_{\BB_k}$}}
\put(150,10){\makebox(0,0){$W_{\AA_0}+W_{\CC_k}$}}
\put(250,10){\makebox(0,0){$W_{\CC_k}+W_{\BB_k}$}}

\put(-30,30){\makebox(0,0){$(1,-1,0)$}}
\put(50,30){\makebox(0,0){$W_{\AA_0}+W_{\BB_k}$}}
\put(150,30){\makebox(0,0){$W_{\CC_k}+W_{\BB_k}$}}
\put(250,30){\makebox(0,0){$W_{\AA_0}+W_{\CC_k}$}}

\put(-30,50){\makebox(0,0){$(-1,0,-1)$}}
\put(50,50){\makebox(0,0){$W_{\AA_0}+W_{\CC_k}$}}
\put(150,50){\makebox(0,0){$W_{\AA_0}+W_{\CC_k}+W_{\BB_k}$}}
\put(250,50){\makebox(0,0){$W_{\BB_k}$}}

\put(-30,70){\makebox(0,0){$(1,0,1)$}}
\put(50,70){\makebox(0,0){$W_{\AA_0}+W_{\CC_k}$}}
\put(150,70){\makebox(0,0){$W_{\BB_k}$}}
\put(250,70){\makebox(0,0){$W_{\AA_0}+W_{\CC_k}+W_{\BB_k}$}}

\put(-30,90){\makebox(0,0){$(0,-1,-1)$}}
\put(50,90){\makebox(0,0){$W_{\CC_k}+W_{\BB_k}$}}
\put(150,90){\makebox(0,0){$W_{\AA_0}+W_{\CC_k}+W_{\BB_k}$}}
\put(250,90){\makebox(0,0){$W_{\AA_0}$}}

\put(-30,110){\makebox(0,0){$(0,1,1)$}}
\put(50,110){\makebox(0,0){$W_{\CC_k}+W_{\BB_k}$}}
\put(150,110){\makebox(0,0){$W_{\AA_0}$}}
\put(250,110){\makebox(0,0){$W_{\AA_0}+W_{\CC_k}+W_{\BB_k}$}}

\put(-30,130){\makebox(0,0){$(-1,-1,1)$}}
\put(50,130){\makebox(0,0){$W_{\AA_0}+W_{\CC_k}+W_{\BB_k}$}}
\put(150,130){\makebox(0,0){$W_{\AA_0}+W_{\BB_k}$}}
\put(250,130){\makebox(0,0){$W_{\CC_k}$}}

\put(-30,150){\makebox(0,0){$(1,1,-1)$}}
\put(50,150){\makebox(0,0){$W_{\AA_0}+W_{\CC_k}+W_{\BB_k}$}}
\put(150,150){\makebox(0,0){$W_{\CC_k}$}}
\put(250,150){\makebox(0,0){$W_{\AA_0}+W_{\BB_k}$}}

\put(-30,170){\makebox(0,0){$(r,r_k,r_{k+1})$}}
\put(50,170){\makebox(0,0){$\CC_O$}}
\put(150,170){\makebox(0,0){$\CC_\leftarrow$}}
\put(250,170){\makebox(0,0){$\CC_\downarrow$}}

\Thicklines \path(-60,-0)(300,0) \path(-60,20)(300,20) \path(-60,40)(300,40)
\path(-60,60)(300,60) \path(-60,80)(300,80) \path(-60,100)(300,100)
\path(-60,120)(300,120) \path(-60,140)(300,140)
\path(-60,159)(300,159) \path(-60,161)(300,161) \path(-60,180)(300,180)
\path(-60,0)(-60,180)(300,180)(300,0) \path(0,0)(0,180) \path(100,0)(100,180)
\path(200,0)(200,180) \path(2,0)(2,180)
\end{picture}
\caption{The contribution of channels whose slits are not removed} \label{Figure9}
\end{figure}

The weights $W_{\AA_0}$, $W_{\BB_k}$ and $W_{\CC_k}$ are given by
\eqref{4.1} as in the case of the square lattice. The cumulative
contribution is
\begin{equation*}
\G^{(0)}_{I,Q} (\xi)=\sum_{(\alpha,\beta)\in (\Z / 3\Z)^2 \setminus
\{ (0,0)\}} \G^{(0)}_{I,Q,\alpha,\beta} (\xi),
\end{equation*}
with $\G^{(0)}_{I,Q,\alpha,\beta}(\xi)$ given by
\begin{equation}\label{5.1}
\sum_{k=0}^\infty
\hspace{-6pt} \sum\limits_{\substack{\gamma\in \FF_I (Q) \\
q-a\equiv \alpha \hspace{-6pt} \pmod{3} \\ q_k -a_k \equiv \beta
\hspace{-6pt} \pmod{3}}} \hspace{-10pt} \int_{t_k}^{t_{k-1}} \frac{2\big(
(q-\xi Q)_+ w_{\AA_0}(t) + (q_{k+1}-\xi Q)_+ w_{\CC_k}(t)
+ (q_k -\xi Q)_+ w_{\BB_k}(t)\big)\, dt}{t^2+t+1} .
\end{equation}

\newtheorem*{R1}{Remark 1.}\label{R5.1}
\begin{R1}
Putting $x=q-a$, $y=q_k$, and employing
$q_k-a_k=q_k-\frac{1+aq_k}{q}=\frac{xy-1}{q}$, we see that the
summation conditions in \eqref{5.1} are equivalent to $(x,q)=1$,
$x\equiv \alpha \pmod{3}$, and $xy\equiv \beta q+1 \pmod{3q}$.
Note that $(\beta q+1,q)=1$ and sum as follows:
\begin{itemize}
\item When $\alpha \neq 0$ Lemma \ref{L3.8} may be applied {\em
(}because $(x,3q)=1${\em )}, followed by Lemma \ref{L3.3}. \item
When $\alpha = 0$ we have $\beta\neq 0$ and $x=3\tilde{x}$,
$\tilde{x}\in \frac{q}{3}(1-I)$. Furthermore, $\beta q+1 \equiv 0
\pmod{3}$, so $q\equiv -\beta \pmod{3}$
and one may first sum, as in Lemma \ref{L3.4}, over $(\tilde{x},y)$
and the conditions
\begin{equation}\label{5.2}
\begin{cases}
q-a=x=3\tilde{x},\quad \tilde{x} \in \frac{q}{3} (1-I) ,\quad (\tilde{x},q)=1, \\
q\equiv -\beta \hspace{-3pt} \pmod{3}, \quad \tilde{x} y \equiv
\frac{\beta q+1}{3} \hspace{-3pt} \pmod{q},\quad y\in \II_{q,k},
\end{cases}
\end{equation}
and then sum over $q\in [1,Q]$ with $q\equiv -\beta
\pmod{3}$ employing Lemma \ref{L3.3}.
\end{itemize}
In all situations where sums over $\gamma\in\FF_I(Q)$ with
$q-a\equiv \alpha \pmod{3}$ and $q_k-a_k \equiv \beta \pmod{3}$,
$(\alpha,\beta)\neq (0,0)$, have to be
evaluated, the resulting constant will be $\frac{1}{8\zeta(2)}$.
\end{R1}

The following elementary estimate will be used throughout.

\begin{lemma}\label{L5.2}
For $b,c,h,\lambda,\mu\in\R$ such that $0\leqslant b,c, c+h \leqslant 1$ and $\lambda c+\mu=0$, there exist
$\theta^\prime,\theta=\theta_{c,h}\in [-3,3]$ such that
\begin{equation*}
\int_c^{c+h} \frac{\lambda t+\mu}{t^2+t+1}\ du =\frac{h^2
\lambda}{2(b^2+b+1)}+ h^3 \theta (\vert \lambda\vert+\vert
\mu\vert) + h^2 \theta^\prime \vert b-c\vert \vert \lambda\vert.
\end{equation*}
\end{lemma}

\begin{proof}
This follows immediately applying Taylor's formula twice:
\begin{equation*}
\begin{split}
\int_c^{c+h} \frac{dt}{t^2+t+1} & =\frac{h}{c^2+c+1}
-\frac{h^2(2c+1)}{2(c^2+c+1)^2} + \xi^\prime h^3,\quad \vert \xi^\prime\vert \leqslant 1 , \\
\int_c^{c+h} \frac{t\ dt}{t^2+t+1} & = \left(\frac{h}{c^2+c+1}
-\frac{h^2(2c+1)}{2(c^2+c+1)^2}\right)c +\frac{h^2}{2(c^2+c+1)} +
\xi^{\prime \prime} h^3,\quad \vert \xi^{\prime\prime}\vert
\leqslant 2,
\end{split}
\end{equation*}
and employing
\begin{equation*}
\left| \frac{1}{c^2+c+1} -\frac{1}{b^2+b+1}\right| \leqslant
3\vert b-c\vert .
\end{equation*}
\end{proof}

Together with $0< q_1 -Q \leqslant q \wedge q^\prime$ Lemma \ref{L5.2} yields
$(t_{-1}=\gamma_1$)
\begin{equation}\label{5.3}
\int_{t_0}^{t_{-1}} \frac{w_{\BB_0}(t)\ dt}{t^2+t+1}
=\int_{t_0}^{\gamma_1} \frac{q^\prime t-a^\prime+2\eps}{t^2+t+1}\ dt
=\frac{(q_1-Q)^2}{2Q^2 q^\prime q_1^2 (\gamma^2+\gamma+1)}
+O\left( \frac{1}{Q^5} \right).
\end{equation}
\begin{equation}\label{5.4}
\begin{split}
\int_{t_0}^{t_{-1}} \frac{w_{\CC_0}(t)\ dt}{t^2+t+1} & =
\int_{\gamma}^{\gamma_1} \frac{qt-a}{t^2+t+1}\ dt
-\int_{\gamma}^{t_0} \frac{qt-a}{t^2+t+1}\ dt -
\int_{t_0}^{\gamma_1} \frac{q^\prime t-a^\prime+2\eps}{t^2+t+1}\ dt
\\ & =q\,\frac{(\gamma_1-\gamma)^2 -(t_0-\gamma)^2}{2(\gamma^2+\gamma+1)}
- q^\prime \frac{(\gamma_1-t_0)^2 }{2(\gamma^2+\gamma+1)} \\ & \qquad
+O\Big( q(\gamma_1-\gamma)^3 +q^\prime (\gamma_1-t_0)^3
+q^\prime (\gamma_1-t_0)^2 (t_0-\gamma)\Big) \\
& =\frac{q_1-Q}{2Q^2 q^\prime q_1 (\gamma^2 +\gamma +1)} \left(
\frac{2Q-q_1}{q_1}+\frac{Q-q}{q^\prime}\right) +O\left(
\frac{1}{Q^3 q^2}\right) .
\end{split}
\end{equation}
\begin{equation}\label{5.5}
\begin{split}
\int_{t_0}^{t_{-1}} \frac{w_{\AA_0}(t)\ dt}{t^2+t+1} & =
\int_{t_0}^{u_0} \frac{a+2\eps -qt}{t^2+t+1}\ dt
- \int_{\gamma_1}^{u_0} \frac{a+2\eps-qt}{t^2+t+1}\ dt  \\
& =q\,\frac{(u_0-t_0)^2 -(u_0-\gamma_1)^2}{2(\gamma^2+\gamma+1)}+
O\Big( q(u_0-t_0)^3 + q(u_0-t_0)^2 (u_0-\gamma) \Big) \\
& =\frac{(q_1-Q)^2 (q_1+q^\prime)}{2Q^2 q^{\prime 2} q_1^2
(\gamma^2 +\gamma+1)} +O\left(\frac{1}{Q^3 qq^\prime}\right).
\end{split}
\end{equation}

For every $k\geqslant 1$ we find
\begin{equation}\label{5.6}
\int_{t_k}^{t_{k-1}} \frac{w_{\BB_k}(t)\ dt}{t^2+t+1}  =
\int_{t_k}^{t_{k-1}} \frac{q_k t-a_k+2\eps}{t^2+t+1}\ dt =
\frac{(Q-q)^2}{2Q^2 q_{k-1}^2 q_k(\gamma^2+\gamma+1)} +O\left(
\frac{1}{qq_{k-1}^2 q_k^2}\right) .
\end{equation}
\begin{equation}\label{5.7}
\begin{split}
\int_{t_k}^{t_{k-1}} \frac{w_{\CC_k}(t)\ dt}{t^2+t+1}
& =\int_{t_k}^{t_{k-1}} \frac{a_{k-1}-2\eps-q_{k-1}t}{t^2+t+1}\ dt \\ & =
\frac{(Q-q)^2}{2Q^2 q_{k-1} q_k^2 (\gamma^2+\gamma +1)}+O\left(
\frac{1}{qq_{k-1}^2 q_k^2}\right) .
\end{split}
\end{equation}
\begin{equation}\label{5.8}
\begin{split}
\int_{t_k}^{t_{k-1}} \frac{2\eps\ dt}{t^2+t+1}  &
=\frac{t_{k-1}-t_k}{Q(\gamma^2+\gamma+1)} +O\Big( \eps
(t_{k-1}-t_k)(t_k-\gamma)+\eps (t_{k-1}-t_k)^2\Big)
\\ & =\frac{Q-q}{Q^2 q_{k-1} q_k
(\gamma^2+\gamma+1)} +O\left( \frac{1}{Qqq_{k-1} q_k^2}\right) .
\end{split}
\end{equation}
\begin{equation}\label{5.9}
\begin{split}
\int_{t_k}^{t_{k-1}} \frac{w_{\AA_0}(t)\ dt}{t^2+t+1} & =
\int_{t_k}^{t_{k-1}}
\frac{2\eps -w_{\BB_k}(t)-w_{\CC_k}(t)}{t^2+t+1} \ dt \\
& = \frac{Q-q}{2Q^2 q_{k-1}q_k (\gamma^2+\gamma+1)} \left(
\frac{q_{k+1}-Q}{q_k}+\frac{q_k -Q}{q_{k-1}} \right) +O\left(
\frac{1}{qq^\prime q_{k-1} q_k^2}\right) .
\end{split}
\end{equation}
The total contribution of the error terms from \eqref{5.3}-\eqref{5.7} and \eqref{5.9}
to $\G^{(0)}_{I,Q} (\xi)$ is
\begin{equation*}
\begin{split}
\ll \hspace{-2pt} \sum_{\gamma\in\FF_I(Q)} \hspace{-1pt} \frac{1}{(q q^\prime )^2} & + \sum_{k=1}^\infty
\sum_{\gamma\in\FF_I(Q)} \left( \frac{1}{q^\prime q_{k-1} q_k^2}
+\frac{1}{qq_{k-1}^2 q_k} \right) \leqslant \hspace{-2pt}
\sum_{\gamma\in\FF_I(Q)} \hspace{-1pt} \frac{1}{(qq^\prime)^2} + \hspace{-5pt}
\sum_{\gamma\in\FF_I(Q)} \hspace{-1pt} \frac{1}{qq^\prime} \sum_{k=1}^\infty
\frac{1}{q_{k-1} q_k} \\
& = 2\sum_{\gamma\in\FF_I(Q)} \frac{1}{(qq^\prime)^2} \leqslant
\frac{2}{Q} \sum_{\gamma\in\FF_I(Q)} \frac{1}{qq^\prime} \leqslant
\frac{2}{Q} \ \bigg( \vert I\vert +\frac{2}{Q}\bigg)
\leqslant \frac{6\vert I\vert}{Q}.
\end{split}
\end{equation*}
When summing over a family of intervals $I$ that partition $[0,1]$ this
adds up to $O(Q^{-1})$, thus all error terms above can be
discarded below. We emphasize that, since the contribution of
each interval $I_{\gamma,k}$ is $\leqslant \vert I_{\gamma,k}\vert$, we can
remove one element of $\FF_I(Q)$ for every $I$.

Applying Remark 1 to the inner sum in \eqref{5.1} with
$x=q-a\in q(1-I)$, $y=q_k \in \II_{q,k}$, and employing formulas
\eqref{5.3}-\eqref{5.9} we find
\begin{equation}\label{5.10}
\G_{I,Q,\alpha,\beta}^{(0)} (\xi) =\sum_{q=1}^Q
\sum_{k=1}^\infty \sum\limits_{\substack{x\in q(1-I),\,
y\in\II_{q,k} \\ x\equiv \alpha \hspace{-6pt} \pmod{3}
\\ xy\equiv \beta q+1 \hspace{-6pt} \pmod{3q}}} f_q (x,y)
+\sum_{q=1}^Q\sum\limits_{\substack{x\in q(1-I),\, y\in (Q-q,Q] \\
x\equiv\alpha \hspace{-6pt}\pmod{3} \\
xy\equiv \beta q+1 \hspace{-6pt} \pmod{3q}}} g_q (x,y),
\end{equation}
with $C^1$ functions $f_q$ and $g_q$ on $\R\times (\R_+ \setminus \{
q,\xi Q-q,\xi Q\})$ given by
\begin{equation*}
\begin{split}
f_q (x,y) = & \frac{q^2}{q^2+(q-x)^2} \, F_q (y), \qquad g_q
(x,y)=\frac{q^2}{q^2+(q-a)^2}\, G_q (y), \\
F_q (y) = & \frac{Q-q}{Q^2 (y-q)y} \left( \frac{y+q-Q}{y} +
\frac{y-Q}{y-q}\right)  (q-\xi Q)_+ \\ &
+\frac{(Q-q)^2}{Q^2 (y-q)y}  \left( \frac{(y-\xi Q)_+}{y-q} +
\frac{(y+q-\xi Q)_+}{y} \right) , \\
G_q (y) = & \frac{y+q-Q}{Q^2 y(y+q)} \left( \frac{2Q-q-y}{y+q} +
\frac{Q-q}{y}\right) (y+q-\xi Q)_+ \\ & + \frac{(y+q-Q)^2}{Q^2
y(y+q)^2} \left( \frac{2y+q}{y} \ (q-\xi Q)_+ +(y-\xi Q)_+
\right) .
\end{split}
\end{equation*}

The innermost sums in \eqref{5.10} will be estimated employing Lemmas
\ref{L3.4} or \ref{L3.8}. We first need to bound $\| f_q\|_\infty$,
$\| \nabla f_q\|_\infty$ on $q(1-I)\times \II_{q,k}$, $k\geqslant
1$, and respectively $\| g_q\|_\infty$, $\| \nabla g_q\|_\infty$
on $q(1-I)\times (Q-q,Q]$. From $y\geqslant Q$ in the first case
and $y+q-Q\leqslant Q\leqslant Q\leqslant y+q\leqslant 2Q$ in the
second one we find for all $y\in\II_{q,k}$, $k\geqslant 1$:
\begin{equation}\label{5.11}
0\leqslant f_q (x,y)\leqslant F_q (y) \leqslant \frac{Q-q}{Q^2
(y-q)y} \Bigg( q+(Q-q)\bigg( 2+\frac{q}{y-q}+\frac{q}{y}\bigg)
\Bigg) \leqslant \frac{4(Q-q)}{Q(y-q)y},
\end{equation}
\begin{equation}\label{5.12}
\vert F_q^\prime (y)\vert
\leqslant \frac{16(\xi +1)(Q-q)}{Q(y-q)^2 y},
\end{equation}
and for all $y\in (Q-q,Q]$:
\begin{equation}\label{5.13}
0\leqslant g_q (x,y) \leqslant G_q (y)\leqslant \frac{1}{Q^2
(y+q)}\cdot 2(y+q)+\frac{1}{Q^2 (y+q)^2}\cdot \big(
(2y+q)q+y^2\big) = \frac{3}{Q^2},
\end{equation}
\begin{equation}\label{5.14}
\vert G_q^\prime (y)\vert \leqslant \frac{20}{Qy^2}.
\end{equation}
From \eqref{5.11} we infer
\begin{equation*}
\sum_{k=1}^\infty \| F_q \|_{\II_{q,k}} \leqslant \frac{2(Q-q)}{Q}
\sum_{k=1}^\infty \frac{1}{\big( Q+(k-2)q\big)\big( Q+(k-1)q\big)}
=\frac{4}{Qq},
\end{equation*}
which leads in turn to
\begin{equation*}
\sum_{q=1}^Q q^{\frac{1}{2}+\delta} \sum_{k=1}^\infty \|
f_q\|_{q(1-I)\times \II_{q,k}} \leqslant \frac{2}{Q} \sum_{q=1}^Q
q^{-\frac{1}{2}+\delta} \leqslant 8Q^{\frac{1}{2}+\delta} .
\end{equation*}
From \eqref{5.12} we infer
\begin{equation*}
\begin{split}
\sum_{k=1}^\infty \| F_q^\prime \|_{\II_{q,k}} & \leqslant
\frac{16 (\xi +1)(Q-q)}{Q} \sum_{k=1}^\infty \frac{1}{\big(
Q+(k-2)q+1\big)^2 \big( Q+(k-1)q+1\big)} \\
& \leqslant \frac{16 (\xi +1)}{Q} \sum_{k=1}^\infty \frac{1}{\big(
Q+(k-2)q+1\big)\big( Q+(k-1)q+1\big)} \leqslant
\frac{16(\xi +1)}{Qq (Q-q+1)} ,
\end{split}
\end{equation*}
which leads to
\begin{equation*}
\sum_{k=1}^\infty \| \nabla f_q \|_{q(1-I)\times \II_{q,k}}
\leqslant \sum_{k=1}^\infty \Bigg( \frac{2}{q} \ \| F_q
\|_{\II_{q,k}} +\| F^\prime_q \|_{\II_{q,k}} \Bigg) \leqslant
\frac{24(\xi +1)}{q^2 (Q-q+1)} ,
\end{equation*}
and finally gives
\begin{equation}\label{5.15}
\begin{split}
\sum_{q=1}^Q q^{\frac{3}{2}+\delta} \sum_{k=1}^\infty \|\nabla f_q
\|_{q(1-I)\times \II_{q,k}} & \ll_\xi \sum_{1\leqslant
q\leqslant \frac{Q}{2}} \frac{q^{-1/2+\delta}}{Q-q+1}
+\sum_{\frac{Q}{2}
\leqslant q \leqslant Q} \frac{q^{-1/2+\delta}}{Q-q+1} \\
& \leqslant \frac{2}{Q} \sum_{q=1}^Q q^{-\frac{1}{2}+\delta}
+\left( \frac{Q}{2}\right)^{-\frac{1}{2}+\delta} \sum_{n=1}^Q
\frac{1}{n} \ll Q^{-\frac{1}{2}+\delta} \log Q, \end{split}
\end{equation}
\begin{equation}\label{5.16}
\sum_{q=1}^Q q^2 \sum_{k=1}^\infty \| \nabla f_q \|_{q(1-I)\times
\II_{q,k}} \ll_\xi \sum_{q=1}^Q \frac{1}{Q-q+1} \ll \log Q.
\end{equation}
Taking $T=[Q^{c^\prime}]$  we infer upon  \eqref{5.15} and \eqref{5.16}, with $E_{c,c^\prime,\delta} (Q)$ as in Theorem \ref{T21},
\begin{equation}\label{5.17}
\begin{split}
\sum_{q=1}^Q \sum_{k=1}^\infty & \Bigg(
Tq^{\frac{1}{2}+\delta} \Big( T\| f_q\|_{q(1-I)\times \II_{q,k}}
+q\| \nabla f_q \|_{q(1-I)\times \II_{q,k}} \Big) +\frac{q\vert
I\vert \vert \II_{q,k}\vert}{T} \ \| \nabla f_q \|_{q(1-I)\times
\II_{q,k}} \Bigg) \\
& \ll_{\xi,\delta} T^2 Q^{-\frac{1}{2}+\delta} + T
Q^{-\frac{1}{2}+\delta} \log Q +\vert I\vert T^{-1}\log Q \ll
E_{c,c^\prime,\delta}(Q),
\end{split}
\end{equation}
uniformly for $\xi$ in compact subsets of $[0,\infty)$.
From \eqref{5.13} and \eqref{5.14} we also get
\begin{equation}\label{5.18}
\sum_{q=1}^Q \Bigg( Tq^{\frac{1}{2}+\delta} \Big(
T\| g_q\|_{q(1-I)\times \II_{q,0}} +q\| \nabla g_q
\|_{q(1-I)\times \II_{q,0}} \Big) +\frac{q^2\vert I\vert }{T} \ \|
\nabla g_q \|_{q(1-I)\times \II_{q,0}} \Bigg) \ll_\delta
E_{c,c^\prime,\delta}(Q) .
\end{equation}

When $\alpha\neq 0$, Lemma \ref{L3.8} applies to the innermost sums in
\eqref{5.10}. Upon \eqref{5.17} and \eqref{5.18} we find
\begin{equation}\label{5.19}
\begin{split}
\G_{I,Q,\alpha,\beta}^{(0)} (\xi)  & \approxeq
\sum\limits_{\substack{1\leqslant q\leqslant Q \\ 3\nmid q}}
\frac{\varphi (q)}{9q^2} \cdot qc_I \left( \sum_{k=1}^\infty
\int_{\II_{q,k}} F_q (y)\, dy +\int_{Q-q}^Q G_q (y)\, dy \right) \\ &
\approxeq \sum\limits_{\substack{1\leqslant q\leqslant Q \\ 3\mid q}}
\frac{\varphi(q)}{6q^2}\cdot qc_I \left( \sum_{k=1}^\infty
\int_{\II_{q,k}} F_q (y)\, dy +\int_{Q-q}^Q G_q (y)\, dy \right) .
\end{split}
\end{equation}
Applying Lemma \ref{L3.3} with $\ell=3$ to the sum over $q$ in
\eqref{5.19} and making the substitution $(q,y)=(Qu,Qw)$ we
gather
\begin{equation*}
\begin{split}
\G_{I,Q,\alpha,\beta}^{(0)} (\xi) & \approxeq \left( \frac{3}{4}
\cdot\frac{1}{9} +\frac{1}{4} \cdot
\frac{1}{6}\right)\frac{c_I}{\zeta(2)} \cdot \int_0^1 \hspace{-3pt} du\left(
\int_{1-u}^1 \hspace{-2pt} dw\, F^{(0.1)} (\xi;u,w)+\int_1^\infty \hspace{-3pt} dw\,
F^{(0.2)}(\xi;u,w)\right) ,
\end{split}
\end{equation*}
with $F^{(0.1)}$ and $F^{(0.2)}$ as in \eqref{5.22}. The total
contribution of cases $\alpha\neq 0$ to $\G_{I,Q}^{(0)}(\xi)$
is obtained by multiplying the quantity above by $6$.

When $\alpha=0$ and $\beta=\pm 1$ we sum as above, with summation
conditions in the inner sum given by \eqref{5.2} and get, employing
Lemmas \ref{L3.8} and \ref{L3.2},
\begin{equation}\label{5.20}
\begin{split}
\G_{I,Q,0,1}^{(0} (\xi) & +\G_{I,Q,0,-1}^{(0)} (\xi)
\approxeq \sum\limits_{\substack{1\leqslant q\leqslant Q \\
(q,3)=1}} \frac{\varphi (q)}{q^2} \cdot \frac{qc_I}{3} \Bigg(
\int_Q^\infty F_q (y)\, dy +\int_{Q-q}^Q G_q (y)\, dy \Bigg) \\
& \approxeq\frac{c_I}{4\zeta (2)} \int_0^1 du \left( \int_{1-u}^1
dw \, F^{(0.1)}(\xi;u,w)+\int_1^\infty dw\,
F^{(0.2)}(\xi;u,w)\right)  .
\end{split}
\end{equation}
Applying Lemma \ref{L3.2} to \eqref{5.20} we finally find
\begin{equation}\label{5.21}
\G_{I,Q}^{(0)}(\xi)\approxeq \frac{c_I}{\zeta(2)} \int_0^1 du
\left( \int_{1-u}^1 dw\ F^{(0.1)} (\xi;u,w) + \int_1^\infty
dw\ F^{(0.2)} (\xi;u,w) \right),
\end{equation}
where
\begin{equation}\label{5.22}
\begin{split}
F^{(0.1)} (\xi;u,w) = & \frac{(w+u-1)^2 (2w+u)}{w^2 (w+u)^2}
 (u-\xi)_+ + \frac{(w+u-1)^2}{w(w+u)^2}
(w-\xi)_+ \\ & \quad + \frac{w+u-1}{w(w+u)} \left( \frac{2-u-w}{w+u}
+ \frac{1-u}{w} \right) (w+u-\xi)_+ , \\ F^{(0.2)}
(\xi;u,w)= & \frac{1-u}{(w-u)w} \left( \frac{w+u-1}{w} +
\frac{w-1}{w-u}\right) (u-\xi)_+ \\ & \quad
+\frac{(1-u)^2}{(w-u)^2 w}  ( w-\xi)_+ +
\frac{(1-u)^2}{(w-u)w^2}  ( w+u-\xi)_+  .
\end{split}
\end{equation}

\newtheorem*{R2}{Remark 2.}\label{R5.3}
\begin{R2}
For each angle $\omega$ the weight $W_{\gamma,k}(t)=W_{\gamma,k}^\hex (t)$ is clearly
no larger than the weight $W^\square_{\gamma,k}(t)$ from the situation where no slit is being removed
(and which corresponds up to some scaling to the case of the the unit square).\footnote{Note that because of the
scaling of $\xi$ this does not imply $\Phi^\hex (\xi) \leqslant \Phi^\square (\xi)$ (see also Figure \ref{Figure2}).}
Since the corresponding limiting distribution $\Phi^\square$ satisfies \eqref{1.2}
\cite{BGW,BZ2,Dah}, it follows that $\Phi^\hex$ satisfies the second inequality in \eqref{1.2}.
The first inequality in \eqref{1.2} follows for instance from
\begin{equation*}
G^{(0)}(\xi)  := \int_0^1 du \int_1^\infty dw\ \frac{(1-u)^2}{(w-u)^2 w} (w-\xi)_+
\geqslant \int_0^1 du\ (1-u)^2 \int_\xi^\infty \frac{v\ dv}{(v+\xi)^3} =\frac{3}{16 \xi} .
\end{equation*}
\end{R2}

\section{Channels with removed slits. The case $r=r_k=\pm 1$}\label{Sect6}
Assume first that $r=r_k=1$. Then $r_{k+1}=r+r_k=-1\pmod{3}$. One
has to analyze the $\CC_\leftarrow$ and the $\CC_\downarrow$
contributions. We shall sum as in Remark 1 above with $x=q-a\in
q(1-I)$, $\alpha=1$, $y=q_k\in\II_{q,k}$, $\beta=1$, considering
\begin{equation*}
\sideset{}{^*}\sum = \sideset{}{^*}\sum\limits_{\substack{\gamma\in\FF_I (Q) \\
q-a\equiv 1 \hspace{-6pt}\pmod{3} \\ q_k-a_k \equiv 1
\hspace{-6pt}\pmod{3}}} .
\end{equation*}

The table in Figure \ref{Figure8} shows that the weights
$W_{\gamma,k} (t)$ for $(\CC_\leftarrow ,r=r_k=1)$ and for
$(\CC_\downarrow,r=r_k=-1)$ do coincide, and so they do for
$(\CC_\downarrow,r=r_k=1)$ and $(\CC_\leftarrow,r=r_k=-1)$. This
eventually shows that the corresponding contribution for
$r=r_k=-1$ has the same main term and error terms as the one for
$r=r_k=1$ (we just need to replace $\beta$ by $-\beta$ and
$\alpha$ by $-\alpha$, which will produce the same main term and
error size). As a result we shall only take $r=r_k=1$ and double the
total contribution in the sequel.

\subsection{The $\CC_\leftarrow$ contribution}\label{Subsec6.1}
The slits $q$ and $q_k$ are removed, while $2q$, $q_{k+1}$,
$q_{k+2}$, $2q_k$ and $q_k+q_{k+1}$ are not because
$2r=r_{k+1}=2r_k=2$, $r_{k+2}=r_k+r_{k+1}=0\nequiv 1 \pmod{3}$.
Denote by $T(\bullet)$, respectively $B(\bullet)$, the height of
the top, respectively bottom, of the slit $\bullet$ with respect
to the top of the strip $\BB_k$, with positive downwards
direction. Since $B(q_{k+1})=2w_{\BB_k}+w_{\CC_k} < T(2q)
=2(w_{\BB_k}+w_{\CC_k}) <B(q_{k+2}) =3w_{\BB_k}+2w_{\CC_k}$ and
$B(2q_k) =w_{\BB_k}-w_{\AA_0}-w_{\CC_k}
<T(q_{k+1}+q_k)=w_{\BB_k}-w_{\AA_0}$, the slits $2q$, $q_{k+2}$,
$q_{k+1}$, $q_{k+1}+q_k$ and $2q_k$ lock all channels $\BB_k$,
$\CC_k$ and $\AA_0$. Two cases arise:

\begin{figure}[ht]
\centering
\unitlength 0.33mm
\begin{picture}(380,190)(-10,10)
\texture{ccccccc 0000}
\shade\path(0,90)(0,120)(30,120)(30,90)(0,90)
\shade\path(85,150)(130,150)(130,140)(85,140)(85,150)
\shade\path(220,75)(220,105)(250,105)(250,75)(220,75)
\shade\path(305,150)(350,150)(350,135)(305,135)(305,150)

\put(-22,75){\makebox(0,0){\small $(-1,-\eps)$}}
\put(-17,154){\makebox(0,0){\small $(-1,\eps)$}}
\put(198,75){\makebox(0,0){\small $(-1,-\eps)$}}
\put(203,154){\makebox(0,0){\small $(-1,\eps)$}}

\put(30,38){\makebox(0,0){\small $q$}}
\put(60,7){\makebox(0,0){\small $2q$}}
\put(85,82){\makebox(0,0){\small $\xi Q$}}
\put(102,192){\makebox(0,0){\small $q_k$}}
\put(130,192){\makebox(0,0){\small $q_{k+1}$}}
\put(160,162){\makebox(0,0){\small $q_{k+2}$}}
\put(250,23){\makebox(0,0){\small $q$}}
\put(280,7){\makebox(0,0){\small $2q$}}
\put(305,67){\makebox(0,0){\small $\xi Q$}}
\put(320,192){\makebox(0,0){\small $q_k$}}
\put(350,172){\makebox(0,0){\small $q_{k+1}$}}
\put(380,127){\makebox(0,0){\small $q_{k+2}$}}
\put(-10,100){\makebox(0,0){\small $\AA_0$}}
\put(-10,130){\makebox(0,0){\small $\CC_k$}}
\put(-10,143){\makebox(0,0){\small $\BB_k$}}
\put(210,90){\makebox(0,0){\small $\AA_0$}}
\put(210,120){\makebox(0,0){\small $\CC_k$}}
\put(210,142){\makebox(0,0){\small $\BB_k$}}

\put(47,100){\makebox(0,0){\small $w_{\AA_0}$}}
\put(148,145){\makebox(0,0){\small $w_{\BB_k}$}}
\put(148,130){\makebox(0,0){\small $w_{\CC_k}$}}
\put(368,142){\makebox(0,0){\small $w_{\BB_k}$}}
\put(368,117){\makebox(0,0){\small $w_{\CC_k}$}}
\put(267,95){\makebox(0,0){\small $w_{\AA_0}$}}
\put(148,115){\makebox(0,0){\small $w_{\BB_k}$}}
\put(368,97.5){\makebox(0,0){\small $w_{\BB_k}$}}
\put(396,75){\makebox(0,0){\small $w_{\CC_k}$}}
\put(176,98){\makebox(0,0){\small $w_{\CC_k}$}}
\put(115,165){\makebox(0,0){\small $w_{\AA_0}$}}

\path(30,75)(60,75) \path(0,90)(160,90) \path(160,110)(0,110)
\path(0,120)(30,120) \path(0,140)(100,140)\path(100,150)(130,150)

\path(250,75)(380,75)\path(380,90)(220,90)\path(220,105)(250,105)
\path(320,135)(220,135)\path(350,150)(320,150)

\dottedline{2}(30,45)(30,120) \dottedline{2}(100,140)(100,185)
\dottedline{2}(250,105)(250,30) \dottedline{2}(320,135)(320,185)

\thicklines \path(0,75)(30,75) \path(30,120)(130,120)
\path(100,140)(130,140)\path(0,150)(100,150)

\path(220,75)(250,75)\path(250,105)(350,105)\path(350,135)(320,135)
\path(320,150)(220,150)

\Thicklines \path(0,75)(0,150) \path(60,15)(60,90)
\path(130,110)(130,185) \path(160,80)(160,155)

\path(220,75)(220,150) \path(280,60)(280,15)
\path(380,45)(380,120) \path(350,90)(350,165)

\thinlines \path(125,140)(125,185)
\path(126.5,143)(125,140)(123.5,143)
\path(126.5,182)(125,185)(123.5,182)
\path(30,45)(60,45)\path(165,90)(165,110)\path(163.5,93)(165,90)(166.5,93)
\path(163.5,107)(165,110)(166.5,107) \path(385,60)(385,90)
\path(383.5,63)(385,60)(386.5,63)
\path(383.5,87)(385,90)(386.5,87) \path(355,90)(355,105)
\path(353.5,93)(355,90)(356.5,93)
\path(353.5,102)(355,105)(356.5,102) \path(135,110)(135,120)
\path(133.5,113)(135,110)(136.5,113)
\path(133.5,117)(135,120)(136.5,117)

\path(85,90)(85,150) \path(305,75)(305,150)
\path(60,15)(90,15)(90,60)(170,60) \path(130,185)(170,185)
\path(160,80)(170,80) \path(160,155)(170,155)

\path(35,75)(35,120) \path(33.5,78)(35,75)(36.5,78)
\path(33.5,117)(35,120)(36.5,117)

\path(135,120)(135,150) \path(133.5,123)(135,120)(136.5,123)
\path(133.5,137)(136.5,143) \path(133.5,143)(136.5,137)
\path(133.5,147)(135,150)(136.5,147)

\path(355,105)(355,150)\path(353.5,108)(355,105)(356.5,108)
\path(353.5,132)(356.5,138)\path(353.5,138)(356.5,132)\path(353.5,147)(355,150)(356.5,147)

\path(255,75)(255,105)\path(253.5,78)(255,75)(256.5,78)
\path(253.5,102)(255,105)(256.5,102)

\path(250.5,30)(280,30) \path(280,60)(380,60)
\path(380,45)(390,45) \path(380,120)(390,120)
\path(350,165)(390,165)

\end{picture}

\caption{The case $t \in I_{\gamma,k}$, $(\CC_\leftarrow, r=r_k=1)$,
 and $w_{\BB_k}+w_{\CC_k}<w_{\AA_0}$
respectively $w_{\BB_k}+w_{\CC_k}>w_{\AA_0}$} \label{Figure10}
\end{figure}

\subsubsection{$w_{\BB_k}<w_{\AA_0} \ (\Longleftrightarrow
t=\tan\omega <\gamma_{k+1})$}\label{Subsub6.1.1}

In this case $\BB_k$ is locked by the slit $q_{k+1}$ while $\AA_0$
is locked by the slits $q_{k+1}$, $q_{k+2}$ and $2q$. Two subcases
arise:

(I) $\ w_{\BB_k}+w_{\CC_k}<w_{\AA_0} \ (\Longleftrightarrow t <\frac{a+\eps}{q})$.
Then $B(q_{k+1})<T(2q)<2\eps$, so $\AA_0$ is
locked by the slits $q_{k+1}$, $q_{k+2}$ and $2q$ (see left-hand
side of Figure \ref{Figure10}). The widths of the relevant three
sub-channels of $\AA_0$ are (from bottom to top) $2\eps
-T(2q)=2(a+\eps -qt)$, $w_{\CC_k}$ and $w_{\BB_k}$,
so\footnote{Only $\AA_0$ and $\BB_k$ have to taken into account
here because $\CC_k$ has been already considered in the previous
section.}
\begin{equation}\label{6.1}
\begin{split}
W_{\gamma,k}(t)=W^{(1)}_{\gamma,k}(t) = & 2(a+\eps-qt) \cdot
(2q-\xi Q)_+ \wedge q +w_{\CC_k}(t) \cdot (q_{k+2}-\xi Q)_+ \wedge q \\
& +w_{\BB_k}(t) \cdot (q_{k+1}-\xi Q)_+ \wedge q+w_{\BB_k}(t)
\cdot (q_{k+1}-\xi Q)_+ \wedge q_k .
\end{split}
\end{equation}

(II) $\ w_{\BB_k} <w_{\AA_0}<w_{\BB_k}+w_{\CC_k}\
(\Longleftrightarrow t >\frac{a+\eps}{q})$. Then
$B(q_{k+1})=2w_{\BB_k}+w_{\CC_k} <2\eps <
T(2q)=2w_{\BB_k}+2w_{\CC_k}$, so $\AA_0$ is locked by the slits
$q_{k+1}$ and $q_{k+2}$ (see right-hand side of Figure
\ref{Figure10}). The widths of the relevant two sub-channels of
$\AA_0$ are (from bottom to top) $2\eps - B(q_{k+1}) = a_{k+1} -
q_{k+1}t$ and $w_{\BB_k}$, hence in this case
\begin{equation}\label{6.2}
\begin{split}
W_{\gamma,k}(t) = W_{\gamma,k}^{(2)}(t) = & (a_{k+1}-q_{k+1}t)\cdot
(q_{k+2}-\xi Q)_+ \wedge q +w_{\BB_k}(t)\cdot (q_{k+1}-\xi
Q)_+ \wedge q \\ & +w_{\BB_k}(t) \cdot (q_{k+1}-\xi Q)_+
\wedge q_k .
\end{split}
\end{equation}

\subsubsection{$w_{\AA_0} <w_{\BB_k}\
(\Longleftrightarrow t > \gamma_{k+1})$}\label{Subsub6.1.2}

In this case $\AA_0$ is locked by the slits $q_{k+1}$ and $\BB_k$
by $q_{k+1}$, $q_{k+1}+q_k$ and $2q_k$. Two subcases arise:

(III) $\ w_{\AA_0}<w_{\BB_k}<w_{\AA_0}+w_{\CC_k} \
(\Longleftrightarrow \gamma_{k+1} < t <\frac{a_k-\eps}{q_k})$.
Then $B(2q_k)<0<T(q_{k+1})$, so $\BB_k$ is locked by the slits
$q_{k+1}$ and $q_{k+1}+q_k$. The widths of the relevant two
sub-channels of $\BB_k$ are $w_{\AA_0}$ and $T(q_{k+1})=q_{k+1}t
-a_{k+1}$, hence in this case
\begin{equation}\label{6.3}
\begin{split}
W_{\gamma,k}(t) = W^{(3)}_{\gamma,k}(t) = & w_{\AA_0}(t) \cdot
(q_{k+1}-\xi Q)_+ \wedge q +w_{\AA_0}(t) \cdot
(q_{k+1}-\xi Q)_+ \wedge q_k \\ & + (q_{k+1} t -a_{k+1}) \cdot
(q_{k+1}+q_k -\xi Q)_+ \wedge q_k .
\end{split}
\end{equation}

(IV) $\ w_{\AA_0}+w_{\CC_k} <w_{\BB_k} \ (\Longleftrightarrow
\frac{a_k-\eps}{q_k}< t)$. Then $0<B(2q_k)<T(q_{k+1})$, so $\BB_k$
is locked by the slits $q_{k+1}$, $q_{k+1}+q_k$ and $2q_k$. The
widths of the relevant three sub-channels of $\BB_k$ are
$w_{\AA_0}$, $T(q_{k+1})-B(2q_k)=w_{\CC_k}$ and $B(2q_k)=2 (q_k t
-a_k+\eps)$, showing that in this case
\begin{equation}\label{6.4}
\begin{split}
W_{\gamma,k}(t) = & W^{(4)}_{\gamma,k} (t) = w_{\AA_0}(t) \cdot
(q_{k+1}-\xi Q )_+ \wedge q +w_{\AA_0}(t) \cdot (q_{k+1}-\xi Q)_+ \wedge q_k \\
& + w_{\CC_k}(t) \cdot (q_{k+1}+q_k -\xi Q)_+ \wedge q_k
+2(q_k t -a_k +\eps) \cdot (2q_k -\xi Q)_+ \wedge q_k .
\end{split}
\end{equation}

\subsection{The $\CC_\downarrow$ contribution}\label{Subsec6.2}
Since $r=1,r_k=1,r_k+r_{k+1}=r+r_{k+1}=0\nequiv 1 \pmod{3}$, none
of the corresponding slits is being removed. Moreover
$2r_{k+1}=1\nequiv 0\pmod{3}$, thus the slit $2q_{k+1}$ is not
removed either and the slits $q_{k+2}=q+q_{k+1}$, $q_k+q_{k+1}$
and $2q_{k+1}$ lock the central channel (see Figure
\ref{Figure11}). Furthermore, ordering $w_{\BB_k}$,
$B(q_k+q_{k+1})=2w_{\BB_k}-w_{\AA_0}$, $w_{\BB_k}+w_{\CC_k}$ and
$T(q_{k+2})=2w_{\BB_k}+w_{\CC_k}-w_{\AA_0}$ we find as in
Subsection \ref{Subsec6.1} that the corresponding weight $W_{\gamma,k} (t)$ is given
by\footnote{Only the contribution of $\CC_k$ needs to be taken
into account here.}

\begin{equation}\label{6.5}
\begin{split}
 & \begin{cases}
(w_{\BB_k}-w_{\AA_0})  (q_k+q_{k+1}-\xi Q)_+ \wedge q_{k+1} & \\
\quad +(w_{\AA_0}+w_{\CC_k}-w_{\BB_k}) (2q_{k+1}-\xi Q)_+
\wedge q_{k+1} & \mbox{\rm if
$w_{\AA_0}<w_{\BB_k}<w_{\AA_0}+w_{\CC_k},$} \\
w_{\CC_k}  (q_k+q_{k+1}-\xi Q)_+ \wedge q_{k+1} & \mbox{\rm
if $w_{\AA_0}+w_{\CC_k} <w_{\BB_k},$} \\
(w_{\AA_0}-w_{\BB_k})  (q_{k+2}-\xi Q)_+ \wedge q_{k+1} &
\\ \quad +(w_{\BB_k}+w_{\CC_k}-w_{\AA_0})  (2q_{k+1}-\xi Q)_+
\wedge q_{k+1} & \mbox{\rm if $w_{\BB_k}<w_{\AA_0}
<w_{\BB_k}+w_{\CC_k},$} \\ w_{\CC_k}  (q_{k+2}-\xi Q)_+ \wedge
q_{k+1} & \mbox{\rm if $w_{\BB_k}+w_{\CC_k}<w_{\AA_0},$}
\end{cases} \\
& =\begin{cases} W_{\gamma,k}^{(5)} (t) := w_{\CC_k}(t) \cdot
(q_{k+2}-\xi Q)_+ \wedge q_{k+1} & \mbox{\rm if $t
< \frac{a+\eps}{q} \wedge \gamma_{k+1} ,$} \\
W_{\gamma,k}^{(6)} (t) := (a_{k+1}-q_{k+1} t) \cdot
(q_{k+2}-\xi Q)_+ \wedge q_{k+1} &
\\ \qquad \qquad +2 (qt-a-\eps) \cdot (2q_{k+1}-\xi Q)_+ \wedge
q_{k+1} & \mbox{\rm if $\frac{a+\eps}{q} <t <\gamma_{k+1},$} \\
W_{\gamma,k}^{(7)}(t) := (q_{k+1}t -a_{k+1}) \cdot
(q_k+q_{k+1}-\xi Q)_+ \wedge q_{k+1} & \\
\qquad \qquad +2( a_k-\eps-q_k t ) \cdot (2q_{k+1}-\xi Q)_+ \wedge
q_{k+1} & \mbox{\rm if $\gamma_{k+1} <t <\frac{a_k-\eps}{q_k},$} \\
W_{\gamma,k}^{(8)}(t) := w_{\CC_k}(t) \cdot (q_k+q_{k+1}-\xi
Q)_+ \wedge q_{k+1} & \mbox{\rm if $ \frac{a_k-\eps}{q_k} \vee
\gamma_{k+1} <t.$}
\end{cases}
\end{split}
\end{equation}

\begin{figure}[ht]
\centering
\unitlength 0.36mm
\begin{picture}(270,175)(0,0)
\texture{ccccccc 0000}
\shade\path(95,100)(190,100)(190,90)(220,90)(220,75)(110,75)(110,90)(95,90)(95,100)

\put(-25,57){\makebox(0,0){\small $(0,-1-\eps)$}}
\put(-25,127){\makebox(0,0){\small $(0,-1+\eps)$}}

\put(30,2){\makebox(0,0){\small $q$}}
\put(80,172){\makebox(0,0){\small $q_k$}}
\put(95,68){\makebox(0,0){\small $\xi Q$}}
\put(110,42){\makebox(0,0){\small $q_{k+1}$}}
\put(140,2){\makebox(0,0){\small $q_{k+2}$}}
\put(190,162){\makebox(0,0){\small $q_k+q_{k+1}$}}
\put(220,112){\makebox(0,0){\small $2q_{k+1}$}}

\put(-10,69){\makebox(0,0){\small $\AA_0$}}
\put(-10,87){\makebox(0,0){\small $\CC_k$}}
\put(-10,110){\makebox(0,0){\small $\BB_k$}}
\put(263,82){\makebox(0,0){\small $w_{\AA_0}+w_{\CC_k}-w_{\BB_k}$}}
\put(160,95){\makebox(0,0){\small $w_{\BB_k}-w_{\AA_0}$}}
\put(96,112){\makebox(0,0){\small $w_{\BB_k}$}}
\put(46,67){\makebox(0,0){\small $w_{\AA_0}$}}

\path(0,75)(30,75) \path(110,75)(220,75) \path(220,90)(190,90)
\path(190,100)(110,100)\path(80,100)(0,100)

\dottedline{2}(110,50)(110,115)

\thicklines \path(30,75)(110,75)
\path(0,60)(30,60)\path(110,100)(80,100) \path(80,125)(0,125)

\Thicklines \path(0,60)(0,125) \path(30,75)(30,10)
\path(80,100)(80,165) \path(190,90)(190,155)
\path(220,40)(220,105) \path(140,65)(140,10)

\thinlines \path(110,50)(140,50) \path(140,65)(220,65)
\path(110,115)(190,115) \path(95,100)(95,75) \path(220,90)(0,90)
\path(30,10)(60,10)(60,25)(140,25)
\path(80,165)(160,165)(160,140)(190,140)
\path(220,40)(250,40)(250,55)(260,55) \path(220,105)(260,105)
\path(190,155)(260,155) \path(225,75)(225,90)
\path(223.5,78)(225,75)(226.5,78)
\path(223.5,87)(225,90)(226.5,87) \path(185,90)(185,100)
\path(183.5,93)(185,90)(186.5,93)
\path(183.5,97)(185,100)(186.5,97) \path(85,100)(85,125)
\path(83.5,103)(85,100)(86.5,103)
\path(83.5,122)(85,125)(86.5,122) \path(35,60)(35,75)
\path(33.5,63)(35,60)(36.5,63) \path(33.5,72)(35,75)(36.5,72)

\end{picture}

\caption{The case $t \in I_{\gamma,k}$, $(\CC_\downarrow, r=r_k=1)$,
 $w_{\AA_0}<w_{\BB_k}<w_{\AA_0}+w_{\CC_k}$}
\label{Figure11}
\end{figure}

\subsection{Estimating the total contribution}\label{Subsec6.3}
\subsubsection{$q_{k+1}>2Q$}\label{Subsub6.3.1}

In this case $k\geqslant 1$ and $t_{k-1}<\frac{a_k-\eps}{q_k}
<\gamma_{k+1}<\frac{a+\eps}{q}$. The cumulative contribution of
$\CC_\leftarrow$ and $\CC_\downarrow$ when $r=r_k=1$ or $r=r_k=-1$
and arising from \eqref{6.1}-\eqref{6.5} is
\begin{equation*}
\begin{split}
\G_{I,Q}^{(1.1)}(\xi) & =2\sum_{k=1}^\infty \
\sideset{}{^*}\sum_{\substack{q_k \in \II_{q,k} \\ q_{k+1}> 2Q}}
\int_{t_k}^{t_{k-1}} \frac{W^{(1)}_{\gamma,k}(t)+
+W^{(5)}_{\gamma,k}(t)}{t^2+t+1}\, dt = \G_{I,Q}^{(1.1.1)}
(\xi) + \G_{I,Q}^{(1.1.2)}
(\xi)+\G_{I,Q}^{(1.1.3)}(\xi),
\end{split}
\end{equation*}
with
\begin{equation*}
\begin{split}
& \G_{I,Q}^{(1.1.1)} (\xi) =2\sum_{k=1}^\infty  \sideset{}{^*}\sum_{\substack{q_k \in \II_{q,k} \\
q_{k+1}> 2Q}} \Big( (q_{k+1}-\xi Q)_+ \wedge q+(q_{k+1}-\xi Q)_+ \wedge
q_k \Big) \int_{t_k}^{t_{k-1}} \frac{w_{\BB_k}(t)\, dt}{t^2+t+1} , \\
& \G_{I,Q}^{(1.1.2)}(\xi)= 2 \sum_{k=1}^\infty \ \sideset{}{^*}\sum_{\substack{q_k \in \II_{q,k} \\
q_{k+1}> 2Q}} (2q-\xi Q)_+ \wedge q \int_{t_k}^{t_{k-1}}
\frac{2 (a+\eps-qt)\, dt}{t^2+t+1} ,\\
& \G_{I,Q}^{(1.1.3)}(\xi)=2\sum_{k=1}^\infty \ \sideset{}{^*}\sum_{\substack{q_k \in \II_{q,k} \\
q_{k+1}> 2Q}} \Big( (q_{k+2}-\xi Q)_+ \wedge q
+(q_{k+2}-\xi Q)_+ \wedge q_{k+1}\Big) \int_{t_k}^{t_{k-1}}
\frac{w_{\CC_k}(t)\, dt}{t^2+t+1} .
\end{split}
\end{equation*}
This quantity is estimated employing \eqref{5.6}, \eqref{A1.1}
and \eqref{5.7}. The cumulative contribution of the error terms from
those formulas to $\sum_I \G^{(1.1)}_{I,Q}(\xi)$ is
\begin{equation*}
\begin{split}
\ll \sum_I \sum_{\gamma\in\FF_I (Q)} \sum_{k=1}^\infty
\frac{1}{Qqq_{k-1}^2} & \leqslant \sum_I \sum_{\gamma\in\FF_I (Q)}
\frac{1}{Qq} \left( \frac{1}{q^{\prime
2}}+\frac{1}{qq^\prime}\right) \leqslant \frac{2}{Q} \sum_I
\sum\limits_{\gamma \in \FF_I(Q)} \frac{1}{qq^\prime} \leqslant
\frac{6\vert I\vert}{Q} ,
\end{split}
\end{equation*}
so they can be discarded. Following the outline from the end of Section \ref{Sect5} we
find
\begin{equation}\label{6.6}
\G_{I,Q}^{(1.1)}(\xi) \approxeq \frac{c_I}{8\zeta(2)} \int_0^1
du \int_{2-u}^\infty dw \, F^{(1.1)}(\xi;u,w),
\end{equation}
where
\begin{equation*}
\begin{split}
F^{(1.1)}(\xi;u,w)= & \frac{(1-u)^2}{(w-u)^2 w}\cdot
\Big( (w+u-\xi)_+ \wedge u+(w+u-\xi)_+ \wedge w\Big)
\\ & +\frac{1-u}{(w-u)w} \left(
\frac{w+u-2}{w-u}+\frac{w+2u-2}{w}\right) \cdot (2u-\xi)_+ \wedge u \\
& +\frac{(1-u)^2}{(w-u)w^2}\cdot \Big( (w+2u-\xi)_+ \wedge
u+(w+2u-\xi)_+ \wedge (w+u)\Big).
\end{split}
\end{equation*}

\subsubsection{$q_{k+1}\leqslant 2Q$}\label{Subsub6.3.2}
In this case $\frac{a+\eps}{q}\leqslant \gamma_{k+1}\leqslant
\frac{a_k-\eps}{q_k}$. Furthermore, when $k\geqslant 1$ we have
$\frac{a+\eps}{q}\leqslant t_k \leqslant \gamma_{k+1}\leqslant
\frac{a_k-\eps}{q_k}\leqslant t_{k-1}$ if $q_{k+1}\leqslant 2Q-q$,
and $t_k \leqslant \frac{a+\eps}{q} \leqslant
\gamma_{k+1}\leqslant \frac{a_k-\eps}{q_k} \leqslant t_{k-1}$ if
$2Q-q\leqslant q_{k+1}\leqslant 2Q$. When $k=0$ the only
difference is that $\gamma_1=t_{-1}\leqslant
\frac{a_0-\eps}{q_0}=\frac{a^\prime -\eps}{q^\prime}$. In this
case the cumulative contribution of $\CC_\leftarrow$ and
$\CC_\downarrow$ arising from \eqref{6.1}-\eqref{6.5} when
$r=r_k=1$ or $r=r_k=-1$ is
\begin{equation*}
\G_{I,Q}^{(1.2)}(\xi)=\G_{I,Q}^{(1.2.1)}(\xi)+\cdots
+\G_{I,Q}^{(1.2.5)}(\xi),
\end{equation*}
with
\begin{equation*}
\G_{I,Q}^{(1.2.1)}(\xi)=2\sum_{k=1}^\infty
\sideset{}{^*}\sum_{\substack{q_k\in\II_{q,k} \\ q_{k+1}\leqslant
2Q}} \int_{\frac{a_k-\eps}{q_k}}^{t_{k-1}} \frac{
W_{\gamma,k}^{(4)} (t)+W_{\gamma,k}^{(8)}(t)}{t^2+t+1}\, dt ,
\end{equation*}
\begin{equation*}
\G_{I,Q}^{(1.2.2)}(\xi)= 2 \sum_{k=1}^\infty
\sideset{}{^*}\sum_{\substack{q_k\in\II_{q,k} \\ q_{k+1}\leqslant
2Q}} \int_{\gamma_{k+1}}^{\frac{a_k-\eps}{q_k}} \frac{
W_{\gamma,k}^{(3)} (t) +W_{\gamma,k}^{(7)}(t)}{t^2+t+1}\, dt ,
\end{equation*}
\begin{equation*}
\G_{I,Q}^{(1.2.3)}(\xi) = 2\sum_{k=0}^\infty
\sideset{}{^*}\sum_{\substack{q_k\in\II_{q,k} \\ q_{k+1}\leqslant
2Q-q}} \int_{t_k}^{\gamma_{k+1}} \frac{W_{\gamma,k}^{(2)}(t) +
W_{\gamma,k}^{(6)}(t)}{t^2+t+1}\, dt ,
\end{equation*}
\begin{equation*}
\G_{I,Q}^{(1.2.4)}(\xi)=2\sum_{k=0}^\infty
\sideset{}{^*}\sum_{\substack{q_k\in\II_{q,k} \\ 2Q-q< q_{k+1}
\leqslant 2Q}} \int_{t_k}^{\frac{a+\eps}{q}}
\frac{W_{\gamma,k}^{(1)} (t)+ W_{\gamma,k}^{(5)}(t)}{t^2+t+1}\, dt
,
\end{equation*}
\begin{equation*}
\G_{I,Q}^{(1.2.5)}(\xi)=2\sum_{k=0}^\infty
\sideset{}{^*}\sum_{\substack{q_k\in\II_{q,k} \\ 2Q-q <
q_{k+1}\leqslant 2Q}} \int_{\frac{a+\eps}{q}}^{\gamma_{k+1}}
\frac{W_{\gamma,k}^{(2)} (t)+ W_{\gamma,k}^{(6)}(t)}{t^2+t+1}\, dt
.
\end{equation*}
The total contribution of the error terms from \eqref{A1.2},
\eqref{A1.3} and \eqref{A1.4} to $\sum_I \G_{I,Q}^{(1.2.1)}
(\xi)$ is
\begin{equation*}
\begin{split}
\ll \sum_{I} \sum_{\gamma\in\FF_I(Q)} & \sum_{k=1}^\infty \left(
\frac{1}{qq_{k-1}^3} +\frac{1}{q^2 q_{k-1} q_k}\right) \leqslant
\sum_{I} \sum_{\gamma\in\FF_I(Q)} \left( \frac{1}{q q^{\prime 3}}
+\frac{1}{Qq^2 q^\prime}+ \frac{1}{(qq^\prime)^2}
+\frac{1}{q^3 q^\prime} \right) \\ &
\leqslant \frac{2}{Q} \sum_{\gamma\in \FF(Q)} \frac{1}{qq^\prime}
+\frac{2}{Q} \sum_{q=1}^Q \frac{\varphi(q)}{q^2} +\sum_I \frac{8}{\vert I\vert Q}
\ll Q^{c-1},
\end{split}
\end{equation*}
showing that they can be discarded in the sequel. The same holds
for $\G^{(1.2.2)}_{I,Q}(\xi) , \ldots ,
\G_{I,Q}^{(1.2.5)}(\xi)$, where for
$\G_{I,Q}^{(1.2.4)}(\xi)$ and $\G_{I,Q}^{(1.2.5)}(\xi)$
one uses the fact that $k$ takes exactly one value as a result of
$q_{k+1}\in (2Q-q,2Q]$.

Employing \eqref{A1.2}, \eqref{A1.3}, \eqref{A1.4}, and proceeding
as in Section \ref{Sect5} we find
\begin{equation}\label{6.7}
\G_{I,Q}^{(1.2.1)}(\xi) \approxeq \frac{c_I}{8\zeta(2)}
\int_0^1 du \int_1^{2-u} dw\, F^{(1.2.1)}(\xi;u,w),
\end{equation}
with
\begin{equation*}
\begin{split}
F^{(1.2.1)} (\xi; u,w) & = \frac{(2-u-w)^2}{4(w-u)w^2}\cdot
\Big( (2w+u-\xi)_+ \wedge w+(2w+u-\xi)_+ \wedge (w+u)\Big) \\
& +\frac{2-u-w}{2(w-u)w} \left(
\frac{2w+u-2}{2w}+\frac{w-1}{w-u}\right) \cdot
\Big( (w+u-\xi)_+ \wedge u +(w+u-\xi)_+ \wedge w\Big) \\ &
 +\frac{(2-u-w)^2}{2(w-u)^2 w} \cdot (2w-\xi)_+
\wedge w .
\end{split}
\end{equation*}

Employing \eqref{A1.5}, \eqref{A1.6} and \eqref{A1.7} we find
\begin{equation}\label{6.8}
\G_{I,Q}^{(1.2.2)}(\xi) \approxeq \frac{c_I}{8\zeta(2)}
\int_0^1 du \int_1^{2-u} dw\, F^{(1.2.2)}(\xi;u,w),
\end{equation}
with
\begin{equation*}
\begin{split}
F^{(1.2.2)} & (\xi; u,w) = \frac{(2-u-w)^2}{4w^2 (w+u)} \cdot
\Big( (2w+u-\xi)_+ \wedge w +(2w+u-\xi)_+ \wedge (w+u)\Big) \\
& +\frac{2-u-w}{2w(w+u)} \left( \frac{w+u-1}{w+u} +
\frac{2w+u-2}{2w}\right)\cdot \Big( (w+u-\xi)_+
\wedge u +(w+u-\xi)_+ \wedge w\Big)  \\
& +\frac{(2-u-w)^2}{2w(w+u)^2} \cdot ( 2w +2u - \xi )_+ \wedge (w+u).
\end{split}
\end{equation*}

Employing \eqref{A1.8}, \eqref{A1.9} and \eqref{A1.10} we find
\begin{equation}\label{6.9}
\G_{I,Q}^{(1.2.3)} (\xi)\approxeq \frac{c_I}{8\zeta(2)}
\int_0^1 du \int_{1-u}^{2-2u} dw\, F^{(1.2.3)} (\xi;u,w) ,
\end{equation}
with
\begin{equation*}
\begin{split}
F^{(1.2.3)}(\xi;u,w) = & \frac{(w+u-1)^2}{w(w+u)^2} \cdot
\Big( (w+u-\xi)_+ \wedge u+(w+u-\xi)_+ \wedge w\Big) \\
& +\frac{w+u-1}{w(w+u)} \left( \frac{2-u-w}{w+u} +\frac{2-2u-w}{w}
\right) \cdot (2w+2u-\xi)_+ \wedge (w+u) \\
& +\frac{(w+u-1)^2}{w^2 (w+u)} \cdot \Big( (w+2u-\xi)_+ \wedge
u+(w+2u-\xi)_+ \wedge (w+u)\Big).
\end{split}
\end{equation*}

Employing \eqref{A1.13}-\eqref{A1.17} we find
\begin{equation}\label{6.10}
\G_{I,Q}^{(1.2.4)} (\xi) \approxeq \frac{c_I}{8\zeta(2)}
\int_0^1 du \int_{2-2u}^{2-u} dw\, F^{(1.2.4)} (\xi;u,w) ,
\end{equation}
with
\begin{equation*}
\begin{split}
& F^{(1.2.4)} (\xi; u,w) = \frac{(w+2u-2)^2}{4u^2 w} \cdot
\Big( (w+u-\xi)_+ \wedge u+(w+u-\xi)_+ \wedge w\Big) \\
& \quad +\frac{w+2u-2}{2uw} \left( \frac{2-u-w}{2u}
+\frac{1-u}{w} \right) \cdot \Big( (w+2u-\xi)_+ \wedge u+(w+2u-\xi)_+ \wedge (w+u)\Big) \\
& \quad +\frac{(w+2u-2)^2}{2uw^2} \cdot (2u-\xi)_+ \wedge u .
\end{split}
\end{equation*}

Employing \eqref{A1.17}, \eqref{A1.18} and \eqref{A1.19} we find
\begin{equation}\label{6.11}
\G_{I,Q}^{(1.2.5)} (\xi) \approxeq \frac{c_I}{8\zeta(2)}
\int_0^1 du \int_{2-2u}^{2-u} dw\, F^{(1.2.5)} (\xi;u,w) ,
\end{equation}
with
\begin{equation*}
\begin{split}
F^{(1.2.5)} & (\xi; u,w) = \frac{(2-u-w)^2}{4u^2(w+u)} \cdot
\Big( (w+2u-\xi)_+ \wedge u+(w+2u-\xi)_+ \wedge (w+u)\Big) \\
& +\frac{2-u-w}{2u(w+u)} \left( \frac{w+u-1}{w+u} +
\frac{w+2u-2}{2u} \right) \cdot \Big( (w+u-\xi)_+ \wedge u+(w+u-\xi)_+ \wedge w \Big) \\
& +\frac{(2-u-w)^2}{2u(w+u)^2} \cdot (2w+2u-\xi)_+ \wedge (w+u).
\end{split}
\end{equation*}

\section{Channels with removed slits. The case $r=0$}\label{Sect7}

In this case $r_k=\pm 1$. The $\CC_O$ contributions for
$(r,r_k)=(0,1)$ and respectively $(r,r_k)=(0,-1)$ do coincide. The
$\CC_\leftarrow$ contribution for $(r,r_k)=(0,1)$ coincides with
the $\CC_\downarrow$ contribution for $(r,r_k)=(0,-1)$. We shall
sum as in \eqref{5.2}, considering

\begin{equation*}
\sideset{}{^*}\sum =\sideset{}{^*}\sum_{\substack{\gamma\in\FF_I(Q) \\
q\equiv a \hspace{-6pt} \pmod{3} \\ q_k \nequiv a_k \hspace{-6pt}
\pmod{3}}} .
\end{equation*}
It suffices to only analyze the $\CC_O$ and $\CC_\leftarrow$
contributions of $(r,r_k)=(0,1)$, allowing at the very end $\beta$
to take both values $-1$ and $1$. The final result will express
the $\CC_O$ contribution when $(r,r_k)=(0,\pm 1)$, and
respectively the sum of the $\CC_\leftarrow$ contribution for
$(r,r_k)=(0,1)$ and of the $\CC_\downarrow$ contribution for
$(r,r_k)=(0,-1)$.

\subsection{The $\CC_O$ contribution}\label{Subsec7.1}
The situation is described in Figure \ref{Figure12}.
Two cases arise:

\subsubsection{$w_{\AA_0} \leqslant w_{\BB_k} \ (\Longleftrightarrow
t \geqslant\gamma_{k+1})$}\label{Subsub7.1.1}

In this case $k\geqslant 1$. Since $\gamma_{k+1}\leqslant
t_{k-1}$, we must also have $q_{k+1}\leqslant 2Q$. The channel
$\AA_0$ is locked by the slit $q_{k+1}$ and
$W_{\gamma,k}(t)=w_{\AA_0}(t)\cdot (q_{k+1}-\xi Q)_+ \wedge
q$, with contribution
\begin{equation*}
\G_{I,Q}^{(2.1)}(\xi)=\sum_{k=1}^\infty
\sideset{}{^*}\sum_{\substack{q_k \in\II_{q,k} \\ q_{k+1}\leqslant
2Q}} (q_{k+1}-\xi Q)_+ \wedge q \int_{\gamma_{k+1}}^{t_{k-1}}
\frac{w_{\AA_0}(t)\, dt}{t^2+t+1}  ,
\end{equation*}
estimated in Subsection \ref{SubSub11.1.1} as
\begin{equation}\label{7.1}
\approxeq \frac{c_I}{8\zeta(2)} \int_0^1
\hspace{-3pt} du \int_1^{2-u}  dw\, \frac{2-u-w}{(w-u)(w+u)}  \left(
\frac{w+u-1}{w+u}+\frac{w-1}{w-u}\right) \cdot (w+u-\xi)_+ \wedge u.
\end{equation}

\subsubsection{$w_{\AA_0} >w_{\BB_k} \ (\Longleftrightarrow t <\gamma_{k+1})$ and $\xi Q
<q_{k+1}$}\label{Subsub7.1.2}

In this situation we have
\begin{equation*}
W_{\gamma,k}(t) =w_{\BB_k}(t) (q_{k+1}-\xi Q) + \big(
w_{\AA_0} (t) -w_{\BB_k}(t)\big) q,
\end{equation*}
with contribution (according to whether $\gamma_{k+1}\leqslant t_{k-1}$ or $t_{k-1}< \gamma_{k+1}$)
\begin{equation*}
\G_{I,Q}^{(2.2)} (\xi)=\sum_{k=0}^\infty
\sideset{}{^*}\sum_{\substack{q_k\in\II_{q,k}
\\ \xi Q<q_{k+1}\leqslant 2Q}} \int_{t_k}^{\gamma_{k+1}}
\frac{W_{\gamma,k}(t)\, dt}{t^2+t+1}  + \sum_{k=1}^\infty
\sideset{}{^*}\sum_{\substack{q_k \in \II_{q,k} \\
q_{k+1}>(\xi\vee 2)Q}} \int_{t_k}^{t_{k-1}}
\frac{W_{\gamma,k}(t)\, dt}{t^2+t+1} .
\end{equation*}

Employing \eqref{A1.8}, \eqref{A1.9}, respectively \eqref{5.6},
\eqref{A2.2} and the procedure described at the beginning of
Appendix 2 we find
\begin{equation}\label{7.2}
\G_{I,Q}^{(2.2)}(\xi) \approxeq \frac{c_I}{8\zeta(2)} \int_0^1
\hspace{-3pt} du \left(\int_{(\xi\vee 1) \wedge 2 - u}^{2-u} \hspace{-6pt} dw\,
F^{(2.2.1)}(\xi;u,w) + \int_{\xi\vee 2 -u}^\infty \hspace{-6pt} dw\,
F^{(2.2.2)}(\xi;u,w) \right),
\end{equation}
with
\begin{equation*}
\begin{split}
& F^{(2.2.1)}(\xi;u,w)=\frac{(w+u-1)^2}{w(w+u)} \left(
\frac{w+u}{w}-\frac{\xi}{w+u}\right) ,\\ &
F^{(2.2.2)}(\xi;u,w)=\frac{(1-u)^2(w+u-\xi)}{(w-u)^2 w}
+\frac{(1-u)u}{(w-u)w} \left(
\frac{w+u-1}{w}+\frac{w+u-2}{w-u}\right) .
\end{split}
\end{equation*}

\begin{figure}[ht]
\centering
\unitlength 0.43mm
\begin{picture}(220,160)(-40,-7)
\texture{ccccccc 0000}
\shade\path(0,35)(0,40)(30,40)(30,35)(0,35)
\shade\path(155,40)(170,40)(170,75)(155,75)(155,40)

\put(-17,32){\makebox(0,0){\small $(0,-\eps)$}}
\put(-14,124){\makebox(0,0){\small $(0,\eps)$}}

\put(30,-7){\makebox(0,0){\small $q$}}
\put(60,-7){\makebox(0,0){\small $2q$}}
\put(110,145){\makebox(0,0){\small $q_k$}}
\put(155,-6){\makebox(0,0){\small $\xi Q$}}
\put(140,145){\makebox(0,0){\small $q_{k+1}$}}
\put(170,130){\makebox(0,0){\small $q_{k+2}$}}
\put(200,95){\makebox(0,0){\small $q_{k+3}$}}

\put(-10,60){\makebox(0,0){\small $\AA_0$}}
\put(-10,96){\makebox(0,0){\small $\CC_k$}}
\put(-10,114){\makebox(0,0){\small $\BB_k$}}

\put(154,98){\makebox(0,0){\small $w_{\CC_k}$}}
\put(154,80){\makebox(0,0){\small $w_{\BB_k}$}}
\put(184,63){\makebox(0,0){\small $w_{\CC_k}$}}
\put(184,45){\makebox(0,0){\small $w_{\BB_k}$}}
\put(-55,72){\makebox(0,0){\small $w_{\AA_0}+w_{\CC_k}$}}

\path(140,75)(170,75) \path(0,35)(200,35)\path(170,40)(200,40)

\dottedline{2}(0,35)(0,120) \dottedline{2}(30,0)(30,85)
\dottedline{2}(60,0)(60,50)

\thicklines \path(0,120)(110,120) \path(110,110)(140,110)
\path(140,85)(30,85) \path(0,35)(30,35)

\Thicklines \path(110,110)(110,140) \path(140,75)(140,140)
\path(170,40)(170,125) \path(200,5)(200,90)

\thinlines \path(0,110)(110,110)\path(0,85)(30,85)
\path(0,75)(140,75)\path(170,50)(0,50)\path(0,40)(170,40)

\path(145,75)(145,110) \path(143.5,107)(145,110)(146.5,107)
\path(143.5,88)(145,85)(146.5,88)
\path(143.5,82)(145,85)(146.5,82)
\path(143.5,78)(145,75)(146.5,78)

\path(175,40)(175,75) \path(173.5,72)(175,75)(176.5,72)
\path(173.5,53)(175,50)(176.5,53)
\path(173.5,47)(175,50)(176.5,47)
\path(173.5,43)(175,40)(176.5,43)

\path(-35,35)(-35,110) \path(-36.5,38)(-35,35)(-33.5,38)
\path(-36.5,107)(-35,110)(-33.5,107)

\path(155,0)(155,75)

\end{picture}

\caption{The case $t \in I_{\gamma ,k}$, $r=0$, $r_k=\pm 1$,
$\CC_{O}$, $w_{\BB_k}<w_{\AA_0}$, $n_0=2$, $N=1$} \label{Figure12}
\end{figure}

\subsubsection{$w_{\AA_0} >w_{\BB_k} \ (\Longleftrightarrow t <\gamma_{k+1})$ and $\xi Q \geqslant
q_{k+1}$}\label{7.1.3}

Consider the integer $N$ for which $q_{k+N} \leqslant \xi
Q<q_{k+N+1}$, that is $1\leqslant N:=\Big\lfloor \frac{\xi
Q-q_k}{q}\Big\rfloor$. We will keep $N\geqslant 1$ and $q\leqslant
Q$ fixed and sum over
\begin{equation*}
y=q_k\in\JJ_{q,N}:=\big( \xi Q-(N+1)q,\xi Q-Nq\big].
\end{equation*}
Let $n_0:=\Big\lfloor
\frac{w_{\AA_0}+w_{\CC_k}}{w_{\BB_k}+w_{\CC_k}}\Big\rfloor
=\Big\lfloor \frac{a_k-q_k t}{qt -a}\Big\rfloor \geqslant 1$, so
$\gamma_{k+n_0+1}<t \leqslant \gamma_{k+n_0}$ and
$B(q_{k+n_0})=w_{\BB_k}+n_0 (w_{\BB_k}+w_{\CC_k})\leqslant 2\eps
<B(q_{k+n_0+1})$. This shows that the channel $\AA_0$ is locked by
the slits $q_{k+1},\ldots,q_{k+n_0},q_{k+n_0+1}$. Suppose that
$N\geqslant \xi$. Then $q_k \leqslant \xi Q-Nq \leqslant
N(Q-q)$, showing that $q_{k+N}\leqslant NQ$ and $\gamma_{k+N}
<t_k$. So $n_0 <N$ and $\xi Q \geqslant q_{k+N} \geqslant
q_{k+n_0+1}$, showing that in this case $W_{\gamma,k,N}(t)=0$, $\forall
t\in I_{\gamma,k}$.
It remains that $N<\xi$. The following cases arise:

(1) $\ B(q_{k+N}) \leqslant 2\eps <B(q_{k+N+1})$
($\Longleftrightarrow \gamma_{k+N+1} <t\leqslant \gamma_{k+N}$).
In this case
\begin{equation}\label{7.3}
\begin{split}
W_{\gamma,k,N}(t)=W_{\gamma,k,N}^{(1)}(t) & := \big( 2\eps -B(q_{k+N})\big) (q_{k+N+1}-\xi
Q) \\ & =(a_{k+N}-q_{k+N}t) (q_{k+N+1}-\xi Q) .
\end{split}
\end{equation}

(2) $\ B(q_{k+N+1})\leqslant 2\eps$ ($\Longleftrightarrow
t\leqslant \gamma_{k+N+1}$). In this case
\begin{equation}\label{7.4}
\begin{split}
W_{\gamma,k,N}(t) = W_{\gamma,k}^{(2)}(t) := & \big( w_{\BB_k}(t)+w_{\CC_k}(t)\big)
(q_{k+N+1}-\xi Q) +\big( 2\eps -B(q_{k+N+1})\big) q \\  = &
1-\xi Q \big( w_{\BB_k}(t)+w_{\CC_k}(t)\big) .
\end{split}
\end{equation}

We find
\begin{equation*}
\begin{split}
\G_{I,Q}^{(2.3)}(\xi) & =\sum_{1\leqslant N<\xi}
\sum_{k=0}^\infty\ \sideset{}{^*}\sum_{q_k \in \II_{q,k}\cap\JJ_{q,N}} \left(
\int_{I_{\gamma,k}\cap (\gamma_{k+N+1},\gamma_{k+N}]}
\frac{W^{(1)}_{\gamma,k,N}(t)\, dt}{t^2+t+1}
+\int_{I_{\gamma,k}\cap (\gamma,\gamma_{k+N+1}]}
\frac{W^{(2)}_{\gamma,k,N}(t)\, dt}{t^2+t+1} \right) \\
& =\G_{I,Q}^{(2.3.1)}(\xi) +\G_{I,Q}^{(2.3.2)}(\xi) .
\end{split}
\end{equation*}
Note that
\begin{equation*}
I_{\gamma,k}\cap (\gamma_{k+N+1},\gamma_{k+N}] = \begin{cases}
\emptyset & \mbox{\rm if $q_k<N(Q-q)$ or $Q+(N+1)(Q-q)\leqslant q_k,$} \\
(t_k,\gamma_{k+N}] &
\mbox{\rm if $N(Q-q)\leqslant q_k <(N+1)(Q-q),$} \\
(\gamma_{k+N+1},\gamma_{k+N}] & \mbox{\rm if $(N+1)(Q-q) \leqslant
q_k <Q+N(Q-q),$} \\
(\gamma_{k+N+1},t_{k-1}] & \mbox{\rm if $Q+N(Q-q) \leqslant q_k <
Q+(N+1)(Q-q),$}
\end{cases}
\end{equation*}
and respectively
\begin{equation*}
I_{\gamma,k}\cap (\gamma,\gamma_{k+N+1}]=
\begin{cases}
\emptyset & \mbox{\rm if $q_k <(N+1)(Q-q),$} \\
(t_k,\gamma_{k+N+1}] & \mbox{\rm if $(N+1)(Q-q) \leqslant q_k
<Q+(N+1)(Q-q),$} \\
I_{\gamma,k} & \mbox{\rm if $Q+(N+1)(Q-q) \leqslant q_k .$}
\end{cases}
\end{equation*}

Summing as in \eqref{5.2} and
employing \eqref{A2.3}, \eqref{A2.4}, \eqref{A2.5} and Lemmas
\ref{L3.4} and \ref{L3.2}, then changing $y+Nq$ to $y$ and making
the substitution $(q,y)=(Qu,Qw)$, we find
\begin{equation}\label{7.5}
\begin{split}
& \G_{I,Q}^{(2.3.1)} (\xi) \approxeq \frac{c_I}{8\zeta(2)}
\sum_{1\leqslant N < \xi} \int_0^1 du \left\{
\int_{[\xi-u,\xi] \cap [N,N+1-u]} dw \,
\frac{(w+u-\xi)(w-N)^2}{(w-Nu)^2 w} \right.
\\ & \hspace{5cm} + \int_{[\xi-u,\xi] \cap [N+1-u,N+1]} dw\,
\frac{w+u-\xi}{(w+u)^2 w} \\ & + \left. \int_{[\xi -u,\xi] \cap [N+1,N+2-u]}
 dw\, \frac{(w+u-\xi)(N+2-u-w)}{\big(
w-(N+1)u\big) (w+u)} \left(
\frac{1}{w+u}+\frac{w-(N+1)}{w-(N+1)u}\right)\right\} .
\end{split}
\end{equation}
In a similar way (but replacing $y+(N+1)q$ by $y$) we infer from
\eqref{A2.6} and \eqref{A2.7} that
\begin{equation}\label{7.6}
\begin{split}
\G_{I,Q}^{(2.3.2)}(\xi) & \approxeq \frac{c_I}{8\zeta (2)}
\sum_{1\leqslant N<\xi} \int_0^1 du \int_{[\xi ,\xi+u]\cap
[N+1,N+2]}dw\, \frac{w-(N+1)}{\big( w-(N+1)u\big) w} \\ & \hspace{3cm} \cdot \left(
2-\frac{\xi}{w}-\frac{\xi (1-u)}{w-(N+1)u}\right)
 \\ &  +\frac{c_I}{8\zeta (2)}
\sum_{1\leqslant N<\xi} \int_0^1 du \int_{[\xi ,\xi+u] \cap [N+2,\infty)} dw\,
\frac{1-u}{\big( w-(N+2)u\big) \big( w-(N+1)u\big)} \\ & \hspace{3cm} \cdot \left(
2-\frac{\xi (1-u)}{w-(N+2)u} -\frac{\xi (1-u)}{w-(N+1)u}
\right) .
\end{split}
\end{equation}

The inner integrals in \eqref{7.5} and the first one in
\eqref{7.6} can be nonzero only when $N>\xi -1$,
$\xi-2<N<\xi$ or $\xi -2<N<\xi -1$.

\subsection{The $\CC_\leftarrow$ contribution}\label{Subsec7.2}

In this case we shall analyze the contribution of the channels
$\BB_k$ and $\CC_k$. All slits $q_{k+n}$, $n\in\Z$, are removed,
while neither $q$ nor any of $2q_k+nq=q_k+q_{k+n}$, $n\in\Z$, is
being removed. Since $T(2q_k)=-2w_{\AA_0}-2w_{\CC_k}<0$ and
$B(2q_k+nq) = w_{\BB_k} - w_{\AA_0} - w_{\CC_k} + n (w_{\BB_k} +
w_{\CC_k})$, it follows that $\BB_k \cup \CC_k$ is locked exactly
by two of the slits $2q_k+nq$, $n\geqslant 0$. To make this
precise let
$n_0:=\Big\lfloor\frac{w_{\AA_0}+2w_{\CC_k}}{w_{\BB_k}+w_{\CC_k}}
\Big\rfloor =\Big\lfloor \frac{a_k+a_{k-1}-2\eps
-(q_k+q_{k-1})t}{qt -a} \Big\rfloor\geqslant 0$. We have $0<B(2q_k
+n_0 q)=w_{\BB_k}-w_{\AA_0}-w_{\CC_k} +n_0 (w_{\BB_k}+w_{\CC_k})
\leqslant w_{\BB_k}+w_{\CC_k}$. The situation is described in
Figure \ref{Figure13}. Consider also
\begin{equation}\label{7.7}
\lambda_{k,n}:=\frac{a_{k}+a_{k+n-1}-2\eps}{q_{k}+q_{k+n-1}}
\searrow \gamma \quad \mbox{\rm as $n\rightarrow \infty,$}
\end{equation}
so $\lambda_{k,n_0+1} <t \leqslant \lambda_{k,n_0}$. Note that
$\lambda_{0,1}\geqslant t_{-1} =\gamma_1 >\lambda_{0,2}$ and
$\lambda_{k,0}>t_{k-1}$ when $k\geqslant 1$, showing that for
every $k\geqslant 0$ the intervals
$(\lambda_{k,n+1},\lambda_{k,n}]$ cover $I_{\gamma,k}$. The
following cases arise:

\subsubsection{$0<B(2q_k+n_0 q)\leqslant w_{\BB_k}\ (\Longleftrightarrow\
\lambda_{k,n_0+1}<t \leqslant \gamma_{k+n_0})$}\label{Subsub7.2.1}

In this case $\CC_k$ is locked by the slit $2q_k+(n_0+1)q$ and
$\BB_k$ by the slits $2q_k+n_0q=q_k+q_{k+n_0}$ and
$2q_k+(n_0+1)q=q_k+q_{k+n_0+1}$. The widths of the three relevant
sub-channels of $\BB_k \cup \CC_k$ are (from bottom to top)
$w_{\CC_k}$, $w_{\BB_k}-B(2q_k+n_0 q)= a_{k+n_0}-q_{k+n_0}t$,
$B(2q_k+n_0 q) =w_{\BB_k} -(a_{k+n_0}-q_{k+n_0}t)$, and so
\begin{equation*}
\begin{split}
W_{\gamma,k}(t)  = & W_{\gamma,k,n_0}^{(1)} (t) = w_{\CC_k}(t) \cdot
(q_{k+1}+q_{k+n_0}-\xi Q)_+ \wedge q_{k+1} +w_{\BB_k}(t) \cdot
(q_k+q_{k+n_0}-\xi Q)_+ \wedge q_k \\
& +(a_{k+n_0}-q_{k+n_0}t) \cdot \Big( (q_{k+1}+q_{k+n_0}-\xi Q
)_+\wedge q_k -(q_k+q_{k+n_0}-\xi Q)_+ \wedge q_k \Big).
\end{split}
\end{equation*}

\begin{figure}[ht]
\centering
\unitlength 0.38mm
\begin{picture}(240,200)(0,0)
\texture{ccccccc 0000}
\shade\path(0,85)(0,95)(80,95)(80,85)(0,85)
\shade\path(0,115)(0,130)(50,130)(50,115)(0,115)
\shade\path(100,95)(100,115)(160,115)(160,95)(100,95)

\put(-17,42){\makebox(0,0){\small $(0,-\eps)$}}
\put(-14,134){\makebox(0,0){\small $(0,\eps)$}}

\put(30,-6){\makebox(0,0){\small $q$}}
\put(50,196){\makebox(0,0){\small $q_k$}}
\put(127,197){\makebox(0,0){\small $2q_k+q$}}
\put(160,186){\makebox(0,0){\small $2q_k+2q$}}
\put(190,39){\makebox(0,0){\small $2q_k+3q$}}
\put(100,166){\makebox(0,0){\small $\xi Q$}}
\put(80,64){\makebox(0,0){\small $q_{k+1}$}}
\put(140,-6){\makebox(0,0){\small $q_{k+3}$}}
\put(112,19){\makebox(0,0){\small $q_{k+2}$}}

\put(-10,65){\makebox(0,0){\small $\AA_0$}}
\put(-10,95){\makebox(0,0){\small $\CC_k$}}
\put(-10,123){\makebox(0,0){\small $\BB_k$}}

\put(186.5,117){\makebox(0,0){\small $w_{\BB_k}+w_{\CC_k}$}}
\put(218,72){\makebox(0,0){\small $w_{\BB_k}+w_{\CC_k}$}}

\path(110,25)(140,25) \path(140,25)(140,65)(190,65)
\path(80,85)(190,85) \path(160,95)(190,95) \path(130,140)(160,140)
\path(80,155)(130,155) \path(160,180)(210,180)
\path(210,165)(240,165) \path(190,135)(240,135)
\path(80,70)(110,70) \path(160,95)(190,95)
 \path(80,115)(160,115)\path(50,130)(160,130)

\dottedline{2}(50,115)(50,190) \dottedline{2}(80,70)(80,155)
\dottedline{2}(110,25)(110,110) \dottedline{2}(140,65)(140,0)

\thicklines \path(0,130)(50,130) \path(50,115)(80,115)
\path(80,85)(30,85) \path(30,45)(0,45)

\Thicklines \path(0,45)(0,130) \path(130,140)(130,190)
\path(160,95)(160,180) \path(190,50)(190,135) \path(30,0)(30,85)
\path(210,165)(210,190) \path(240,135)(240,190)

\thinlines \path(100,0)(100,160) \path(0,85)(30,85)
\path(0,110)(160,110) \path(0,115)(50,115) \path(0,95)(160,95)
\path(165,95)(165,140) \path(195,50)(195,95)
\path(163.5,98)(165,95)(166.5,98)
\path(163.5,137)(165,140)(166.5,137)
\path(193.5,53)(195,50)(196.5,53)
\path(193.5,92)(195,95)(196.5,92)

\end{picture}
\caption{The case $t \in I_{\gamma,k}$, $r=0$, $r_k=1$,
$\CC_{\leftarrow}$, $n_0=2$, $w_{\BB_k} <B(2q_k +n_0 q)\leqslant
w_{\BB_k}+w_{\CC_k}$} \label{Figure13}
\end{figure}

\subsubsection{$w_{\BB_k}<B(2q_k+n_0q) \leqslant
w_{\BB_k}+w_{\CC_k}\ (\Longleftrightarrow\ \gamma_{k+n_0} <t
\leqslant \lambda_{k,n_0})$}\label{Subsub7.2.2}

In this case $\BB_k$ is locked by the slit $q_k+q_{k+n_0}$ and
$\CC_k$ by the slits $q_k+q_{k+n_0}$ and $q_k+q_{k+n_0+1}$. The
widths of the three relevant sub-channels of $\BB_k \cup \CC_k$
are $w_{\BB_k}+w_{\CC_k}-B(2q_k+n_0 q)=w_{\CC_k}-(q_{k+n_0}t -
a_{k+n_0})$, $B(2q_k+n_0 q)-w_{\BB_k}=q_{k+n_0} t -a_{k+n_0}$,
$w_{\BB_k}$, and so
\begin{equation*}
\begin{split}
W_{\gamma,k}(t) = & W_{\gamma,k,n_0}^{(2)} (t) = w_{\CC_k}(t)\cdot
(q_{k+1}+q_{k+n_0}-\xi Q)_+ \wedge q_{k+1}  +w_{\BB_k}(t) \cdot
(q_k+q_{k+n_0}-\xi Q)_+ \wedge q_k \\
& -(q_{k+n_0}t -a_{k+n_0})\cdot\Big( (q_{k+1}+q_{k+n_0}-\xi Q
)_+ \wedge q_{k+1} -(q_k+q_{k+n_0}-\xi Q)_+ \wedge q_{k+1}
\Big) .
\end{split}
\end{equation*}

The following five cases arise:

(I) $\ q_k <(n-1)(Q-q)$. Then $\lambda_{k,n} <t_k$, thus
$I_{\gamma,k}\cap (\lambda_{k,n+1},\lambda_{k,n}]=\emptyset$.

(II) $\ (n-1)(Q-q) \leqslant q_k <n(Q-q)$. Then $n\geqslant
2$ and $\gamma_{k+n} <t_k \leqslant \lambda_{k,n} < t_{k-1}$ in
both cases $k\geqslant 1$ and $k=0$. The corresponding
contribution is
\begin{equation*}
\G_{I,Q}^{(2.4.1)}(\xi)=\sum_{n=2}^\infty \sum_{k=0}^\infty \
\sideset{}{^*}\sum_{\substack{q_k\in \II_{q,k} \\ (n-1)(Q-q)
\leqslant q_k <n(Q-q)}} \int_{t_k}^{\lambda_{k,n}}
\frac{W^{(2)}_{\gamma,k,n}(t)\, dt}{t^2+t+1} .
\end{equation*}
Employing \eqref{A2.8}-\eqref{A2.11}, Lemmas \ref{L3.8}
and \ref{L3.3}, and the change of variable $(q,y)=(Qu,Qw)$ we find
(according to whether $k=0$ or $k\geqslant 1$)
\begin{equation}\label{7.8}
\begin{split}
\G_{I,Q}^{(2.4.1)} (\xi) \approxeq & \frac{c_I}{8\zeta(2)}
\sum_{n=2}^\infty \int_{1-\frac{1}{n-1}}^1 du \int_{(n-1)(1-u)}^{1} dw\
F_n^{(2.4.1)}(\xi;u,w) \\ &  + \frac{c_I}{8\zeta(2)} \sum_{n=2}^\infty
\int_0^{1-\frac{1}{n-1}} du
\int_{(n-1)(1-u)}^{n(1-u)} dw\,   F_n^{(2.4.1)}(\xi;u,w)
\\ &  + \frac{c_I}{8\zeta(2)} \sum_{n=2}^\infty
\int_{1-\frac{1}{n-1}}^{1-\frac{1}{n}} du \int_1^{n(1-u)} dw\,
F_n^{(2.4.1)}(\xi;u,w) ,
\end{split}
\end{equation}
where
\begin{equation*}
\begin{split}
& F_n^{(2.4.1)} (\xi;u,w) = \frac{\big( w - (n-1) (1-u) \big)^2
}{w\big( 2w+(n-1)u\big)^2}
\cdot (2w+nu-\xi)_+ \wedge w  \\
& \quad +\frac{w-(n-1)(1-u)}{w\big( 2w+(n-1)u\big)} \left(
\frac{1-u}{w}-\frac{n(1-u)-w}{w}\right) \cdot \big(
2w+(n+1)u-\xi\big)_+ \wedge (w+u)
\\ & \quad +\frac{w-(n-1)(1-u)}{w\big( 2w+(n-1)u\big)} \left(
\frac{1+n(1-u)-w}{2w+(n-1)u} +\frac{n(1-u)-w}{w} \right)  \cdot
(2w+nu-\xi)_+ \wedge (w+u) .
\end{split}
\end{equation*}

(III) $\ n(Q-q) \leqslant q_k <Q+n(Q-q)$. Upon
$\lambda_{0,2}<\gamma_1 =t_{-1}\leqslant
\lambda_{0,1}=\frac{a^\prime-\eps}{q^\prime}$ we see that
$I_{\gamma,k}\cap(\lambda_{k,n+1},\lambda_{k,n}] =
(\lambda_{k,n+1},\gamma_{k+n}] \cup (\gamma_{k+n},\lambda_{k,n}]$
when $k\geqslant 1$ and when $k=0$ and $n\geqslant 2$. When $k=0$
and $n=1$ this interval coincides with $(\lambda_{0,2},\gamma_1]$.
The cumulative contribution is thus
\begin{equation*}
\begin{split}
& \G^{(2.4.2)}_{I,Q}(\xi) = \sum_{n=1}^\infty \sum_{k=1}^\infty\
\sideset{}{^*}\sum_{\substack{q_k\in\II_{q,k} \\
n(Q-q) \leqslant q_k < Q+n(Q-q)}} \left(
\int_{\lambda_{k,n+1}}^{\gamma_{k+n}}
\frac{W^{(1)}_{\gamma,k,n}(t)\, dt}{t^2+t+1}
+\int_{\gamma_{k+n}}^{\lambda_{k,n}} \frac{W^{(2)}_{\gamma,k,n}(t)\ dt}{t^2+t+1}\right) \\
& \qquad +\sum_{n=2}^\infty \
\sideset{}{^*}\sum_{\substack{q^\prime \in \II_{q,0} \\
q^\prime \geqslant n(Q-q)}} \left(
\int_{\lambda_{0,n+1}}^{\gamma_{n}} \frac{W^{(1)}_{\gamma,0,n}(t)\,
dt}{t^2+t+1} +\int_{\gamma_{n}}^{\lambda_{0,n}}
\frac{W^{(2)}_{\gamma,0,n}(t)\, dt}{t^2+t+1} \right)
+\sideset{}{^*}\sum_{q^\prime \in \II_{q,0}}
\int_{\lambda_{0,2}}^{\gamma_1} \frac{W^{(1)}_{\gamma,0,1}(t)\, dt}{t^2+t+1} \\
& \quad = \G_{I,Q}^{(2.4.2.1)}(\xi) + \G_{I,Q}^{(2.4.2.2)}(\xi)
+ \G_{I,Q}^{(2.4.2.3)}(\xi) .
\end{split}
\end{equation*}

For fixed $k$ we have $\frac{q_k-Q}{Q-q} <n\leqslant
\frac{q_k}{Q-q}$, so $n$ can take at most $1+\Big\lfloor
\frac{Q}{Q-q}\Big\rfloor \leqslant \frac{2Q}{Q-q}$ values.
Employing \eqref{A2.12}-\eqref{A2.15} we find
\begin{equation}\label{7.9}
\G^{(2.4.2.1)}_{I,Q} (\xi) \approxeq \frac{c_I}{8\zeta(2)}
\sum_{n=1}^\infty \int_0^1 du
\int_{n(1-u)\vee 1}^{n(1-u)+1} dw \, F_n^{(2.4.2)}(\xi;u,w),
\end{equation}
where
\begin{equation*}
\begin{split}
& F_n^{(2.4.2)}(\xi;u,w) =\frac{2-u}{\big( 2w+(n-1)u\big)
(2w+nu)} \\ & \qquad \cdot \left( \frac{1+(n+1)(1-u)-w}{2w+nu}
+\frac{1+n(1-u)-w}{2w+(n-1)u} \right) \cdot
\big( 2w+(n+1)u-\xi \big)_+ \wedge (w+u) \\
&  + \frac{2-u}{\big( 2w+(n-1)u\big) (2w+nu)} \left(
\frac{w-(n-1)(1-u)}{2w+(n-1)u} + \frac{ w-n(1-u)}{ 2w+nu}\right)
 \cdot (2w+nu-\xi)_+ \wedge w
\\ & + \frac{\big( w -n(1-u) \big)^2}{(w+nu)(2w+nu)^2} \cdot \Big(
\big( 2w+(n+1)u-\xi\big)_+ \wedge w -(2w+nu -\xi)_+ \wedge w \Big) \\
& - \frac{\big( 1+n(1-u)-w\big)^2}{(w+nu)\big(
2w+(n-1)u\big)^2}\cdot \Big( (2w+ (n+1) u-\xi)_+ \wedge (w+u)
-(2w+nu-\xi)_+ \wedge (w+u)\Big) .
\end{split}
\end{equation*}
Employing \eqref{A2.13}-\eqref{A2.16} we find
\begin{equation}\label{7.10}
\G_{I,Q}^{(2.4.2.2)}(\xi) \approxeq \frac{c_I}{8\zeta(2)}
\sum_{n=2}^\infty \int_0^1 du \int_{n(1-u)\wedge 1}^1 dw\,
F^{(2.4.2)}_n (\xi;u,w) .
\end{equation}
From \eqref{7.9} and \eqref{7.10} we gather
\begin{equation}\label{7.11}
\begin{split}
\G_{I,Q}^{(2.4.2.1)}(\xi)+ \G_{I,Q}^{(2.4.2.2)}(\xi) \approxeq &
\frac{c_I}{8\zeta(2)} \int_0^1 du \int_1^{2-u} dw\ F_1^{(2.4.2)}(\xi;u,w) \\
& + \frac{c_I}{8\zeta(2)} \sum_{n=1}^\infty \int_{n(1-u)}^{n(1-u)+1}
dw\ F_n^{(2.4.2)} (\xi; u,w).
\end{split}
\end{equation}
Employing \eqref{A2.17}, \eqref{A2.18}, \eqref{A2.19} we find
\begin{equation}\label{7.12}
\G_{I,Q}^{(2.4.2.3)}(\xi) \approxeq \frac{c_I}{8\zeta(2)}
\int_0^1 du \int_{1-u}^1 dw\, F^{(2.4.2.3)} (\xi;u,w) ,
\end{equation}
where
\begin{equation*}
\begin{split}
& F^{(2.4.2.3)} (\xi;u,w) = \frac{w+u-1}{(w+u)(2w+u)} \left(
\frac{3-2u-w}{2w+u} + \frac{2-u-w}{w+u} \right) \cdot (2w+2u-\xi)_+ \wedge (w+u) \\
& \qquad + \frac{(w+u-1)^2}{(w+u)^2 (2w+u)} \cdot (2w+u-\xi)_+
\wedge w + \frac{(w+u-1)^2}{(w+u)(2w+u)^2} \cdot (2w+2u-\xi)_+
\wedge w .
\end{split}
\end{equation*}

(IV) $\ Q+n(Q-q) \leqslant q_k <Q+(n+1)(Q-q)$. Then
$k\geqslant 1$ and $(n+2)Q-q_{k+n+1} \leqslant Q-q$. In this case
$I_{\gamma,k} \cap (\lambda_{k,n+1},\lambda_{k,n}]
=(\lambda_{k,n+1},t_{k-1}]$ and $t_{k-1}<\gamma_{k+n}$, so the
corresponding contribution is
\begin{equation*}
\G_{I,Q}^{(2.4.3)} (\xi)=\sum_{n=0}^\infty \sum_{k=1}^\infty \
\sideset{}{^*}\sum_{\substack{q_k\in\II_{q,k} \\
Q+n(Q-q) \leqslant q_k < Q+(n+1)(Q-q)}}
\int_{\lambda_{k,n+1}}^{t_{k-1}} \frac{W^{(1)}_{\gamma,k,n}(t)\
dt}{t^2+t+1} .
\end{equation*}
Note also that $n$ can only take the value $n=\left\lfloor
\frac{q_k-Q}{Q-q}\right\rfloor$ for each $k$. Employing
\eqref{A2.20}, \eqref{A2.21}, \eqref{A2.22} we find
\begin{equation}\label{7.13}
\G_{I,Q}^{(2.4.3)} (\xi) \approxeq \frac{c_I}{8\zeta(2)}
\sum_{n=0}^\infty \int_0^1 du \int_{n(1-u)+1}^{(n+1)(1-u)+1} dw\,
F_n^{(2.4.3)} (\xi;u,w),
\end{equation}
where
\begin{equation*}
\begin{split}
& F_n^{(2.4.3)} (\xi;u,w) = \frac{\big( 1 +
(n+1)(1-u)-w\big)^2}{(w-u)(2w+nu)^2} \cdot \big( 2w+(n+1)u-\xi\big)_+ \wedge (w+u) \\
&  +\frac{\big(1+(n+1)(1-u)-w\big)^2}{(w-u)^2 (2w+nu)}
\cdot (2w+nu-\xi)_+ \wedge w \\
&  + \frac{1+(n+1)(1-u)-w}{(w-u)(2w+nu)} \left(
\frac{w-n(1-u)}{2w+nu} + \frac{w-n(1-u)-1}{w-u}\right)
\cdot \big( 2w+(n+1)u-\xi\big)_+ \wedge w .
\end{split}
\end{equation*}

(V) $\ q_k \geqslant Q+(n+1)(Q-q)$. Then $k\geqslant 1$,
$\lambda_{k,n+1} \geqslant t_{k-1}$ and $I_{\gamma,k}\cap
(\lambda_{k,n+1},\lambda_{k,n}]=\emptyset$.

\section{Channels with removed slits. The case $r_k=0$}\label{Sect8}
The main term and the error term of the $\CC_O$ contribution of
$(r,r_k)=(1,0)$ and of $(r,r_k)=(-1,0)$ coincide because the
corresponding weights are given by the same formulas and one only
has to replace the summation condition $q-a\equiv 1 \pmod{3}$ by
$q-a\equiv -1 \pmod{3}$. The same thing holds for the
$\CC_\leftarrow$ contribution of $(r,r_k)=(1,0)$ and the
$\CC_\downarrow$ contribution of $(r,r_k)=(-1,0)$ (see Figure
\ref{Figure8}). So it suffices to take $(r,r_k)=(0,1)$ in the
sequel, doubling the $\CC_O$ and the $\CC_\leftarrow$
contributions. This time we consider
\begin{equation*}
\sideset{}{^*}\sum =\sideset{}{^*}\sum\limits_{\substack{\gamma\in\FF_I (Q) \\
q-a\equiv 1\hspace{-4pt} \pmod{3} \\ q_k \equiv a_k
\hspace{-4pt}\pmod{3}}}.
\end{equation*}

\subsection{The channel $\CC_O$}\label{Subsec8.1}
The slit $q_k$ is removed, while $q+n q_k$, $n \geqslant 0$, are
not, as shown in Figure \ref{Figure14}. The
following three cases arise:

\begin{figure}[ht]
\centering
\unitlength 0.38mm
\begin{picture}(247,200)(-10,-7)
\texture{ccccccc 0000}
\shade\path(0,135)(70,135)(70,155)(0,155)(0,135)
\shade\path(120,110)(120,135)(170,135)(170,110)(120,110)

\put(-17,70){\makebox(0,0){\small $(0,-\eps)$}}
\put(-14,159){\makebox(0,0){\small $(0,\eps)$}}

\put(30,-2){\makebox(0,0){\small $q$}}
\put(70,187){\makebox(0,0){\small $q_k$}}
\put(140,187){\makebox(0,0){\small $2q_k$}}
\put(210,187){\makebox(0,0){\small $3q_k$}}
\put(120,173){\makebox(0,0){\small $\xi Q$}}
\put(100,22){\makebox(0,0){\small $q_{k+1}$}}
\put(170,47){\makebox(0,0){\small $q+2q_k$}}
\put(240,72){\makebox(0,0){\small $q+3q_k$}}

\put(-10,80){\makebox(0,0){\small $\AA_0$}}
\put(-10,92){\makebox(0,0){\small $\CC_k$}}
\put(-10,130){\makebox(0,0){\small $\BB_k$}}

\put(117,92){\makebox(0,0){\small $w_{\CC_k}$}}
\put(117,103){\makebox(0,0){\small $w_{\AA_0}$}}
\put(-53,120){\makebox(0,0){\small $w_{\BB_k}+w_{\CC_k}$}}
\put(200,122){\makebox(0,0){\small $w_{\AA_0}+w_{\CC_k}$}}
\put(270,147){\makebox(0,0){\small $w_{\AA_0}+w_{\CC_k}$}}

\path(30,85)(100,85) \path(100,110)(170,110)
\path(170,135)(240,135)

\dottedline{2}(0,75)(0,155) \dottedline{2}(70,100)(70,180)
\dottedline{2}(140,125)(140,180) \dottedline{2}(210,150)(210,180)

\thicklines \path(0,75)(30,75) \path(30,85)(100,85)
\path(100,100)(70,100) \path(0,155)(70,155)

\Thicklines \path(30,5)(30,85)\path(100,30)(100,110)
\path(170,55)(170,135) \path(240,80)(240,160)

\thinlines \path(0,85)(30,85) \path(0,100)(100,100)
\path(0,110)(100,110) \path(0,125)(170,125) \path(0,135)(170,135)
\path(0,150)(240,150) \path(0,155)(240,155)

\path(105,85)(105,110)
\path(103.5,88)(105,85)(106.5,88)
\path(103.5,97)(105,100)(106.5,97)
\path(103.5,103)(105,100)(106.5,103)
\path(103.5,107)(105,110)(106.5,107)

\thinlines \path(-30,85)(-30,155)
\path(-31.5,88)(-30,85)(-28.5,88)
\path(-31.5,152)(-30,155)(-28.5,152)
\path(120,110)(120,165)

\path(175,110)(175,135) \path(173.5,113)(175,110)(176.5,113)
\path(173.5,132)(175,135)(176.5,132)

\path(245,135)(245,160) \path(243.5,138)(245,135)(246.5,138)
\path(243.5,157)(245,160)(246.5,157)

\end{picture}
\caption{The case $t \in I_{\gamma,k}$, $r_k=0$, $\CC_{O}$,
$w_{\AA_0}<w_{\BB_k}$, $n_0=2$, $N=1$} \label{Figure14}
\end{figure}

\subsubsection{$w_{\AA_0}\geqslant w_{\BB_k}\
(\Longleftrightarrow t \leqslant
\gamma_{k+1})$}\label{Subsub8.1.1}

The channel $\BB_k$ is locked by the slit $q_{k+1}$ and
$W_{\gamma,k}(t) =w_{\BB_k}(t) \cdot (q_{k+1}-\xi Q)_+ \wedge
q_k$, with contribution
\begin{equation*}
\G_{I,Q}^{(3.1)}(\xi)=\sum_{k=0}^\infty
\sideset{}{^*}\sum_{\substack{q_k\in \II_{q,k} \\ q_{k+1}\leqslant
2Q}} \int_{t_k}^{\gamma_{k+1}} \frac{W_{\gamma,k} (t)\,
dt}{t^2+t+1} + \sum_{k=1}^\infty \sideset{}{^*}\sum_{\substack{q_k\in\II_{q,k} \\
q_{k+1} > 2Q}} \int_{t_k}^{t_{k-1}} \frac{W_{\gamma,k}(t)\,
dt}{t^2+t+1} .
\end{equation*}
Employing \eqref{A1.8} (which also holds for $k=0$) and \eqref{5.6} we find
\begin{equation}\label{8.1}
\G_{I,Q}^{(3.1)}(\xi) \approxeq \frac{c_I}{8\zeta(2)} \int_0^1
du \left( \int_{1-u}^{2-u} dw\, F^{(3.1.1)}(\xi;u,w) +
\int_{2-u}^\infty dw\, F^{(3.1.2)}(\xi;u,w) \right),
\end{equation}
where
\begin{equation*}
\begin{split}
F^{(3.1.1)}(\xi;u,w) & =\frac{(w+u-1)^2}{(w+u)^2w} \cdot
(w+u-\xi)_+ \wedge w,\\
F^{(3.1.2)}(\xi;u,v) & =\frac{(1-u)^2}{(w-u)^2 w} \cdot
(w+u-\xi)_+ \wedge w.
\end{split}
\end{equation*}

\subsubsection{$w_{\AA_0}<w_{\BB_k}\ (\Longleftrightarrow\
t >\gamma_{k+1})$ and $\xi Q<q_{k+1}$}\label{Subsub8.1.2}

In this case $k\geqslant 1$, $q_{k+1}\leqslant 2Q$, and
\begin{equation*}
W_{\gamma,k} (t)=w_{\AA_0}(t) \cdot (q_{k+1}-\xi Q)_+ \wedge
q_k +\big( w_{\BB_k}(t) -w_{\AA_0}(t)\big)  q_k ,
\end{equation*}
with contribution
\begin{equation*}
\G_{I,Q}^{(3.2)}(\xi) =\sum_{k=1}^\infty
\sideset{}{^*}\sum_{\substack{q_k\in \II_{q,k}
\\ \xi Q < q_{k+1} \leqslant 2Q}} \int_{\gamma_{k+1}}^{t_{k-1}}
\frac{W_{\gamma,k} (t)\, dt}{t^2+t+1}.
\end{equation*}
Employing \eqref{A2.1} and \eqref{A3.2} we find
\begin{equation}\label{8.2}
\G_{I,Q}^{(3.2)}(\xi) \approxeq \frac{c_I}{8\zeta(2)} \int_0^1
du \int_{(1 \vee (\xi-u)) \wedge (2-u)}^{2-u} dw\,
F^{(3.2)}(\xi;u,w),
\end{equation}
where
\begin{equation*}
F^{(3.2)}(\xi;u,w) = \frac{2-u-w}{(w-u)(w+u)} \left(
\frac{w+u-1}{w+u} + \frac{w-1}{w-u} \right) \cdot (w+u-\xi)_+
\wedge w + \frac{(2-u-w)^2 w}{(w-u)^2 (w+u)} .
\end{equation*}

\subsubsection{$w_{\AA_0}<w_{\BB_k} \ (\Longleftrightarrow
t > q_{k+1})$ and $q_{k+1}\leqslant \xi Q$}\label{Subsub8.1.3}

Again $k\geqslant 1$ and $q_{k+1}\leqslant 2Q$. Consider the
integer $N$ for which $q+Nq_k \leqslant \xi Q <q+(N+1)q_k$,
that is $1\leqslant N=\Big\lfloor \frac{\xi
Q-q}{q_k}\Big\rfloor \leqslant \xi$. Consider also
$n_0:=\Big\lfloor \frac{w_{\BB_k}+w_{\CC_k}}{w_{\AA_0}+w_{\CC_k}}
\Big\rfloor =\Big\lfloor \frac{qt -a}{a_k -q_k t} \Big\rfloor
\geqslant 1$, and let
\begin{equation}\label{8.3}
\lambda_{k,n} := \frac{a+na_k}{q+nq_k} \nearrow \gamma_k \quad
\mbox{\rm as $n\rightarrow\infty$,}
\end{equation}
hence $\lambda_{k,n_0} \leqslant t <\lambda_{k,n_0+1}$ and
$T(q+(n_0+1)q_k)<0\leqslant T(q+n_0 q_k)$. The channel $\BB_k$ is
locked exactly by the slits $q_{k+1}=q+q_k,\ldots ,q+n_0 q_k,
q+(n_0+1)q_k$ (see Figure \ref{Figure14}) and $W_{\gamma,k}(t)$ is given by
\begin{equation*}
\begin{cases} 0 & \mbox{\rm if $N>n_0,$} \\
\big( w_{\BB_k}+w_{\CC_k} -n_0 (w_{\AA_0}+w_{\CC_k})\big) \big(
q+(n_0+1)q_k -\xi Q\big) & \mbox{\rm if $n_0=N,$} \\
(w_{\AA_0}+w_{\CC_k}) \big(q+(N+1)q_k -\xi Q\big)+\big(
w_{\BB_k} +w_{\CC_k} -(N+1)(w_{\AA_0}+w_{\CC_k}) \big)
q_k & \mbox{\rm if $N<n_0,$} \end{cases}
\end{equation*}
or equivalently by
\begin{equation}\label{8.4}
\begin{cases} 0 & \mbox{\rm if $t \in I_{\gamma,k}
\cap (t_k,\lambda_{k,N}],$} \\
W^{(1)}_{\gamma,k,N}(t):= \big( q_{k+1}+Nq_k -\xi Q\big) \big(
(q+Nq_k)t -a-Na_k\big) & \mbox{\rm if $t \in
I_{\gamma,k} \cap (\lambda_{k,N} ,\lambda_{k,N+1}],$} \\
W^{(2)}_{\gamma,k}(t): =1-\xi Q \big( w_{\AA_0}(t) +
w_{\CC_k}(t)\big) & \mbox{\rm if $t\in I_{\gamma,k} \cap
(\lambda_{k,N+1},t_{k-1}].$}
\end{cases}
\end{equation}

We shall start by keeping $N\geqslant 1$ and $q\leqslant Q$ fixed, and
summing over
\begin{equation*}
y=q_k\in\JJ_{q,N}:=\left( \frac{\xi Q-q}{N+1},\frac{\xi Q-q}{N}\right].
\end{equation*}
Since $\lambda_{k,1}=\gamma_{k+1} >t_k$ and $\gamma_k >t_{k-1}$
for all $k$, nonzero contribution only arises from one of the
following two subcases:

(I) $\ N(q_k -Q)\leqslant Q-q<(N+1)(q_k -Q).$ Then $t_k
<\gamma_{k+1}\leqslant \lambda_{k,N} \leqslant t_{k-1}
<\lambda_{k,N+1}$ and the contribution is
\begin{equation*}
\G_{I,Q}^{(3.3.1)}(\xi) =\sum_{1\leqslant N\leqslant \xi}
\sum_{k=1}^\infty \sideset{}{^*}\sum_{\substack{q_k\in \II_{q,k}
\cap \JJ_{q,N} \\ Q+\frac{Q-q}{N+1} <q_k \leqslant
Q+\frac{Q-q}{N}}} \int_{\lambda_{k,N}}^{t_{k-1}}
\frac{W^{(1)}_{\gamma,k,N}(t)\, dt}{t^2+t+1} .
\end{equation*}
In this case $N$ can only take the value $N=\left\lfloor
\frac{Q-q}{q_k -Q} \right\rfloor$ for fixed $k$. Employing
\eqref{A3.3} we find
\begin{equation}\label{8.5}
\G_{I,Q}^{(3.3.1)}(\xi) \approxeq \frac{c_I}{8\zeta(2)}
\sum_{1\leqslant N\leqslant \xi} \int_0^1 du
\int_{\left[1+\frac{1-u}{N+1},1+\frac{1-u}{N}\right] \cap \left[
\frac{\xi-u}{N+1},\frac{\xi-u}{N}\right]} dw\,
F_N^{(3.3.1)} (\xi;u,w),
\end{equation}
where
\begin{equation*}
F_N^{(3.3.1)}(\xi;u,w) =\frac{\big( 1-u-N(w-1) \big)^2 \big(
(N+1)w+u-\xi\big) }{(w-u)^2 (Nw+u)} .
\end{equation*}

(II) $\ (N+1)(q_k-Q) \leqslant Q-q$. Then $t_k
<\gamma_{k+1}\leqslant \lambda_{k,N} <\lambda_{k,N+1} \leqslant
t_{k-1}$. The contribution
\begin{equation*}
\G_{I,Q}^{(3.3.2)}(\xi) =\sum_{1\leqslant N\leqslant \xi}
\sum_{k=1}^\infty
\sideset{}{^*}\sum_{\substack{q_k\in\II_{q,k}\cap \JJ_{q,N} \\ q_k
\leqslant Q+\frac{Q-q}{N+1}}} \left(
\int_{\lambda_{k,N}}^{\lambda_{k,N+1}}
\frac{W^{(1)}_{\gamma,k,N}(t)\, dt}{t^2+t+1} +
\int_{\lambda_{k,N+1}}^{t_{k-1}} \frac{W^{(2)}_{\gamma,k}(t)\,
dt}{t^2+t+1} \right)
\end{equation*}
is estimated upon \eqref{A3.4} and \eqref{A3.7} as
\begin{equation}\label{8.6}
\G_{I,Q}^{(3.3.2)}(\xi) \approxeq \frac{c_I}{8\zeta(2)}
\sum_{1\leqslant N\leqslant \xi} \int_0^1 du \int_{\left[ 1
,1+\frac{1-u}{N+1} \right] \cap \left[
\frac{\xi-u}{N+1},\frac{\xi-u}{N} \right]} dw\,
F_N^{(3.3.2)}(\xi;u,w),
\end{equation}
where \begin{equation*}
\begin{split}
F_N^{(3.3.2)}(\xi;u,w) = & \frac{(N+1)w+u-\xi}{(Nw+u)\big(
(N+1)w+u\big)^2} \\ & + \frac{1-u-(N+1)(w-1)}{(w-u)\big((N+1)w+u
\big)} \left( 2-\frac{\xi}{(N+1)w+u} -
\frac{\xi (w-1)}{w-u} \right) .
\end{split}
\end{equation*}
In this case $N\in \{\lfloor \xi\rfloor -1,\lfloor
\xi\rfloor \}$ suffices as a result of the inequality $\xi
- u < N+2-u$.

\subsection{The channel $\CC_{\leftarrow}$}\label{Subsec8.2}

In this case all slits $q+nq_k$ are removed, while slits
$2q+nq_k$, $n\geqslant 0$, are not. Exactly two of the later ones
lock $\AA_0 \cup \CC_k$. Letting $n_0:=\Big\lfloor
\frac{w_{\BB_k}+2w_{\CC_k}}{w_{\AA_0}+w_{\CC_k}} \Big\rfloor=
\Big\lfloor \frac{a_{k-2}-2\eps -q_{k-2}t}{a_k -q_k t} \Big\rfloor
\geqslant 0$, we have $w_{\BB_k}\leqslant
T(2q+n_0q_k)=2(w_{\BB_k}+w_{\CC_k})-n_0
(w_{\AA_0}+w_{\CC_k})<2\eps$, so $\AA_0 \cup \CC_k$ is locked
exactly by the slits $2q+n_0 q_k$ and $2q+(n_0+1)q_k$. The
situation is described in Figure \ref{Figure15}. Consider also
\begin{equation}\label{8.7}
\mu_{k,n}:=\frac{2a+na_k+2\eps}{2q+nq_k} \quad \mbox{\rm and}
\quad \nu_{k,n}:=\frac{a+na_k}{q+nq_k}.
\end{equation}

When $q_{k} \leqslant 2Q$, $\mu_{k,n}\nearrow \gamma_k
> t_{k-1}$ as $n\rightarrow \infty$ and $\mu_{k,0}
=\frac{a+\eps}{q} \leqslant t_{k-1}$, $\forall k\geqslant 0$. We
also have $\mu_{k,n_0-1}\leqslant t <\mu_{k,n_0}$ and the
following two situations can arise:

(I) $\ w_{\BB_k} \leqslant T(2q+n_0q_k)<w_{\BB_k}+w_{\CC_k}\
(\Longleftrightarrow  \mu_{k,n_0-1} \leqslant t <\nu_{k,n_0}).$
Then $n_0 \geqslant 1$ and the widths of the relevant three
sub-channels of $\AA_0 \cup \BB_k$ are (from bottom to top)
$w_{\AA_0}$, $w_{\BB_k}+w_{\CC_k} -T(2q+n_0q_k)= a+n_0 q_k-(q+n_0
q_k)t$, $T(2q+n_0 q_k)-w_{\BB_k}=w_{\CC_k} -\big(a+n_0a_k
-(q+n_0q_k)t\big)$, yielding
\begin{equation*}
\begin{split}
& W_{\gamma,k}(t) = W_{\gamma,k,n_0}^{(1)} (t) = w_{\AA_0}(t) \cdot (2q+n_0 q_k
-\xi Q)_+ \wedge q +w_{\CC_k}(t) \cdot ( q_{k+2}+n_0q_k -\xi Q)_+ \wedge q_{k+1} \\
&  \quad -\big( a+n_0 a_k-(q+n_0 q_k)t)\big) \cdot \Big( (
q_{k+2}+n_0q_k -\xi Q)_+ \wedge q_{k+1} - (2q+n_0q_k-\xi
Q)_+\wedge q_{k+1} \Big) .
\end{split}
\end{equation*}

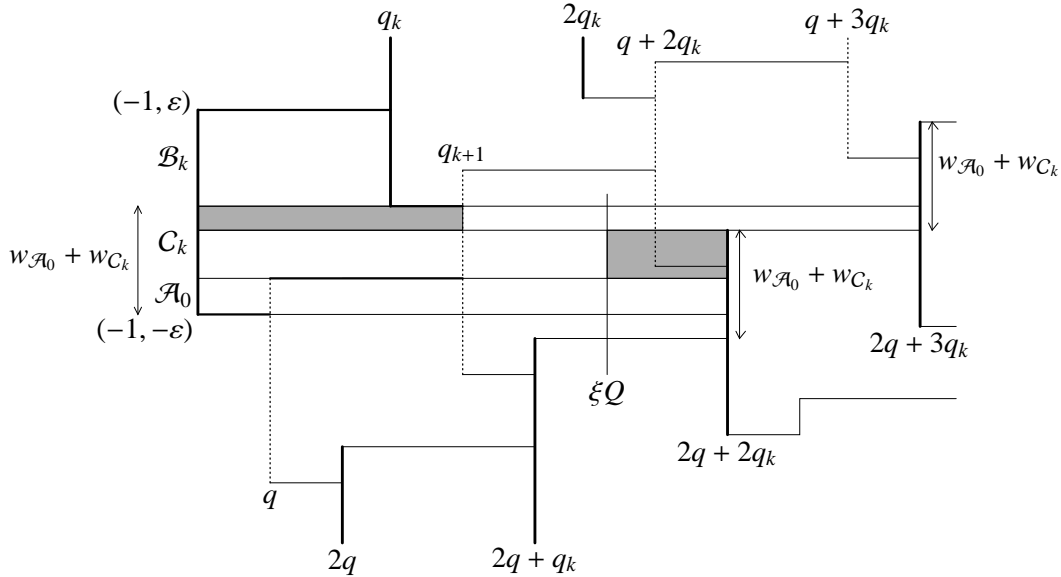
\begin{figure}[ht]
\centering
\unitlength 0.32mm
\begin{picture}(310,220)(-10,0)
\texture{ccccccc 0000}
\shade\path(0,130)(0,140)(110,140)(110,130)(0,130)
\shade\path(220,130)(220,110)(170,110)(170,130)(220,130)

\put(-22,88){\makebox(0,0){\small $(-1,-\eps)$}}
\put(-19,183){\makebox(0,0){\small $(-1,\eps)$}}

\put(30,18){\makebox(0,0){\small $q$}}
\put(60,-8){\makebox(0,0){\small $2q$}}
\put(80,217){\makebox(0,0){\small $q_k$}}
\put(110,162){\makebox(0,0){\small $q_{k+1}$}}
\put(140,-7){\makebox(0,0){\small $2q+q_{k}$}}
\put(170,62){\makebox(0,0){\small $\xi Q$}}
\put(220,37){\makebox(0,0){\small $2q+2q_{k}$}}
\put(300,82){\makebox(0,0){\small $2q+3q_{k}$}}
\put(160,219){\makebox(0,0){\small $2q_k$}}
\put(270,217){\makebox(0,0){\small $q+3q_{k}$}}
\put(192,208){\makebox(0,0){\small $q+2q_{k}$}}

\put(-10,103){\makebox(0,0){\small $\AA_0$}}
\put(-10,125){\makebox(0,0){\small $\CC_k$}}
\put(-10,160){\makebox(0,0){\small $\BB_k$}}

\put(256,110){\makebox(0,0){\small $w_{\AA_0}+w_{\CC_k}$}}
\put(-53,118){\makebox(0,0){\small $w_{\AA_0}+w_{\CC_k}$}}
\put(333,157){\makebox(0,0){\small $w_{\AA_0}+w_{\CC_k}$}}

\path(30,95)(220,95) \path(220,130)(300,130)
\path(110,140)(300,140)

\dottedline{2}(30,25)(30,110) \dottedline{2}(110,70)(110,155)
\dottedline{2}(190,115)(190,200) \dottedline{2}(270,160)(270,210)

\thicklines \path(0,95)(30,95) \path(30,110)(110,110)
\path(110,140)(80,140) \path(0,180)(80,180)

\Thicklines \path(0,95)(0,180) \path(60,0)(60,40)
\path(140,85)(140,0) \path(220,130)(220,45) \path(300,90)(300,175)
\path(80,140)(80,210)\path(160,185)(160,210)

\thinlines \path(220,45)(250,45)(250,60)(315,60)
\path(300,90)(315,90) \path(300,175)(315,175)

\path(0,110)(30,110) \path(60,110)(220,110)
\path(190,115)(220,115) \path(0,130)(220,130) \path(0,140)(80,140)
\path(225,130)(225,85) \path(223.5,88)(225,85)(227.5,88)
\path(223.5,127)(225,130)(227.5,127) \path(140,85)(220,85)
\path(110,70)(140,70) \path(60,40)(140,40) \path(30,25)(60,25)

\path(170,145)(170,70) \path(110,155)(190,155)
\path(190,185)(160,185)\path(190,200)(270,200)
\path(270,160)(300,160) \path(-25,95)(-25,140)
\path(-26.5,98)(-25,95)(-23.5,98)
\path(-26.5,137)(-25,140)(-23.5,137)

\path(305,175)(305,130) \path(303.5,133)(305,130)(307.5,133)
\path(303.5,172)(305,175)(307.5,172)

\end{picture}
\caption{The case $t \in I_{\gamma,k}$, $r=1$, $r_k=0$,
$r_{k+1}=-1$, $\CC_{\leftarrow}$, $n_0=2$} \label{Figure15}
\end{figure}

(II) $\ w_{\BB_k}+w_{\CC_k} \leqslant T(2q+n_0q_k)<2\eps\
(\Longleftrightarrow \nu_{k,n_0} \leqslant t <\mu_{k,n_0}).$ Then
$n_0 \geqslant 0$ and the widths of the relevant three
sub-channels of $\AA_0 \cup \CC_k$ are: $2\eps -T(2q+n_0 q_k) =
w_{\AA_0}-\big((q+n_0 q_k) t-a-n_0 q_k \big)$, $T(2q+n_0
q_k)-(w_{\BB_k}+w_{\CC_k})=(q+n_0 q_k)t-a-n_0 a_k$, $w_{\CC_k}$,
yielding
\begin{equation*}
\begin{split}
W_{\gamma,k}(t) & = W_{\gamma,k,n_0}^{(2)} (t) = w_{\AA_0}(t)\cdot (2q+n_0 q_k
-\xi Q)_+ \wedge q +w_{\CC_k}(t) \cdot ( q_{k+2} + n_0 q_k
-\xi Q \big)_+ \wedge q_{k+1} \\ & \quad + \big( (q+n_0 q_k) t -a-
n_0 a_k \big) \cdot \Big( \big( q_{k+2}+n_0 q_k -\xi Q \big)_+
\wedge q -(2q+n_0 q_k -\xi Q)_+ \wedge q\Big) .
\end{split}
\end{equation*}
The following three cases arise:

\subsubsection{$n_0=0 \ (\Longleftrightarrow t
< \frac{a+\eps}{q})$}\label{Subsub8.2.1}

In this case $\AA_0$ is locked by the slits $2q$ and $q_{k+2}$.
One sees that $W_{\gamma,k}(t)=W_{\gamma,k,0}^{(2)}(t)$. When
$q_{k+2}\leqslant 2Q$ we have $\frac{a+\eps}{q} <t_k$ with zero
contribution, so we must take $q_{k+1}>2Q-q$. When $q_{k+1}>2Q$ we
have $k\geqslant 1$ and $\frac{a+\eps}{q}> t_{k-1}$, with contribution
\begin{equation*}
\G_{I,Q}^{(3.4.1)}(\xi)=\sum_{k=1}^\infty
\sideset{}{^*}\sum_{\substack{q_k\in\II_{q,k} \\ q_{k+1}>2Q}}
\int_{t_k}^{t_{k-1}} \frac{W_{\gamma,k,0}^{(2)}(t)\, dt}{t^2+t+1} .
\end{equation*}
Employing \eqref{5.6}, \eqref{5.7}, \eqref{5.9} and $qt - a =
w_{\BB_k} (t) +w_{\CC_k} (t)$ we find
\begin{equation}\label{8.8}
\G_{I,Q}^{(3.4.1)}(\xi) \approxeq \frac{c_I}{8\zeta (2)}
\int_0^1 du \int_{2-u}^\infty dw\, F^{(3.4.1)}(\xi;u,w),
\end{equation}
where
\begin{equation*}
\begin{split}
& F^{(3.4.1)} (\xi;u,w) = \frac{(1-u)^2(2w-u)}{(w-u)^2 w^2} \cdot
\Big( (w+2u-\xi)_+ \wedge u -(2u-\xi)_+ \wedge u \Big) \\
&  +\frac{1-u}{(w-u)w} \left(
\frac{w+u-1}{w}+\frac{w-1}{w-u}\right) \cdot (2u-\xi)_+ \wedge
u +\frac{(1-u)^2}{(w-u)w^2} \cdot (w+2u-\xi)_+ \wedge (w+u).
\end{split}
\end{equation*}
When $2Q-q<q_{k+1}\leqslant 2Q$ we have
$t_k<\frac{a+\eps}{q}\leqslant t_{k-1}$ and $q_k$ takes exactly
one value. The corresponding contribution
\begin{equation*}
\G_{I,Q}^{(3.4.2)}(\xi)=\sum_{k=0}^\infty
\sideset{}{^*}\sum_{\substack{q_k\in\II_{q,k} \\ 2Q-q <q_{k+1}
\leqslant 2Q}} \int_{t_k}^{\frac{a+\eps}{q}}
\frac{W^{(2)}_{\gamma,k,0}(t)\, dt}{t^2+t+1}
\end{equation*}
is estimated upon \eqref{A1.11}, \eqref{A1.12}, \eqref{A1.15}, \eqref{A1.16} as
\begin{equation}\label{8.9}
\G_{I,Q}^{(3.4.2)} (\xi) \approxeq \frac{c_I}{8\zeta(2)}
\int_0^1 du \int_{2-2u}^{2-u} dw\, F^{(3.4.2)}(\xi;u,w) ,
\end{equation}
where
\begin{equation*}
\begin{split}
& F^{(3.4.2)} (\xi;u,w) = \frac{w+2u-2}{2uw} \left(
\frac{1-u}{w}+\frac{2-u-w}{2u}\right)
\cdot (w+2u-\xi)_+ \wedge (w+u) \\
& \qquad + \frac{(w+2u-2)(w-2u+2)}{4uw^2} \cdot (w+2u-\xi)_+
\wedge u +\frac{(w+2u-2)^2}{2uw^2}\cdot (2u-\xi)_+ \wedge u.
\end{split}
\end{equation*}

\subsubsection{$n_0 \geqslant 1$ and $k=0$}\label{Subsub8.2.2}

Then $\mu_{0,1}=\frac{2a^\prime+a+2\eps}{2q+q^\prime}
>\gamma_1=t_{-1}=\nu_{0,1} >\frac{a+\eps}{q}$, with
contribution
\begin{equation*}
\G_{I,Q}^{(3.4.3)} (\xi) = \sideset{}{^*}\sum_{\substack{q >
Q/2 \\ q^\prime > 2(Q-q)}} \int_{\frac{a+\eps}{q}}^{\gamma_1}
\frac{W^{(1)}_{\gamma,0,1}(t)\, dt}{t^2+t+1} .
\end{equation*}
Employing \eqref{A3.8}, \eqref{A3.9}, \eqref{A3.10} we find
\begin{equation}\label{8.10}
\G_{I,Q}^{(3.4.3)}(\xi) \approxeq \frac{c_I}{8\zeta(2)}
\int_{\frac{1}{2}}^1 du \int_{2-2u}^1 dw\, F^{(3.4.3)}
(\xi;u,w),
\end{equation}
where
\begin{equation*}
\begin{split}
& F^{(3.4.3)}(\xi;u,w) = \frac{(2-u-w)(3w+3u-2)}{4u (w+u)^2}
\cdot (w+2u-\xi)_+ \wedge u \\
&  +\frac{(2-u-w)^2}{2u(w+u)^2} \cdot ( 2w +2u - \xi )_+ \wedge (w+u)  +
\frac{(2-u-w)^2}{4u^2(w+u)} \cdot (w+2u-\xi)_+ \wedge (w+u)  .
\end{split}
\end{equation*}

\subsubsection{$n_0\geqslant 1$ and $k\geqslant 1$}\label{Subsub8.2.3}

Note first that $\mu_{k,n} \leqslant t_k$ when $q_k \leqslant
\frac{2(Q-q)}{n+1}$ and $t_{k-1} <\mu_{k,n-1}$ when $q_k
>Q+\frac{Q-q}{n}$. In both cases $I_{\gamma,k} \cap
[\mu_{k,n-1},\mu_{k,n}]$ has measure zero, so we shall only
consider $q_k \in \left[ \frac{2(Q-q)}{n+1} , Q + \frac{Q-q}{n}
\right]$. The following four subcases arise:

(I) $\ \frac{2(Q-q)}{n+1} < q_k \leqslant \frac{2(Q-q)}{n}
\wedge \left( Q+\frac{Q-q}{n+1}\right)$. Since $q_k
>Q$ we have $Q\leqslant \frac{2(Q-q)}{n}$, so $n=1$ and
$q\leqslant \frac{Q}{2}$. Furthermore $\mu_{k,0} < t_k <\nu_{k,1}
< \mu_{k,1} < t_{k-1}$ and the contribution is
\begin{equation*}
\G^{(3.4.4)}_{I,Q} (\xi) = \sum_{k=1}^\infty
\sideset{}{^*}\sum_{\substack{q_k \in \II_{q,k},\, q\leqslant Q/2 \\
q_k \leqslant 2(Q-q) \wedge \frac{3Q-q}{2}}} \left(
\int_{t_k}^{\nu_{k,1}} \frac{W^{(1)}_{\gamma,k,1}(t)\, dt}{t^2+t+1}
+\int_{\nu_{k,1}}^{\mu_{k,1}} \frac{W^{(2)}_{\gamma,k,1}(t)\,
dt}{t^2+t+1} \right) .
\end{equation*}
Employing \eqref{A3.11}-\eqref{A3.14} we find
\begin{equation}\label{8.11}
\begin{split}
\G^{(3.4.4)}_{I,Q}(\xi) \approxeq & \frac{c_I}{8\zeta(2)}
\int_{0}^{\frac{1}{3}} du \int_1^{\frac{3-u}{2}} dw\,
F^{(3.4.4)}(\xi;u,w) \\ & +\frac{c_I}{8\zeta(2)} \int_{\frac{1}{3}}^{\frac{1}{2}} du
\int_1^{2-2u} dw\, F^{(3.4.4)}(\xi;u,w) ,
\end{split}
\end{equation}
where
\begin{equation*}
\begin{split}
F^{(3.4.4)} & (\xi;u,w) = \frac{(w+u-1)^2(3w+2u)}{w^2
(w+u)(w+2u)} \cdot (w+2u-\xi)_+ \wedge u \\
&  + \left\{ \frac{2(w+u-1)}{w(w+2u)} \left( \frac{1-u}{w}
+\frac{3-u-2w}{w+2u}\right) -\frac{(w+u-1)^2}{w^2(w+u)} \right\} \cdot (2w+2u-\xi)_+ \wedge (w+u) \\
&  +\frac{(w+u-1)^2}{w^2(w+u)} \cdot (w+2u-\xi)_+ \wedge
(w+u) +\frac{(w+u-1)^2}{(w+u)(w+2u)^2} \cdot (2w+2u- \xi)_+
\wedge u .
\end{split}
\end{equation*}

(II) $\ \frac{2(Q-q)}{n} < q_k \leqslant Q +
\frac{Q-q}{n+1}$. Then $t_k <\mu_{k,n-1} < \nu_{k,n} <\mu_{k,n}
\leqslant t_{k-1}$, with contribution
\begin{equation*}
\G^{(3.4.5)}_{I,Q}(\xi)=\sum_{k=1}^\infty \sum_{n=1}^\infty
\sideset{}{^*}\sum_{\substack{q_k\in\II_{q,k}
\\ \frac{2(Q-q)}{n} < q_k \leqslant
Q + \frac{Q-q}{n+1}}} \left( \int_{\mu_{k,n-1}}^{\nu_{k,n}}
\frac{W^{(1)}_{\gamma,k,n}(t)\, dt}{t^2+t+1}
+\int_{\nu_{k,n}}^{\mu_{k,n}} \frac{W^{(2)}_{\gamma,k,n}(t)\,
dt}{t^2+t+1} \right) .
\end{equation*}
Employing \eqref{A3.15}-\eqref{A3.19} we find
(according to whether $n=1$ or $n\geqslant 2$)\footnote{Note that $2(Q-q)<q_k \leqslant Q+\frac{Q-q}{2}$
implies $q>\frac{Q}{3}$.}
\begin{equation}\label{8.12}
\begin{split} \G_{I,Q}^{(3.4.5)}(\xi) \approxeq &
\frac{c_I}{8\zeta(2)}\int_{\frac{1}{3}}^{\frac{1}{2}} du
\int_{2-2u}^{\frac{3-u}{2}} dw \, F^{(3.4.5)}_1 (\xi;u,w) \\ & +\frac{c_I}{8\zeta (2)}
\int_{\frac{1}{2}}^1 du \int_u^{\frac{3-u}{2}} dw\,
F^{(3.4.5)}_1 (\xi;u,w) \\ & +\frac{c_I}{8\zeta (2)} \sum_{n=2}^\infty \int_0^1 du
\int_1^{1+\frac{1-u}{n+1}}\hspace{-3pt} dw\,
F_n^{(3.4.5)}(\xi;u,w) ,
\end{split}
\end{equation}
where
{\small \begin{equation*}
\begin{split}
& F_n^{(3.4.5)}(\xi;u,w) = \frac{\big(
n(w-1)+u\big)^2}{(nw+u)(nw+2u)^2} \cdot \big(
(n+1)w+2u-\xi\big)_+ \wedge u
\\ & \qquad \qquad \qquad + \frac{\big(
1-u-n(w-1)\big)^2}{(nw+u)\big( (n-1)w+2u\big)^2} \cdot
(nw+2u-\xi)_+ \wedge (w+u) \\
& +\left\{ \frac{2-w}{(nw+2u)\big( (n-1)w+2u\big)} \left(
\frac{(n-1)(w-1)+u}{(n-1)w+2u} + \frac{n(w-1)+u}{nw+2u} \right) -
\frac{\big( n(w-1) +u\big)^2}{(nw+u)(nw+2u)^2} \right\} \\ &
\hspace{4cm} \cdot (nw+2u-\xi)_+ \wedge u \\
& + \Bigg\{ \frac{2-w}{(nw+2u)\big( (n-1)w+2u\big)} \left(
\frac{1-u-n(w-1)}{(n-1)w+2u} + \frac{1-u-(n+1)(w-1)}{nw+2u} \right)
\\ & \qquad \hspace{2cm} - \left. \frac{\big(
1-u-n(w-1)\big)^2}{(nw+u)\big( (n-1)w+2u\big)^2} \right\}
\cdot\big( (n+1)w+2u-\xi\big)_+ \wedge (w+u).
\end{split}
\end{equation*}}

(III) $\ Q+\frac{Q-q}{n+1} < q_k \leqslant \frac{2(Q-q)}{n}$.
This gives $n=1$, $\mu_{k,0} \leqslant t_k <\nu_{k,1} =\gamma_{k+1}
<t_{k-1}<\mu_{k,1}$, and
\begin{equation*}
\G_{I,Q}^{(3.4.6)}(\xi) =\sum_{k=1}^\infty
\sideset{}{^*}\sum_{\substack{q_k\in\II_{q,k} \\ \frac{3Q-q}{2} <
q_k \leqslant 2(Q-q)}} \left( \int_{t_k}^{\nu_{k,1}}
\frac{W^{(1)}_{\gamma,k,n}(t)\, dt}{t^2+t+1} +
\int_{\nu_{k,1}}^{t_{k-1}} \frac{W^{(2)}_{\gamma,k,n}(t)\,
dt}{t^2+t+1} \right) .
\end{equation*}
Employing \eqref{5.7}, \eqref{5.9}, \eqref{A1.9} and
\eqref{A3.20} we find
\begin{equation}\label{8.13}
\G_{I,Q}^{(3.4.6)}(\xi) \approxeq \frac{c_I}{8\zeta(2)}
\int_0^{\frac{1}{3}} du \int_{\frac{3-u}{2}}^{2-2u} dw\,
F^{(3.4.6)}(\xi;u,w) ,
\end{equation}
where
\begin{equation*}
\begin{split}
F^{(3.4.6)} & (\xi;u,w) = \left\{ \frac{1-u}{(w-u)w} \left(
\frac{w+u-1}{w}+\frac{w-1}{w-u}\right) -
\frac{(2-w-u)^2}{(w-u)^2(w+u)} \right\} \cdot (w+2u-\xi)_+
\wedge u \\ & + \left\{ \frac{(1-u)^2}{(w-u)w^2}
-\frac{(w+u-1)^2}{(w+u)w^2} \right\} \cdot (2w+2u-\xi)_+
\wedge (w+u) \\ & + \frac{(w+u-1)^2}{(w+u)w^2} \cdot
(w+2u-\xi)_+ \wedge (w+u) + \frac{(2-w-u)^2}{(w-u)^2 (w+u)}
\cdot (2w+2u-\xi)_+ \wedge u  .
\end{split}
\end{equation*}

(IV) $\ \frac{2(Q-q)}{n}\vee \left(Q+\frac{Q-q}{n+1}\right)
< q_k \leqslant Q+\frac{Q-q}{n}$. For fixed $k$ there is only one
value $n$ can take, namely $\Big\lfloor \frac{Q-q}{q_k
-Q}\Big\rfloor$. We have $t_k < \mu_{k,n-1} \leqslant \nu_{k,n}
\leqslant t_{k-1} < \mu_{k,n}$ and
\begin{equation*}
\G_{I,Q}^{(3.4.7)} (\xi) =\sum_{k=1}^\infty \sum_{n=1}^\infty
\sideset{}{^*}\sum_{\substack{q_k\in \II_{q,k} \\
\frac{2(Q-q)}{n} \vee \left(Q+\frac{Q-q}{n+1}\right) < q_k
\leqslant Q+\frac{Q-q}{n}}} \left( \int_{\mu_{k,n-1}}^{\nu_{k,n}}
\frac{W^{(1)}_{\gamma,k,n}(t)\, dt}{t^2+t+1}
+\int_{\nu_{k,n}}^{t_{k-1}} \frac{W^{(2)}_{\gamma,k,n}(t)\,
dt}{t^2+t+1} \right) .
\end{equation*}
Employing \eqref{A3.21}-\eqref{A3.25} we find
\begin{equation}\label{8.14}
\begin{split}
\G_{I,Q}^{(3.4.7)}(\xi) \approxeq & \frac{c_I}{8\zeta(2)}
\int_0^{\frac{1}{3}} du \int_{2-2u}^{2-u} dw\, F_1^{(3.4.7)}
(\xi;u,w) \\ & +\frac{c_I}{8\zeta(2)} \int_{\frac{1}{3}}^1 du
\int_{\frac{3-u}{2}}^{2-u} dw\, F_1^{(3.4.7)}(\xi;u,w) \\ &
 +\frac{c_I}{8\zeta(2)} \sum_{n=2}^\infty \int_0^1 du
\int_{1+\frac{1-u}{n+1}}^{1+\frac{1-u}{n}} dw\,
F_n^{(3.4.7)}(\xi;u,w) ,
\end{split}
\end{equation}
where
\begin{equation*}
\begin{split}
F_n^{(3.4.7)} & (\xi;u,w) = \frac{\big(
1-u-n(w-1)\big)^2}{(w-u)(nw+2u)\big( (n-1)w+2u\big)} \cdot \big(
(n+1)w+2u-\xi\big)_+ \wedge (w+u) \\
& + \left\{\frac{1-u-n(w-1)}{(w-u)\big( (n-1)w+2u\big)}
\left( \frac{(n-1)(w-1)+u}{(n-1)w+2u} +\frac{w-1}{w-u}\right)
-\frac{\big( 1-u-n(w-1)\big)^2}{(w-u)^2 (nw+u)} \right\} \\ & \qquad \qquad \cdot (nw+2u-\xi)_+ \wedge u
\\ &  +\frac{\big(1-u-n(w-1)\big)^2}{(nw+u)\big( (n-1)w+2u\big)^2}
\cdot (nw+2u-\xi)_+ \wedge (w+u) \\ &
+\frac{\big( 1-u-n(w-1)\big)^2}{(w-u)^2 (nw+u)}
\cdot \big( (n+1)w+2u-\xi\big)_+ \wedge u .
\end{split}
\end{equation*}

\section{Channels with removed slits. The case $r_{k+1}=0$}\label{Sect9}
In this section we consider
\begin{equation*}
\sideset{}{^*}\sum =\sum\limits_{\substack{\gamma\in \FF_I (Q) \\
q\nequiv a \hspace{-6pt}\pmod{3} \\ q_{k+1} \equiv a_{k+1}
\hspace{-6pt}\pmod{3} }}.
\end{equation*}
We shall actually sum as in Remark 1 of Section 5 over $x=q-a\in
q(1-I)$, $x\equiv \pm 1 \pmod{3}$, $y= q_{k+1} \in \II_{q,k+1}$,
$\beta =0$, $xy\equiv 1 \pmod{3q}$. The contributions arising from
$x\equiv 1 \pmod{3}$ respectively $x\equiv -1 \pmod{3}$ will have
the same main term and error.

\subsection{The channel $\CC_O$}\label{Subsec9.1}
In the formulas for $\CC_O$ and $(r,r_k)=(1,-1)$, respectively $(r,r_k)=(-1,1)$,
the corresponding main term and error coincide. Hence we only
take $(r,r_k,r_{k+1})=(1,-1,0)$ and double its $\CC_O$
contribution. The slits $q+\ell q_{k+1}$, $l\geqslant 1$, are not
removed and lock the channel $\CC_k$ (see Figure \ref{Figure16}).
The following two situations arise:

\begin{figure}[ht]
\centering
\unitlength 0.32mm
\begin{picture}(340,180)(0,0)
\texture{ccccccc 0000}
\shade\path(0,100)(0,105)(100,105)(100,100)(0,100)
\shade\path(210,90)(210,100)(230,100)(230,90)(210,90)

\put(-17,47){\makebox(0,0){\small $(0,-\eps)$}}
\put(-14,128){\makebox(0,0){\small $(0,\eps)$}}

\put(30,-2){\makebox(0,0){\small $q$}}
\put(70,167){\makebox(0,0){\small $q_k$}}
\put(210,116){\makebox(0,0){\small $\xi Q$}}
\put(100,52){\makebox(0,0){\small $q_{k+1}$}}
\put(130,7){\makebox(0,0){\small $q+q_{k+1}$}}
\put(170,167){\makebox(0,0){\small $q_k +q_{k+1}$}}
\put(197,62){\makebox(0,0){\small $2q_{k+1}$}}
\put(230,17){\makebox(0,0){\small $q+2q_{k+1}$}}
\put(300,72){\makebox(0,0){\small $3q_{k+1}$}}
\put(340,27){\makebox(0,0){\small $q+3q_{k+1}$}}
\put(270,167){\makebox(0,0){\small $q_k+2q_{k+1}$}}

\put(-10,65){\makebox(0,0){\small $\AA_0$}}
\put(-10,92){\makebox(0,0){\small $\CC_k$}}
\put(-10,115){\makebox(0,0){\small $\BB_k$}}

\put(165,83){\makebox(0,0){\small $w_{\AA_0}-w_{\BB_k}$}}
\put(-36,92){\makebox(0,0){\small $w_{\CC_k}$}}
\put(265,93){\makebox(0,0){\small $w_{\AA_0}-w_{\BB_k}$}}
\put(365,103){\makebox(0,0){\small $w_{\AA_0}-w_{\BB_k}$}}

\path(100,105)(330,105) \path(100,80)(130,80)
\path(130,90)(230,90) \path(270,100)(330,100)

\dottedline{2}(100,60)(100,135) \dottedline{2}(200,70)(200,145)
\dottedline{2}(300,80)(300,155) \dottedline{2}(0,50)(0,125)

\thicklines \path(0,125)(70,125) \path(70,105)(100,105)
\path(100,80)(30,80) \path(0,50)(30,50)

\Thicklines \path(30,5)(30,80) \path(130,15)(130,90)
\path(230,25)(230,100) \path(330,35)(330,110)
\path(170,115)(170,160) \path(70,105)(70,160)
\path(270,125)(270,160)

\thinlines \path(0,105)(70,105) \path(100,60)(130,60)
\path(0,80)(30,80) \path(0,90)(130,90) \path(0,100)(270,100)
\path(135,80)(135,90) \path(133.5,83)(135,80)(136.5,83)
\path(133.5,87)(135,90)(136.5,87) \path(-25,80)(-25,105)
\path(-26.5,83)(-25,80)(-23.5,83)
\path(-26.5,102)(-25,105)(-23.5,102)\path(210,80)(210,110)

\path(30,5)(60,5)(60,35)(130,35) \path(100,135)(170,135)
\path(170,115)(200,115) \path(200,145)(270,145)
\path(270,125)(300,125) \path(300,155)(340,155)
\path(330,110)(340,110) \path(330,35)(340,35)
\path(300,80)(330,80) \path(130,15)(160,15)(160,45)(230,45)
\path(200,70)(230,70) \path(230,25)(260,25)(260,55)(330,55)

\path(235,90)(235,100) \path(335,100)(335,110)
\path(233.5,93)(235,90)(236.5,93)
\path(233.5,97)(235,100)(236.5,97)
\path(333.5,103)(335,100)(336.5,103)
\path(333.5,107)(335,110)(336.5,107)

\end{picture}
\caption{The case $t \in I_{\gamma,k}$, $r=\pm 1$, $r_k=\mp 1$,
$r_{k+1}=0$, $\CC_{O}$, $w_{\AA_0}>w_{\BB_k}$, $N=1$, $n_0=2$}
\label{Figure16}
\end{figure}

\subsubsection{$w_{\AA_0}>w_{\BB_k}$ $(\Longleftrightarrow t
< \gamma_{k+1})$}\label{Subsub9.1.1}

We split the analysis in the following two cases:

(I) $\ \xi Q < q_{k+2}$. Then (since
$w_{\AA_0}-w_{\BB_k} > w_{\CC_k} \Longleftrightarrow t<\frac{a+\eps}{q}$) we find
\begin{equation*}
W_{\gamma,k}(t)=\begin{cases} W^{(1)}_{\gamma,k} (t) :=
w_{\CC_k}(t) \cdot (q_{k+2}-\xi Q) \wedge q_{k+1} &
\mbox{\rm if $t<\frac{a+\eps}{q},$} \\
W^{(2)}_{\gamma,k}(t) : = \big( (w_{\AA_0}(t)-w_{\BB_k}(t)\big) \cdot
(q_{k+2}-\xi Q) \wedge q_{k+1} & \\ \qquad \qquad + \big( w_{\CC_k}(t)-w_{\AA_0}(t) +w_{\BB_k}(t)\big) q_{k+1}
& \mbox{\rm if $t\geqslant \frac{a+\eps}{q}.$}
\end{cases}
\end{equation*}

$\bullet$ \ When $q_{k+1} >2Q$ we have $k\geqslant 1$ and $t_{k-1} < \gamma_{k+1} < \frac{a+\eps}{q}$.
with contribution
\begin{equation*}
\G_{I,Q}^{(4.1.1.1)}(\xi) = \sum_{k=1}^\infty
\sideset{}{^*} \sum\limits_{\substack{q_{k+1}\in \II_{q,k+1} \\
q_{k+1} > (\xi Q-q) \vee 2Q}} \int_{t_k}^{t_{k-1}}
W^{(1)}_{\gamma,k} (t)\, dt
\end{equation*}
estimated upon \eqref{5.7} as
\begin{equation}\label{9.1}
\G_{I,Q}^{(4.1.1.1)} (\xi) \approxeq \frac{c_I}{8\zeta(2)}
\int_0^1 du \int_{ 2 \vee (\xi -u) \vee (1+u)}^\infty
dw\, \frac{(1-u)^2}{(w-2u)(w-u)^2} \cdot (w+u-\xi)_+ \wedge w.
\end{equation}

$\bullet$ \ When $q_{k+1} \leqslant 2Q$ we have $\frac{a+\eps}{q} \leqslant \gamma_{k+1}
\leqslant t_{k-1}$ and $t_k < \frac{a+\eps}{q} \Longleftrightarrow
q_{k+2} > 2Q$ for all $k\geqslant 0$. In this case the contribution
\begin{equation*}
\begin{split}
\G_{I,Q}^{(4.1.1.2)}(\xi) & = \sum_{k=0}^\infty
\sideset{}{^*} \sum\limits_{\substack{q_{k+1} \in \II_{q,k+1} \\
2Q \geqslant q_{k+1} >(\xi \vee 2)Q -q}} \left(
\int_{t_k}^{\frac{a+\eps}{q}} W^{(1)}_{\gamma,k}(t) \, dt +
\int_{\frac{a+\eps}{q}}^{\gamma_{k+1}} W^{(2)}_{\gamma,k}(t)\, dt
\right) \\ & \quad + \sum_{k=0}^\infty \sideset{}{^*}
\sum\limits_{\substack{q_{k+1} \in \II_{q,k+1} \\
2Q-q \geqslant q_{k+1} > \xi Q-q}} \int_{t_k}^{\gamma_{k+1}}
W^{(2)}_{\gamma,k} (t)\, dt
\end{split}
\end{equation*}
is estimated upon \eqref{A1.15}, \eqref{A1.16}, \eqref{A1.18} (which also holds for $k=0$),
\eqref{A1.9}, \eqref{A1.20}, \eqref{A4.1}, \eqref{5.4}, as
\begin{equation}\label{9.2}
\begin{split}
\G_{I,Q}^{(4.1.1.2)}(\xi) \approxeq & \frac{c_I}{8\zeta (2)} \int_0^1 du
 \int_{(\xi \vee 2 -u) \wedge 2}^2  dw\, F^{(4.1.1.2.1)}(\xi; u,w) \\ & +
\frac{c_I}{8\zeta (2)} \int_0^1 du \int_{((\xi-u) \vee 1) \wedge (2-u)}^{2-u}
dw\, F^{(4.1.1.2.2)}(\xi; u,w) ,
\end{split}
\end{equation}
where
\begin{equation*}
\begin{split}
F^{(4.1.1.2.1)}(\xi; u,w) & =
\left( \frac{w+u-2}{2u(w-u)} \bigg( \frac{1-u}{w-u}
+\frac{2-w}{2u}\bigg) + \frac{(2-w)^2}{4u^2 w} \right) \cdot
(w+u-\xi) \wedge w  +\frac{(2-w)^2}{2uw} , \\
F^{(4.1.1.2.2)}(\xi;u,w) & = \frac{w-1}{w-u} \left(
\frac{2-u-w}{w-u} +\frac{2-w}{w}\right) +
\frac{(w-1)^2}{(w-u)^2 w} \cdot (w+u-\xi) \wedge w .
\end{split}
\end{equation*}

(II) $\ \xi Q \geqslant q_{k+2}$. Let $N$ be the unique integer for which
$q+Nq_{k+1} \leqslant \xi Q < q+(N+1)q_{k+1} =q_{k+2}+Nq_{k+1}$, that is
$1\leqslant N:=\left\lfloor \frac{\xi Q-q}{q_{k+1}}\right\rfloor \leqslant \xi$.
The corresponding range of $q_{k+1}$ is
\begin{equation*}
y=q_{k+1}\in \JJ_{q,N}:=\left( \frac{\xi Q-q}{N+1},\frac{\xi Q-q}{N}\right] .
\end{equation*}
We take $n_0:=\Big\lfloor \frac{w_{\CC_k}}{w_{\AA_0}-w_{\BB_k}} \Big\rfloor
= \Big\lfloor \frac{a_{k-1}-2\eps -q_{k-1}t}{a_{k+1}-q_{k+1}t} \Big\rfloor \geqslant 1$.
Then $t \geqslant
\frac{a+\eps}{q}$. We only need $\frac{a+\eps}{q} \leqslant
\gamma_{k+1}$, that is $q_{k+1} \leqslant 2Q$. In this case we
take
\begin{equation}\label{9.3}
\lambda_{k,n} :=\frac{na_{k+1}-a_{k-1}+2\eps}{nq_{k+1}-q_{k-1}}
\nearrow \gamma_{k+1} \leqslant t_{k-1} \quad \mbox{\rm as
$n\rightarrow \infty$,}
\end{equation}
hence $\lambda_{k,n_0} \leqslant t <\lambda_{k,n_0+1}$ and $T(q+n_0
q_{k+1}) =w_{\BB_k}+w_{\CC_k}-n_0 (w_{\AA_0}-w_{\BB_k})\geqslant
w_{\BB_k} > T\big( q+(n_0+1)q_{k+1}\big)$.
This shows that $\CC_k$
is locked by the slits $q+q_{k+1},\ldots,q+(n_0+1)q_{k+1}$ and
$W_{\gamma,k}(t)=W_{\gamma,k,N}(t)$ is given by
\begin{equation*}
\begin{split}
& \begin{cases} 0 & \mbox{\rm if $n_0 <N,$} \\
\big( w_{\CC_k} -N(w_{\AA_0}-w_{\BB_k})\big)\big (
q +(N+1) q_{k+1}-\xi Q\big) & \mbox{\rm if $n_0=N,$} \\
(w_{\AA_0}-w_{\BB_k}) \big( q +(N+1) q_{k+1} -\xi Q\big) +\big(
w_{\CC_k}-(N+1) (w_{\AA_0}-w_{\BB_k})\big) q_{k+1}
& \mbox{\rm if $n_0 >N,$} \end{cases} \\
 & = \begin{cases} 0 & \mbox{\rm if $t \in I_{\gamma,k} \cap
(\gamma,\lambda_{k,N}],$} \\ W^{(1)}_{\gamma,k,N}(t) := \big(
q+(N+1)q_{k+1}-\xi Q\big) & \\
\qquad \qquad \quad \cdot \big( (Nq_{k+1}-q_{k-1})t
-Na_{k+1}+a_{k-1}-2\eps \big) &
\mbox{\rm if $t \in I_{\gamma,k} \cap (\lambda_{k,N} , \lambda_{k,N+1}],$} \\
W^{(2)}_{\gamma,k}(t) := w_{\CC_k}(t) q_{k+1} - \big( w_{\AA_0}(t)
- w_{\BB_k}(t)\big) (\xi Q- q) & \mbox{\rm if $t \in
I_{\gamma,k} \cap (\lambda_{k,N+1},\gamma_{k+1}].$}
\end{cases}
\end{split}
\end{equation*}
The following three subcases arise:

(II$_1$) $\ q_{k+1} \leqslant Q+\frac{Q-q}{N+1}$. Then
$\lambda_{k,N+1}\leqslant t_k$,
$W_{\gamma,k}(t)=W^{(2)}_{\gamma,k}(t)$, $\forall t \in
(t_k,\gamma_{k+1}]$, with contribution
\begin{equation*}
\G^{(4.1.2)}_{I,Q}(\xi)=\sum_{1\leqslant N\leqslant \xi}
\sum_{k=0}^\infty\ \sideset{}{^*}\sum_{\substack{q \leqslant
\xi Q \\ q_{k+1}\in \II_{q,k+1}\cap\JJ_{q,N} \\ q_{k+1}
\leqslant Q+\frac{Q-q}{N+1}}} \int_{t_k}^{\gamma_{k+1}}
\frac{W^{(2)}_{\gamma,k} (t)\, dt}{t^2+t+1}.
\end{equation*}
Employing \eqref{5.4}, \eqref{A4.1}, \eqref{A1.9} (which also holds
for $k=0$) we find
\begin{equation}\label{9.4}
\G^{(4.1.2)}_{I,Q} (\xi) \approxeq \frac{c_I}{8\zeta(2)}
\sum_{1\leqslant N \leqslant \xi} \int_{0}^{\xi \wedge 1}
du \int_{\left[1,1+\frac{1-u}{N+1}\right] \cap \left[
\frac{\xi-u}{N+1},\frac{\xi-u}{N}\right]} dw\,
F^{(4.1.2)}(\xi;u,w),
\end{equation}
where
\begin{equation*}
F^{(4.1.2)}(\xi;u,w) = \frac{w-1}{w-u} \left(
\frac{1-u}{w-u}+\frac{2-w}{w}\right) -
\frac{(w-1)^2 (\xi -u)}{(w-u)^2 w} .
\end{equation*}

(II$_2$) $\ Q+\frac{Q-q}{N+1} < q_{k+1} \leqslant
Q+\frac{Q-q}{N}$. In this case $\lambda_{k,N} \leqslant t_k <
\lambda_{k,N+1}$. The contribution
\begin{equation*}
\G^{(4.1.3)}_{I,Q} (\xi) =\sum_{1\leqslant N\leqslant \xi}
\sum_{k=0}^\infty\ \sideset{}{^*}\sum_{\substack{q\leqslant \xi Q\\
q_{k+1}\in \II_{q,k+1}\cap\JJ_{q,N} \\ Q+\frac{Q-q}{N+1} <q_{k+1}
\leqslant Q+\frac{Q-q}{N}}} \left( \int_{t_k}^{\lambda_{k,N+1}}
\frac{W^{(1)}_{\gamma,k,N}(t)\, dt}{t^2+t+1} +
\int_{\lambda_{k,N+1}}^{\gamma_{k+1}} \frac{W^{(2)}_{\gamma,k}(t)\,
dt}{t^2+t+1} \right)
\end{equation*}
is estimated upon \eqref{A4.4}, \eqref{A4.5}, \eqref{A4.6}, \eqref{A4.9} as
\begin{equation}\label{9.5}
\G^{(4.1.3)}_{I,Q}(\xi) \approxeq \frac{c_I}{8\zeta (2)}
\sum_{1\leqslant N \leqslant \xi} \int_0^{\xi \wedge 1} du
\int_{\left[ 1 + \frac{1-u}{N+1} ,1+\frac{1-u}{N} \right] \cap
\left[ \frac{\xi-u}{N+1},\frac{\xi-u}{N}\right]} dw\,
F_N^{(4.1.3)}(\xi;u,w),
\end{equation}
where
\begin{equation*}
\begin{split}
F_N^{(4.1.3)}  & (\xi;u,w) = \frac{(2-w)^2 \big(
(2N+1)w+3u-\xi\big)}{(Nw+2u)^2 w} \\ &  +  \frac{\big(
(N+1)(w-1)+u-1\big) \big( (N+1)w+u-\xi\big)}{(w-u)(Nw+2u)}
\left( \frac{2-w}{Nw+2u}+\frac{1-u-N(w-1)}{w-u}\right) .
\end{split}
\end{equation*}

(II$_3$) $\ Q+\frac{Q-q}{N}< q_{k+1} \leqslant 2Q$. In this case
$t_k < \lambda_{k,N}$. The contribution
\begin{equation*}
\G^{(4.1.4)}_{I,Q}(\xi) =\sum_{1\leqslant N\leqslant \xi}
\sum_{k=0}^\infty\ \sideset{}{^*}\sum_{\substack{q\leqslant
\xi Q \\ q_{k+1}\in\II_{q,k+1}\cap\JJ_{q,N} \\
Q+\frac{Q-q}{N}<q_{k+1}\leqslant 2Q}} \left(
\int_{\lambda_{k,N}}^{\lambda_{k,N+1}}
\frac{W^{(1)}_{\gamma,k,N}(t)\, dt}{t^2+t+1}
+\int_{\lambda_{k,N+1}}^{\gamma_{k+1}}
\frac{W^{(2)}_{\gamma,k}(t)\, dt}{t^2+t+1} \right)
\end{equation*}
is estimated upon \eqref{A4.4}-\eqref{A4.7} as
\begin{equation}\label{9.6}
\G^{(4.1.4)}(\xi) \approxeq \frac{c_I}{8\zeta(2)}
\sum_{1\leqslant N \leqslant \xi} \int_0^{\xi \wedge 1} du
\int_{\left[ 1 + \frac{1-u}{N} , 2\right] \cap \left[
\frac{\xi-u}{N+1},\frac{\xi-u}{N}\right]} dw\,
F_N^{(4.1.4)} (\xi;u,w),
\end{equation}
where
\begin{equation*}
F_N^{(4.1.4)}(\xi;u,w) = \frac{(2-w)^2}{(Nw+2u)^2} \left(
\frac{(N+1)w+u-\xi}{(N-1)w+2u} + \frac{(2N+1)w+3u-\xi}{w}\right) .
\end{equation*}

\subsubsection{$w_{\AA_0}<w_{\BB_k}$ $(\Longleftrightarrow t
> \gamma_{k+1})$}\label{Subsub9.1.2}

In this situation $k\geqslant 1$ and $\gamma_{k+1}\leqslant
t_{k-1}$, so $q_{k+1} \leqslant 2Q$. Two cases arise:

(I) $\ \xi Q<q_k$. Then $W_{\gamma,k}(t)=w_{\CC_k}(t)
q_{k+1}$, $\forall t \in (\gamma_{k+1},t_{k-1}]$, with
contribution
\begin{equation*}
\G^{(4.2.1)}_{I,Q}(\xi) =\sum_{k=1}^\infty \
\sideset{}{^*}\sum_{q+\xi Q < q_{k+1} \leqslant 2Q}
\int_{\gamma_{k+1}}^{t_{k-1}} \frac{W_{\gamma,k}(t)\, dt}{t^2+t+1}.
\end{equation*}
Employing \eqref{A3.1} we find
\begin{equation}\label{9.7}
\G_{I,Q}^{(4.2.1)}(\xi) \approxeq \frac{c_I}{8\zeta(2)}
\int_0^1 du \int_{2\wedge (u+ \xi\vee 1)}^2 dw\, \frac{(2-w)^2}{(w-2u)w} .
\end{equation}

(II) $\ \xi Q\geqslant q_k$. Consider the unique integer
$0\leqslant N=\left\lfloor \frac{\xi Q-q_k}{q_{k+1}}
\right\rfloor \leqslant \xi$ for which $q_k+N q_{k+1} <\xi
Q\leqslant q_k+(N+1)q_{k+1}=q_{k+2}+Nq_{k+1}$, so the range of
$q_{k+1}$ is
\begin{equation*}
y=q_{k+1} \in \JJ_{q,N}:= \left( \frac{\xi Q + q }{ N + 2 } ,
\frac{\xi Q+q}{N+1}\right] .
\end{equation*}
This time we take $n_0:=\left\lfloor
\frac{w_{\CC_k}}{w_{\BB_k}-w_{\AA_0}}\right\rfloor = \left\lfloor
\frac{a_{k-1}-2\eps-q_{k-1}t}{q_{k+1}t -a_{k+1}} \right\rfloor
\geqslant 0$ and
\begin{equation}\label{9.8}
\lambda_{k,n}:=\frac{na_{k+1}+a_{k-1}-2\eps}{nq_{k+1}+q_{k-1}}
\searrow \gamma_{k+1} > t_k \quad \mbox{\rm as $n\rightarrow
\infty,$}
\end{equation}
hence $\lambda_{k,n_0+1} <t \leqslant \lambda_{k,n_0}$ and $B(q_k+n_0
q_{k+1})=w_{\BB_k}+n_0 (w_{\BB_k}-w_{\AA_0}) \leqslant
w_{\BB_k}+w_{\CC_k} <B\big( q_k +(n_0+1)q_{k+1}\big)$. This shows
that $\CC_k$ is locked by the slits
$q_k+q_{k+1},\ldots,q_k+(n_0+1)q_{k+1}$ and $W_{\gamma,k}(t) = W_{\gamma,k,N} (t)$ is given by
\begin{equation*}
\begin{split}
 & \begin{cases} 0 & \mbox{\rm if $n_0<N,$} \\
\big( w_{\CC_k}-N(w_{\BB_k}-w_{\AA_0})\big)\big(q_k+(N+1)q_{k+1}
-\xi Q\big) & \mbox{\rm if $n_0=N,$} \\
(w_{\BB_k}-w_{\AA_0})\big( q_k +(N+1)q_{k+1}-\xi Q\big)+\big(
w_{\CC_k}-(N+1)(w_{\BB_k}-w_{\AA_0})\big) q_{k+1} & \mbox{\rm if
$n_0 >N,$} \end{cases} \\
& =\begin{cases} 0 & \mbox{\rm if $t\in (\lambda_{k,N},t_{k-1}],$} \\
W^{(3)}_{\gamma,k,N}(t):=\big(q_{k}+(N+1) q_{k+1}-\xi Q\big) & \\
\qquad \qquad \qquad \cdot \big( Na_{k+1}+a_{k-1}-2\eps -
(Nq_{k+1}+q_{k-1}) t\big) &
\mbox{\rm if $t \in (\lambda_{k,N+1},\lambda_{k,N}],$} \\
W^{(4)}_{\gamma,k}(t): = w_{\CC_k}(t) q_{k+1} - \big( w_{\BB_k}(t)
- w_{\AA_0}(t)\big) (\xi Q-q_k) & \mbox{\rm if $t \in
(\gamma_{k+1},\lambda_{k,N+1}].$}
\end{cases}
\end{split}
\end{equation*}
Since $\lambda_{k,0}=t_{k-1}$, the corresponding contribution is given by
\begin{equation*}
\G^{(4.2.2)}_{I,Q}(\xi) =\sum_{0\leqslant N\leqslant \xi}
\sum_{k=1}^\infty
\sideset{}{^*}\sum_{\substack{q_{k+1}\leqslant 2Q \\
q_{k+1}\in\II_{q,k+1}\cap\JJ_{q,N}}} \left(
\int_{\lambda_{k,N+1}}^{\lambda_{k,N}}
\frac{W^{(3)}_{\gamma,k,N}(t)\, dt}{t^2+t+1} +
\int_{\gamma_{k+1}}^{\lambda_{k,N+1}} \frac{W^{(4)}_{\gamma,k}(t)\,
dt}{t^2+t+1} \right).
\end{equation*}
Employing \eqref{A4.12}, \eqref{A4.13}, \eqref{A4.14} we find
\begin{equation}\label{9.9}
\G_{I,Q}^{(4.2.2)}(\xi) \approxeq \frac{c_I}{8\zeta (2)}
\sum_{0\leqslant N \leqslant \xi} \int_0^1 du \int_{[1+u,2]
\cap \left[ \frac{\xi + u}{N+2}, \frac{\xi
+u}{N+1}\right]} dw \, F_N^{(4.2.2)} (\xi;u,w) ,
\end{equation}
where
\begin{equation*}
F_N^{(4.2.2)} (\xi;u,w) = \frac{(2-w)^2}{\big( (N+2) w - 2u
\big)^2} \left( \frac{(N+2)w-u-\xi}{(N+1)w-2u}
+\frac{(2N+4)w-3u-\xi}{w}  \right) .
\end{equation*}

\subsection{The channel $\CC_\leftarrow$ when $r=1$ and $r_k=-1$}\label{Subsec9.2}
The contributions of $(\CC_\leftarrow,r=1,r_k=-1)$ and of
$(\CC_\downarrow,r=-1,r_k=1)$ have the same main and error terms,
so we shall consider below the former situation and double the
result. The slits $q+n q_{k+1}$ are removed, while $2q+n q_{k+1}$
are not, $n \geqslant 0$ (see Figure \ref{Figure17}). Note that
$B(2q)=2(w_{\BB_k}+w_{\CC_k}) > B(q_{k+1}) =w_{\CC_k}+2w_{\BB_k}$.
Again, two cases arise:

\subsubsection{$w_{\AA_0}\leqslant w_{\BB_k} \
(\Longleftrightarrow t \geqslant
\gamma_{k+1})$}\label{Subsub9.2.1}

The slit $q_{k+1}$ locks the channel $\AA_0$ and $W_{\gamma,k} (t)
=w_{\AA_0}(t) \cdot (q_{k+1}-\xi Q)_+ \wedge q$. We must have
$\gamma_{k+1} < t_{k-1}$, so $k\geqslant 1$ and $q_{k+1}\leqslant 2Q$. The
corresponding contribution is
\begin{equation*}
\G_{I,Q}^{(4.3.1)} (\xi)= \sum_{k=1}^\infty
\sideset{}{^*}\sum_{\substack{q_{k+1}\in \II_{q,k+1} \\ q_{k+1}
\leqslant 2Q}} (q_{k+1}-\xi Q)_+ \wedge q
\int_{\gamma_{k+1}}^{t_{k-1}} \frac{w_{\AA_0}(t)\, dt}{t^2+t+1} .
\end{equation*}
Employing \eqref{A4.2} we find
\begin{equation}\label{9.10}
\G_{I,Q}^{(4.3.1)}(\xi) \approxeq \frac{c_I}{8\zeta(2)}
\int_0^1 du \int_{1+u}^2 dw\, \frac{2-w}{(w-2u)w} \left(
\frac{w-1}{w} +\frac{w-u-1}{w-2u}\right) \cdot
(w-\xi)_+ \wedge u .
\end{equation}

\subsubsection{$w_{\AA_0}>w_{\BB_k}\ (\Longleftrightarrow t
< \gamma_{k+1})$}\label{Subsub9.2.2}

Consider first the sub-channel of $\AA_0$ of width $w_{\BB_k}$
locked by the slit $q_{k+1}$. Its contribution is
\begin{equation*}
\begin{split}
\G_{I,Q}^{(4.3.2)}(\xi) & = \sum_{k=0}^\infty\
\sideset{}{^*}\sum_{\substack{q_{k+1}\in\II_{q,k+1} \\
q_{k+1}\leqslant 2Q}} (q_{k+1}-\xi Q)_+ \wedge q
\int_{t_k}^{\gamma_{k+1}} \frac{w_{\BB_k}(t)\, dt}{t^2+t+1} \\ &
\quad +\sum_{k=1}^\infty\
\sideset{}{^*}\sum_{\substack{q_{k+1}\in\II_{q,k+1}
\\ q_{k+1} > 2Q}} (q_{k+1}-\xi Q)_+ \wedge q
\int_{t_k}^{t_{k-1}} \frac{w_{\BB_k}(t)\, dt}{t^2+t+1}.
\end{split}
\end{equation*}
Employing \eqref{A1.8} and \eqref{5.6} we find
\begin{equation}\label{9.11}
\G_{I,Q}^{(4.3.2)}(\xi) \approxeq \frac{c_I}{8\zeta(2)}
\int_0^1 du \left( \ \int_1^2 dw\ F^{(4.3.2.1)}(\xi;u,w) +
\int_2^\infty dw\, F^{(4.3.2.2)}(\xi ;u,w) \right),
\end{equation}
where
\begin{equation*}
\begin{split}
& F^{(4.3.2.1)}(\xi;u,w)=\frac{(w-1)^2}{(w-u) w^2}\cdot
(w-\xi)_+ \wedge u ,\\ & F^{(4.3.2.2)}(\xi;u,w)=
\frac{(1-u)^2}{(w-2u)^2 (w-u)} \cdot (w-\xi)_+ \wedge u.
\end{split}
\end{equation*}

\begin{figure}[ht]
\centering
\unitlength 0.27mm
\begin{picture}(400,200)(-45,-5)
\texture{ccccccc 0000}
\shade\path(0,85)(0,90)(30,90)(30,85)(0,85)
\shade\path(300,80)(300,85)(285,85)(285,80)(300,80)

\put(-7,70){\makebox(0,0){\small $(-1,-\eps)$}}
\put(-7,165){\makebox(0,0){\small $(-1,\eps)$}}

\put(30,22){\makebox(0,0){\small $q$}}
\put(60,-4){\makebox(0,0){\small $2q$}}
\put(50,187){\makebox(0,0){\small $q_k$}}
\put(80,172){\makebox(0,0){\small $q_{k+1}$}}
\put(375,10){\makebox(0,0){\small $2q+4q_{k+1}$}}
\put(300,0){\makebox(0,0){\small $2q+3q_{k+1}$}}
\put(220,-4){\makebox(0,0){\small $2q+2q_{k+1}$}}
\put(140,-4){\makebox(0,0){\small $q_{k+3}=2q+q_{k+1}$}}
\put(110,122){\makebox(0,0){\small $q_{k+2}$}}
\put(160,184){\makebox(0,0){\small $2q_{k+1}$}}
\put(240,189){\makebox(0,0){\small $3q_{k+1}$}}
\put(320,189){\makebox(0,0){\small $4q_{k+1}$}}
\put(190,134){\makebox(0,0){\small $q+2q_{k+1}$}}
\put(270,144){\makebox(0,0){\small $q+3q_{k+1}$}}
\put(350,154){\makebox(0,0){\small $q+4q_{k+1}$}}
\put(285,103){\makebox(0,0){\small $\xi Q$}}

\put(-11,92){\makebox(0,0){\small $\AA_0$}}
\put(-10,118){\makebox(0,0){\small $\CC_k$}}
\put(-10,147){\makebox(0,0){\small $\BB_k$}}

\put(181,72){\makebox(0,0){\small $w_{\AA_0}-w_{\BB_k}$}}
\put(99,62){\makebox(0,0){\small $w_{\AA_0}-w_{\BB_k}$}}
\put(97,97){\makebox(0,0){\small $w_{\BB_k}$}}

\path(80,90)(380,90) \path(30,80)(300,80) \path(300,85)(380,85)

\dottedline{2}(30,105)(30,30) \dottedline{2}(110,40)(110,115)
\dottedline{2}(190,50)(190,125) \dottedline{2}(270,60)(270,135)
\dottedline{2}(350,70)(350,145)

\thicklines  \path(0,80)(30,80) \path(30,105)(80,105)
\path(50,140)(80,140) \path(0,155)(50,155)

\Thicklines \path(380,95)(380,20) \path(300,85)(300,10)
\path(220,75)(220,5) \path(140,65)(140,5) \path(60,55)(60,5)
\path(50,140)(50,180) \path(80,90)(80,165) \path(160,100)(160,175)
\path(240,110)(240,180) \path(320,120)(320,180) \path(0,80)(0,155)

\thinlines \path(133,55)(133,65) \path(131.5,58)(133,55)(134.5,58)
\path(131.5,62)(133,65)(134.5,62) \path(30,30)(60,30)
\path(110,40)(140,40) \path(190,50)(220,50) \path(270,60)(300,60)
\path(350,70)(380,70) \path(110,115)(160,115)
\path(190,125)(240,125) \path(270,135)(320,135)
\path(350,145)(390,145) \path(380,95)(390,95)
\path(160,100)(190,100) \path(240,110)(270,110)
\path(320,120)(350,120) \path(0,90)(80,90) \path(0,105)(30,105)
\path(0,140)(50,140) \path(0,85)(300,85) \path(60,55)(110,55)
\path(110,40)(140,40) \path(140,65)(190,65) \path(213,65)(213,75)
\path(211.5,68)(213,65)(214.5,68)
\path(211.5,72)(213,75)(214.5,72) \path(220,75)(270,75)

\path(285,70)(285,95) \path(-25,80)(-25,90)
\path(-23.5,83)(-25,80)(-26.5,83)
\path(-23.5,87)(-25,90)(-26.5,87)

\path(85,90)(85,105) \path(83.5,93)(85,90)(86.5,93)
\path(83.5,102)(85,105)(86.5,102)

\put(-60,85){\makebox(0,0){\small $w_{\AA_0}-w_{\BB_k}$}}

\end{picture}
\caption{The case $t \in I_{\gamma,k}$, $r=1$, $r_k=-1$,
$r_{k+1}=0$, $\CC_{\leftarrow}$, $w_{\AA_0}>w_{\BB_k}$, $n_0=3$}
\label{Figure17}
\end{figure}

The remaining part $\tilde{\AA}_0$ of $\AA_0$ (of total width
$w_{\AA_0}-w_{\BB_k}$) is locked by the slits $2q+n_0 q_{k+1}$ and
$2q+(n_0+1)q_{k+1}$, with $n_0$ uniquely determined by $2\eps >
2(w_{\BB_k}+w_{\CC_k})-n_0 (w_{\AA_0}-w_{\BB_k}) \geqslant
2w_{\BB_k}+w_{\CC_k}$, or equivalently $n_0=\Big\lfloor
\frac{w_{\CC_k}}{w_{\AA_0}-w_{\BB_k}} \Big\rfloor=\left\lfloor
\frac{a_{k-1}-2\eps-q_{k-1}t}{a_{k+1}-q_{k+1}t} \right\rfloor
\geqslant 0$. The widths of the relevant sub-channels (from bottom
to top) are $2\eps -2(w_{\BB_k}+w_{\CC_k}) +n_0
(w_{\AA_0}-w_{\BB_k}) =(n_0+1) a_{k+1} -a_{k-1} +2\eps -\big(
(n_0+1)q_{k+1} - q_{k-1} \big) t$ and $w_{\CC_k} -n_0
(w_{\AA_0}-w_{\BB_k}) =(n_0 q_{k+1} -q_{k-1}) t -n_0 a_{k+1}
+a_{k-1} -2\eps$. The following two subcases arise:

(I) $\ n_0=0$ $(\Longleftrightarrow t <\frac{a+\eps}{q})$.
From $t_k <\frac{a+\eps}{q}$ we infer $q_{k+2}>2Q$. In this
situation we have
\begin{equation*}
W_{\gamma,k,0}(t) = 2(a+\eps-qt)\cdot (2q-\xi Q)_+ \wedge q +
w_{\CC_k}(t) \cdot (2q+q_{k+1}-\xi Q)_+ \wedge q,
\end{equation*}
with contribution ($q_{k+1} >2Q  \Longleftrightarrow
t_{k-1}<\gamma_{k+1}$)
\begin{equation*}
\G_{I,Q}^{(4.3.3)}(\xi) =\sum_{k=0}^\infty
\sideset{}{^*}\sum_{\substack{q_{k+1}\in \II_{q,k+1} \\ 2Q-q <
q_{k+1} \leqslant 2Q}} \int_{t_k}^{\frac{a+\eps}{q}}
\frac{W_{\gamma,k,0}(t)\, dt}{t^2+t+1} + \sum_{k=1}^\infty
\sideset{}{^*}\sum_{\substack{q_{k+1} \in \II_{q,k+1} \\ q_{k+1}
> 2Q}} \int_{t_k}^{t_{k-1}} \frac{W_{\gamma,k,0}(t)\, dt}{t^2+t+1} .
\end{equation*}
Employing \eqref{A1.14}, \eqref{A1.15}, \eqref{A1.16}, \eqref{5.7},
\eqref{A1.1} we find
\begin{equation}\label{9.12}
\G_{I,Q}^{(4.3.3)}(\xi) \approxeq \frac{c_I}{8\zeta(2)}
\int_0^1 du \left(\ \int_{2-u}^2 dw\ F^{(4.3.3.1)}(\xi;u,w) +
\int_2^\infty dw\ F^{(4.3.3.2)}(\xi;u,w) \right),
\end{equation}
where
\begin{equation*}
\begin{split}
F^{(4.3.3.1)} (\xi;u,w) = & \frac{w+u-2}{2u(w-u)} \left(
\frac{1-u}{w-u}+\frac{2-w}{2u}\right) \cdot  (w+2u-\xi)_+
\wedge u \\ & +\frac{(w+u-2)^2}{2u(w-u)^2}
\cdot (2u-\xi)_+ \wedge u , \\
F^{(4.3.3.2)}(\xi;u,w) = & \frac{(1-u)^2}{(w-2u)(w-u)^2} \cdot (w+2u-\xi)_+ \wedge u \\
& +\frac{1-u}{(w-2u)(w-u)} \left( \frac{w-2}{w-2u} +
\frac{w+u-2}{w-u} \right) \cdot (2u-\xi)_+ \wedge u .
\end{split}
\end{equation*}

(II) $\ n_0 \geqslant 1 \ (\Longleftrightarrow t\geqslant
\frac{a+\eps}{q})$. From $\frac{a+\eps}{q}\leqslant \gamma_{k+1}$
we infer $q_{k+1}\leqslant 2Q$. Taking $\lambda_{k,n}$ as in
\eqref{9.3} we obtain, when $t \in [ \lambda_{k,n_0} ,
\lambda_{k,n_0+1})$,
\begin{equation*}
\begin{split}
W_{\gamma,k,n}(t) = & \big( (n_0+1)a_{k+1}-a_{k-1}+2\eps -\big(
(n_0+1)q_{k+1} -q_{k-1} ) t \big) \cdot (2q+n_0 q_{k+1} -\xi Q
)_+ \wedge q \\ & +\big( (n_0 q_{k+1}-q_{k-1}) t - n_0 a_{k+1} +
a_{k-1} -2\eps\big) \cdot \big( 2q+(n_0+1) q_{k+1}-\xi Q
\big)_+ \wedge q .
\end{split}
\end{equation*}
Ordering $\lambda_{k,n_0}$, $\lambda_{k,n_0+1}$ and $t_k$, the
corresponding contribution takes the form
\begin{equation*}
\G_{I,Q}^{(4.3.4)}(\xi) = \sum_{n=1}^\infty \sum_{k=0}^\infty
\left(\ \sideset{}{^*}\sum_{\substack{q_{k+1}\in \II_{q,k+1}
\\ Q+\frac{Q-q}{n+1} < q_{k+1} \leqslant Q+\frac{Q-q}{n}}}
\hspace{-10pt} \int_{t_k}^{\lambda_{k,n+1}}
\frac{W_{\gamma,k,n}(t)\, dt}{t^2+t+1} +
\sideset{}{^*}\sum_{\substack{q_{k+1}\in \II_{q,k+1} \\
Q+\frac{Q-q}{n} < q_{k+1} \leqslant 2Q}} \hspace{-5pt}
\int_{\lambda_{k,n}}^{\lambda_{k,n+1}} \frac{W_{\gamma,k,n}(t)\,
dt}{t^2+t+1}\right) .
\end{equation*}
Employing \eqref{A4.8}-\eqref{A4.11} we find
\begin{equation}\label{9.13}
\begin{split}
\G_{I,Q}^{(4.3.4)}(\xi) \approxeq & \frac{c_I}{8\zeta(2)} \sum_{n=1}^\infty
\int_0^1 du  \int_{1+\frac{1-u}{n+1}}^{1+\frac{1-u}{n}}  dw\,
F_n^{(4.3.4.1)}(\xi;u,w) \\ & + \frac{c_I}{8\zeta(2)} \sum_{n=1}^\infty \int_0^1 du \int_{1+\frac{1-u}{n}}^2 dw\,
F_n^{(4.3.4.2)}(\xi;u,w)  ,
\end{split}
\end{equation}
where
\begin{equation*}
\begin{split}
& F_n^{(4.3.4.1)}(\xi;u,w) = \frac{\big( (n+1) (w-1) + u - 1
\big)^2}{(w-u)^2 (nw+2u)}  \cdot
(nw+2u-\xi)_+ \wedge u \\
& \quad + \frac{(n+1)(w-1)+u-1}{(w-u) (nw+2u)} \left(
\frac{2-w}{nw+2u} +\frac{1-u-n(w-1)}{w-u} \right) \cdot
\big( (n+1)w+2u-\xi\big)_+ \wedge u ,    \\
& F_n^{(4.3.4.2)}(\xi;u,w) = \frac{(2-w)^2}{(nw+2u)\big(
(n-1)w+2u\big)^2} \cdot (nw+2u-\xi)_+ \wedge u \\ &
\quad + \frac{(2-w)^2}{(nw+2u)^2 \big( (n-1)w+2u\big)}
\cdot \big( (n+1)w+2u-\xi\big)_+ \wedge u .
\end{split}
\end{equation*}

\subsection{The channel $\CC_\leftarrow$ when $r=-1$ and
$r_k=1$}\label{Subsec9.3}

The $(\CC_\leftarrow,r=-1,r_k=1)$ and
$(\CC_\downarrow,r=1,r_k=-1)$ contributions have the same main and
error terms, so we shall consider below the former situation and
double the result. This is analogous to Section \ref{Subsec9.2},
only that this time the slits $q_k+n q_{k+1}$ are removed, while
$2q+n q_{k+1}$ are not. Two cases arise:

\subsubsection{$w_{\BB_k}\leqslant w_{\AA_0}\ (\Longleftrightarrow t \leqslant
\gamma_{k+1})$}\label{Subsub9.3.1}

The slit $q_{k+1}$ locks the channel $\BB_k$ and $W_{\gamma,k} (t) =
w_{\BB_k}(t)\cdot (q_{k+1}-\xi Q)_+ \wedge q_k$. The
contribution
\begin{equation*}
\begin{split}
\G_{I,Q}^{(4.4.1)}(\xi) & = \sum_{k=0}^\infty\
\sideset{}{^*}\sum_{\substack{q_{k+1}\in\II_{q,k+1} \\ q_{k+1}
\leqslant 2Q}} (q_{k+1}-\xi Q)_+ \wedge q_k
\int_{t_k}^{\gamma_{k+1}} \frac{w_{\BB_k}(t)\, dt}{t^2+t+1}
\\ & \quad  + \sum_{k=1}^\infty\ \sideset{}{^*}\sum_{\substack{q_{k+1} \in
\II_{q,k+1} \\ q_{k+1} > 2Q}} (q_{k+1}-\xi Q)_+ \wedge q_k
\int_{t_k}^{t_{k-1}} \frac{w_{\BB_k}(t)\, dt}{t^2+t+1}
\end{split}
\end{equation*}
is estimated upon \eqref{A1.8}, \eqref{5.3}, \eqref{5.6} as
\begin{equation}\label{9.14}
\G_{I,Q}^{(4.4.1)}(\xi) \approxeq \frac{c_I}{8\zeta(2)}
\int_0^1 du \left( \int_1^2 dw\, F^{(4.4.1.1)}(\xi;u,w) +
\int_2^\infty dw\, F^{(4.4.1.2)}(\xi;u,w) \right),
\end{equation}
where
\begin{equation*}
\begin{split}
& F^{(4.4.1.1)}(\xi;u,w)=\frac{(w-1)^2}{(w-u) w^2} \cdot
(w-\xi)_+ \wedge (w-u) , \\ &
F^{(4.4.1.2)}(\xi;u,w)=\frac{(1-u)^2}{(w-2u)^2 (w-u)} \cdot
(w-\xi)_+ \wedge (w-u) .
\end{split}
\end{equation*}

\begin{figure}[ht]
\centering
\unitlength 0.33mm
\begin{picture}(370,205)(-20,0)

\texture{ccccccc 0000}
\shade\path(0,135)(0,145)(70,145)(70,135)(0,135)
\shade\path(240,155)(240,145)(220,145)(220,155)(240,155)

\put(-22,71){\makebox(0,0){\small $(-1,-\eps)$}}
\put(-19,158){\makebox(0,0){\small $(-1,\eps)$}}

\put(30,2){\makebox(0,0){\small $q$}}
\put(70,197){\makebox(0,0){\small $q_k$}}
\put(100,48){\makebox(0,0){\small $q_{k+1}$}}
\put(140,197){\makebox(0,0){\small $2q_{k}$}}
\put(170,187){\makebox(0,0){\small $q_k+q_{k+1}$}}
\put(200,27){\makebox(0,0){\small $2q_{k+1}$}}
\put(240,197){\makebox(0,0){\small $2q_k+q_{k+1}$}}
\put(340,117){\makebox(0,0){\small $2q_k+2q_{k+1}$}}
\put(220,62){\makebox(0,0){\small $\xi Q$}}

\put(-10,82){\makebox(0,0){\small $\AA_0$}}
\put(-10,105){\makebox(0,0){\small $\CC_k$}}
\put(-10,137){\makebox(0,0){\small $\BB_k$}}

\put(274,151){\makebox(0,0){\small $w_{\BB_k}-w_{\AA_0}$}}
\put(138,140){\makebox(0,0){\small $w_{\BB_k}-w_{\AA_0}$}}
\put(122,127){\makebox(0,0){\small $w_{\AA_0}$}}

\path(100,135)(340,135) \path(240,145)(340,145)
\path(70,155)(240,155)

\dottedline{2}(70,120)(70,190) \dottedline{2}(170,100)(170,180)
\dottedline{2}(270,80)(270,160) \dottedline{2}(220,160)(220,70)

\thicklines \path(0,155)(70,155) \path(70,120)(100,120)
\path(100,90)(30,90) \path(0,75)(30,75)

\Thicklines \path(0,75)(0,155) \path(30,10)(30,90)
\path(100,55)(100,135) \path(140,165)(140,190)
\path(200,35)(200,115) \path(240,145)(240,190)
\path(340,125)(340,190)

\thinlines \path(0,120)(100,120) \path(0,135)(100,135)
\path(0,145)(240,145) \path(0,90)(30,90) \path(140,165)(170,165)
\path(270,160)(340,160) \path(340,190)(300,190)(300,200)
\path(170,180)(240,180) \path(170,100)(200,100)
\path(200,115)(270,115)

\path(247,145)(247,165) \path(245,148)(247,145)(249,148)
\path(245,162)(247,165)(249,162)

\path(110,135)(110,155)\path(108,138)(110,135)(112,138)
\path(108,152)(110,155)(112,152)

\path(110,135)(110,120)\path(108.5,132)(110,135)(111.5,132)
\path(108.5,123)(110,120)(111.5,123)
\end{picture}
\caption{The case $t\in I_{\gamma,k}$, $r=-1$, $r_k=1$,
$r_{k+1}=0$, $\CC_{\leftarrow}$, $w_{\AA_0}<w_{\BB_k}$, $n_0=1$}
\label{Figure18}
\end{figure}

\subsubsection{$w_{\AA_0} <w_{\BB_k}\ (\Longleftrightarrow t >
\gamma_{k+1})$}\label{Subsub9.3.2}

Then $k\geqslant 1$ and $q_{k+1} \leqslant 2Q$. Consider first the
sub-channel of $\BB_k$ of width $w_{\AA_0}$ locked by $q_{k+1}$,
with contribution
\begin{equation*}
\G_{I,Q}^{(4.4.2)}(\xi) =\sum_{k=1}^\infty \
\sideset{}{^*}\sum_{\substack{q_{k+1}\in\II_{q,k+1} \\ q_{k+1}
\leqslant 2Q}} (q_{k+1}-\xi Q)_+ \wedge q_k
\int_{\gamma_{k+1}}^{t_{k-1}} \frac{w_{\AA_0}(t)\, dt}{t^2+t+1}
\end{equation*}
estimated upon \eqref{A2.1} as
\begin{equation}\label{9.15}
\approxeq \frac{c_I}{8\zeta(2)}
\int_0^1 du \int_{1+u}^2 dw\, \frac{2-w}{(w-2u)w}  \left(
\frac{w-1}{w}+\frac{w-u-1}{w-2u} \right)\cdot (w-\xi)_+
\wedge (w-u).
\end{equation}

The remaining sub-channel $\tilde{\BB}_k$ of $\BB_k$ (of
width $w_{\BB_k}-w_{\AA_0}$) is locked by the slits $2q_k+n_0
q_{k+1}$ and $2q_k +(n_0+1)q_{k+1}$, with $n_0$ uniquely
determined by $0\leqslant B(2q_k) +n_0(w_{\BB_k} -w_{\AA_0})
<w_{\BB_k}-w_{\AA_0} =T(q_{k+1})$, or equivalently
$n_0:=\Big\lfloor \frac{w_{\CC_k}}{w_{\BB_k}-w_{\AA_0}}\Big\rfloor
= \Big\lfloor \frac{a_{k-1}-2\eps -q_{k-1}t}{q_{k+1}t -a_{k+1}}
\Big\rfloor \geqslant 0$. The situation is described in Figure
\ref{Figure18}. The widths of the relevant two sub-channels of
$\tilde{\BB}_k$ are (from top to bottom) $B(2q_k+n_0q_{k+1})=\big(
(n_0+1)q_{k+1}+q_{k-1}\big)t -(n_0+1) a_{k+1} -a_{k-1} +2\eps$ and
$w_{\BB_k}-w_{\AA_0}-B(2q_k+n_0 q_{k+1})=n_0 a_{k+1}
+a_{k-1} -2\eps -(n_0 q_{k+1}+q_{k-1})t$. Taking this time
$\lambda_{k,n}$ as in \eqref{9.8} we find, for $t\in
(\lambda_{k,n_0+1},\lambda_{k,n_0}]$,
\begin{equation*}
\begin{split}
W_{\gamma,k,n_0}(t) = & \Big(\big( (n_0+1)q_{k+1}+q_{k-1}\big) t
-\big( (n_0+1)a_{k+1} + a_{k-1} - 2\eps\big) \Big) \cdot
(2q_k+n_0q_{k+1}-\xi Q)_+ \wedge q_k \\ & + \Big( n_0
a_{k+1}+a_{k-1}-2\eps -(n_0 q_{k+1}+q_{k-1}) t \Big) \cdot \big(
2q_k +(n_0+1)q_{k+1}-\xi Q\big)_+ \wedge q_k .
\end{split}
\end{equation*}
Since $\lambda_{k,0}=t_{k-1}$ the corresponding contribution is given by
\begin{equation*}
\G_{I,Q}^{(4.4.3)}(\xi) =\sum_{n=0}^\infty \sum_{k=1}^\infty \
\sideset{}{^*}\sum_{\substack{q_{k+1}\in \II_{q,k+1} \\ q_{k+1}
\leqslant 2Q}} \int_{\lambda_{k,n+1}}^{\lambda_{k,n}}
\frac{W_{\gamma,k,n}(t)\, dt}{t^2+t+1} .
\end{equation*}
Employing \eqref{A4.14} and \eqref{A4.15} we find
\begin{equation}\label{9.16}
\G^{(4.4.3)}_{I,Q}(\xi) \approxeq \frac{c_I}{8\zeta(2)}
\int_0^1 du \ \int_{1+u}^2 dw\ \sum_{n=0}^\infty F^{(4.4.3)}_n
(\xi;u,w) ,
\end{equation}
where
\begin{equation*}
\begin{split}
F_n^{(4.4.3)}(\xi;u,w)= & \frac{(2-w)^2}{\big(
(n+1)w-2u\big)^2 \big( (n+2)w-2u \big)} \cdot \big(
(n+2) w-2u-\xi \big)_+ \wedge (w-u) \\
& +\frac{(2-w)^2}{\big( (n+1)w-2u\big) \big( (n+2)w-2u\big)^2}
\cdot \big( (n+3)w-2u-\xi\big)_+ \wedge (w-u) .
\end{split}
\end{equation*}

\begin{figure}[ht]
\centering{\includegraphics*[scale=.7, bb=0 0 240 160]{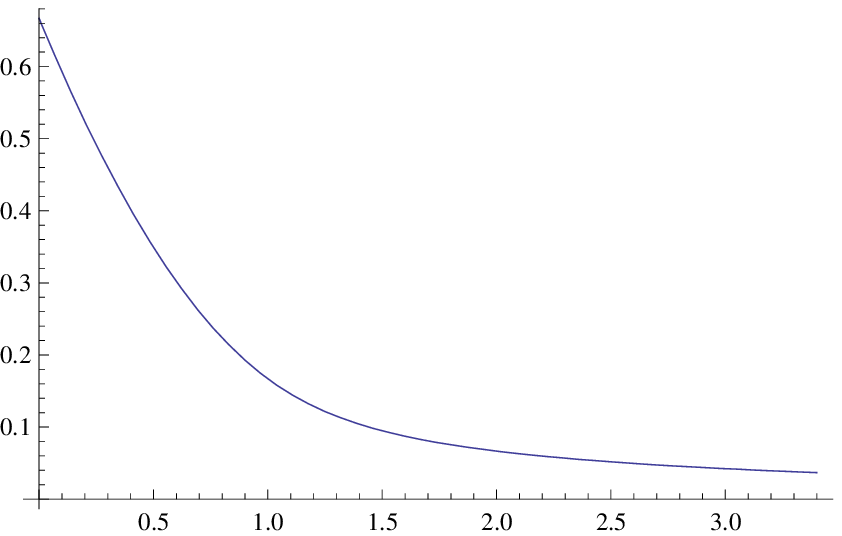}}
\centering{\includegraphics*[scale=.7, bb=0 0 240 160]{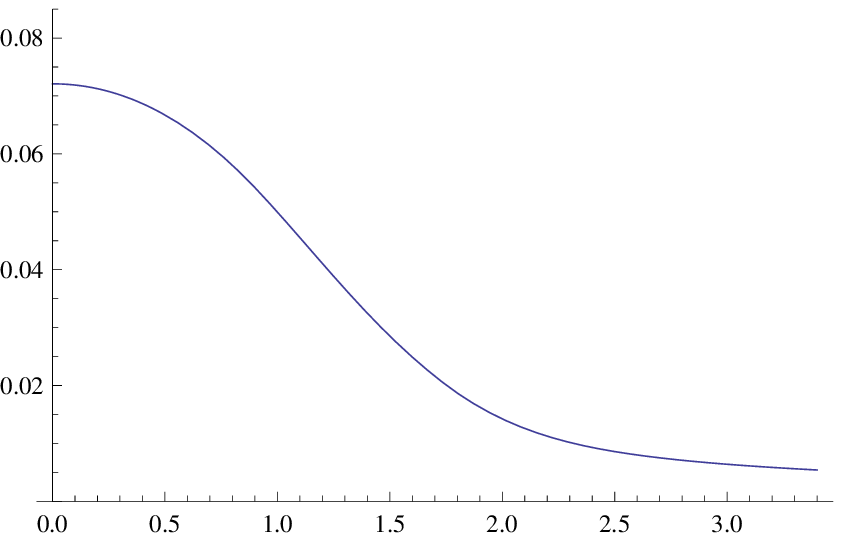}}
\centering{\includegraphics*[scale=.7, bb=0 0 240 160]{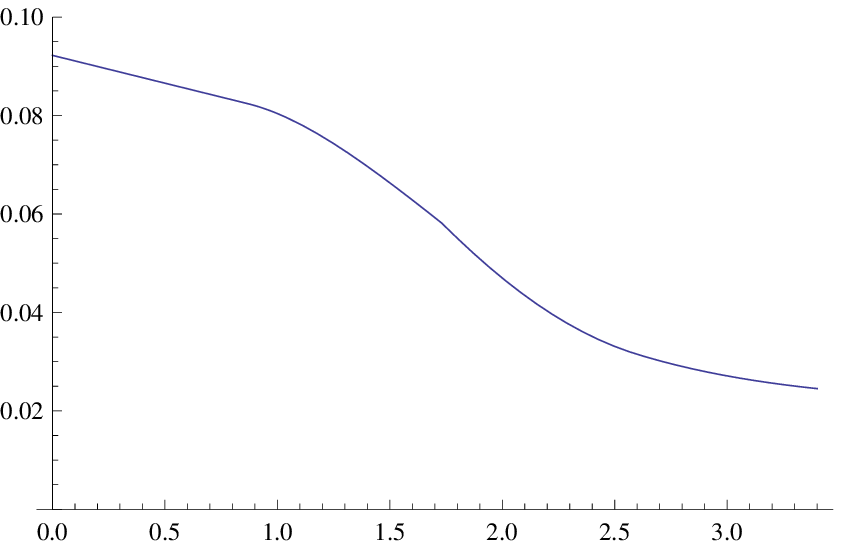}}
\centering{\includegraphics*[scale=.7, bb=0 0 240 160]{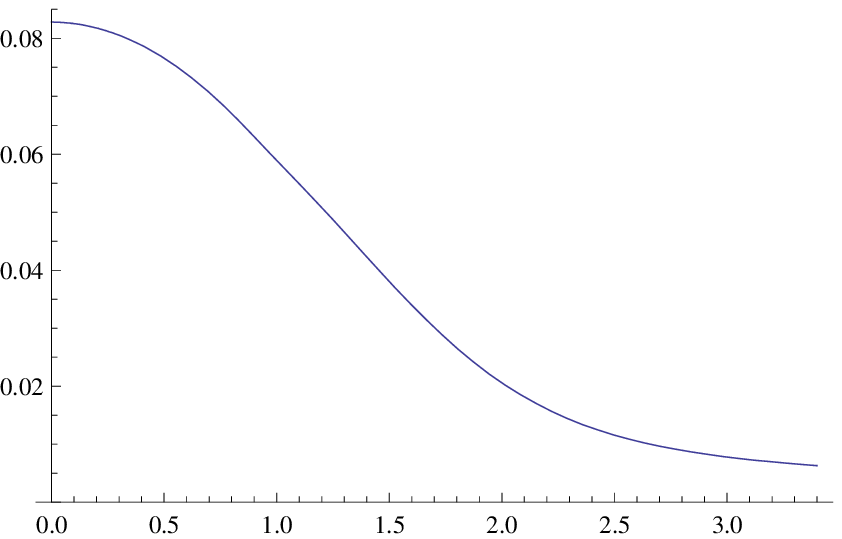}}
\centering{\includegraphics*[scale=.7, bb=0 0 240 160]{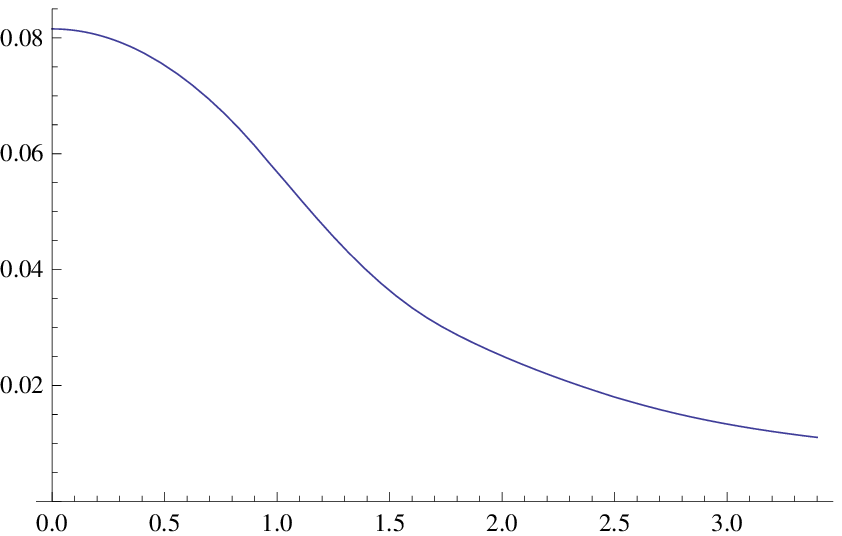}}
\caption{The individual contributions of $\G^{(0)},\ldots,\G^{(4)}$ to $\Phi^\hex$ }
\label{Figure19}
\end{figure}

\newpage

\section{Appendix 1}
\subsection{}\label{SubA1.1}
For $k\geqslant 1$:
\begin{equation}\label{A1.1}
\begin{split}
\int_{t_k}^{t_{k-1}} & \frac{2(a+\eps-qt)\, dt}{t^2+t+1}
=-2\left( \int_\gamma^{t_{k-1}} \frac{qt-a}{t^2+t+1}\ dt
-\int_\gamma^{t_k} \frac{qt-a}{t^2+t+1}\, dt \right) \\ & \hspace{3cm} +
\int_{t_k}^{t_{k-1}} \frac{2\eps}{t^2+t+1}\, dt
\\ & =\frac{Q-q}{2Q^2 q_{k-1} q_k (\gamma^2+\gamma+1)} \left(
\frac{q_{k+1}-2Q}{q_{k-1}} +\frac{q_{k+2}-2Q}{q_k}\right) + O
\left( \frac{1}{Qq^2 q_{k-1}^2} \right) .
\end{split}
\end{equation}

\subsection{}\label{SubA1.2} For $k\geqslant 1$ and $q_{k+1}\leqslant 2Q$:
\begin{equation}\label{A1.2}
\int_{\frac{a_k-\eps}{q_k}}^{t_{k-1}} \frac{w_{\CC_k}(t)\,
dt}{t^2+t+1} = \frac{(2Q-q_{k+1})^2}{8Q^2 q_{k-1}q_k^2
(\gamma^2+\gamma +1)} +O\left( \frac{1}{Qq q_k^3}\right) .
\end{equation}
\begin{equation}\label{A1.3}
\int_{\frac{a_k-\eps}{q_k}}^{t_{k-1}} \frac{2 (q_k t -a_k+\eps)\,
dt}{t^2+t+1} = \frac{(2Q-q_{k+1})^2}{4Q^2 q_{k-1}^2 q_k
(\gamma^2+\gamma+1)} +O\left( \frac{1}{Q^2 q q_k^2}\right).
\end{equation}
\begin{equation}\label{A1.4}
\begin{split}
\int_{\frac{a_k-\eps}{q_k}}^{t_{k-1}} & \frac{w_{\AA_0} (t)\,
dt}{t^2+t+1} = \int_{\frac{a_k-\eps}{q_k}}^{u_0} \frac{a+2\eps
-qt}{t^2+t+1}\, dt -\int_{t_{k-1}}^{u_0}
\frac{a+2\eps -qt}{t^2+t+1}\, dt \\
& =\frac{2Q-q_{k+1}}{4Q^2 q_{k-1}q_k (\gamma^2+\gamma +1)} \left(
\frac{q_k+q_{k+1}-2Q}{2q_k}+\frac{q_k-Q}{q_{k-1}}\right)+O\left(
\frac{1}{q^2 q_{k-1} q_k^2}\right) .
\end{split}
\end{equation}
\begin{equation}\label{A1.5}
\begin{split}
\int_{\gamma_{k+1}}^{\frac{a_k-\eps}{q_k}} & \frac{w_{\AA_0}(t)\,
dt}{t^2+t+1} =\int_{\gamma_{k+1}}^{u_0}
\frac{a+2\eps-qt}{t^2+t+1}\, dt -\int_{\frac{a_k-\eps}{q_k}}^{u_0}
\frac{a+2\eps-qt}{t^2+t+1}\, dt \\ & =\frac{2Q-q_{k+1}}{4Q^2 q_k
q_{k+1} (\gamma^2+\gamma+1)} \left( \frac{q_{k+1}-Q}{q_{k+1}}
+\frac{q_k+q_{k+1}-2Q}{2q_k}\right) +O \left( \frac{1}{q^2 q_k
q_{k+1}^2}\right) .
\end{split}
\end{equation}
\begin{equation}\label{A1.6}
\int_{\gamma_{k+1}}^{\frac{a_k-\eps}{q_k}} \frac{q_{k+1}t
-a_{k+1}}{t^2+t+1}\, dt =\frac{(2Q-q_{k+1})^2}{8Q^2 q_k^2 q_{k+1}
(\gamma^2+\gamma+1)} +O\left( \frac{1}{qq_k^2 q_{k+1}^2}\right).
\end{equation}
\begin{equation}\label{A1.7}
\int_{\gamma_{k+1}}^{\frac{a_k-\eps}{q_k}} \frac{2(a_k -\eps -q_k
t)\, dt}{t^2+t+1} =\frac{(2Q-q_{k+1})^2}{4Q^2 q_k q_{k+1}^2
(\gamma^2+\gamma +1)} +O\left( \frac{1}{qq_k^2 q_{k+1}^2}\right) .
\end{equation}
\begin{equation}\label{A1.8}
\int_{t_k}^{\gamma_{k+1}} \frac{w_{\BB_k}(t)\, dt}{t^2+t+1}  =
\frac{(q_{k+1}-Q)^2}{2Q^2 q_k q_{k+1}^2 (\gamma^2+\gamma +1)}
+O\left(\frac{1}{qq_k^2 q_{k+1}^2}\right) .
\end{equation}
\begin{equation}\label{A1.9}
\begin{split}
\int_{t_k}^{\gamma_{k+1}} \frac{w_{\AA_0} (t)-w_{\BB_k}
(t)}{t^2+t+1} \, dt & =\int_{t_k}^{\gamma_{k+1}}
\frac{a_{k+1}-q_{k+1}t}{t^2+t+1}\ dt \\ & =\frac{(q_{k+1}-Q)^2}{2Q^2
q_k^2 q_{k+1}(\gamma^2+\gamma +1)} +O\left( \frac{1}{qq_k^2
q_{k+1}^2}\right) .
\end{split}
\end{equation}

\subsection{}\label{SubA1.3}
For $k\geqslant 0$ and $q_{k+1}\leqslant 2Q-q$:
\begin{equation}\label{A1.10}
\begin{split}
\int_{t_k}^{\gamma_{k+1}} & \frac{2(qt -a-\eps)}{t^2+t+1}\ dt =
\int_{\frac{a+\eps}{q}}^{\gamma_{k+1}} \frac{2(qt
-a-\eps)}{t^2+t+1} \ dt -\int_{\frac{a+\eps}{q}}^{t_k} \frac{2(qt
-a-\eps)}{t^2+t+1} \ dt
\\ & =\frac{q_{k+1}-Q}{2Q^2 q_k q_{k+1} (\gamma^2+\gamma+1)}
\left( \frac{2Q-q_{k+1}}{q_{k+1}} +\frac{2Q-q_{k+2}}{q_k}\right)
+O\left( \frac{1}{q^3 q_{k+1}^2}\right) .
\end{split}
\end{equation}

\subsection{}\label{SubA1.4}
For  $k\geqslant 0$ and $2Q-q<q_{k+1}\leqslant 2Q$ it follows that
 $0\leqslant q_{k+2}-2Q\leqslant q$, $k=\Big\lfloor
\frac{2Q-q-q^\prime}{q}\Big\rfloor$, and
\begin{equation}\label{A1.11}
\begin{split}
\int_{t_k}^{\frac{a+\eps}{q}} \frac{w_{\AA_0}(t)\, dt}{t^2+t+1} & =
\int_{t_k}^{u_0} \frac{a+2\eps -qt}{t^2+t+1}\, dt -
\int_{\frac{a+\eps}{q}}^{u_0} \frac{a+2\eps -qt}{t^2+t+1}\, dt
\\ & =\frac{(q_{k+2}-2Q)(q_k+2q_{k+1}-2Q)}{8Q^2 qq_k^2
(\gamma^2 +\gamma+1)} +O \left( \frac{1}{q^2 q^\prime q_k^2}
\right) .
\end{split}
\end{equation}
\begin{equation}\label{A1.12}
\begin{split}
\int_{t_k}^{\frac{a+\eps}{q}} \frac{qt -a}{t^2+t+1}\, dt &
=\int_{\gamma}^{\frac{a+\eps}{q}} \frac{qt -a}{t^2+t+1}\, dt -
\int_{\gamma}^{t_k} \frac{qt -a}{t^2+t+1}\, dt \\ &
=\frac{(q_{k+2}-2Q)(q_{k-2}+2Q)}{8Q^2 qq_k^2 (\gamma^2+\gamma+1)}
+O\left( \frac{1}{Q^3 q^2} \right) .
\end{split}
\end{equation}
\begin{equation}\label{A1.13}
\int_{t_k}^{\frac{a+\eps}{q}} \frac{w_{\BB_k}(t)\, dt}{t^2+t+1}
=\frac{(q_{k+2}-2Q)^2}{8Q^2 q^2 q_k (\gamma^2+\gamma+1)}+O\left(
\frac{1}{Q^2 q q_k^2}\right) .
\end{equation}
\begin{equation}\label{A1.14}
\int_{t_k}^{\frac{a+\eps}{q}} \frac{2(a+\eps-qt)\, dt}{t^2+t+1}
=\frac{(q_{k+2}-2Q)^2}{4Q^2 qq_k^2 (\gamma^2+\gamma+1)} +O\left(
\frac{1}{Qqq^\prime q_k^2}\right) .
\end{equation}
\begin{equation}\label{A1.15}
\begin{split}
\int_{t_k}^{\frac{a+\eps}{q}} & \frac{w_{\CC_k}(t)\, dt}{t^2+t+1}
=\int_{t_k}^{t_{k-1}} \frac{a_{k-1}-2\eps-q_{k-1}t}{t^2+t+1}\, dt -
\int_{\frac{a+\eps}{q}}^{t_{k-1}}
\frac{a_{k-1}-2\eps -q_{k-1}t}{t^2+t+1}\, dt  \\
& =\frac{q_{k+2}-2Q}{4Q^2 qq_k (\gamma^2+\gamma +1)} \left(
\frac{Q-q}{q_k} +\frac{2Q-q_{k+1}}{2q}\right) +O\left(
\frac{1}{Qq q_k^3}\right)\quad \mbox{\rm if $k\geqslant 1.$}
\end{split}
\end{equation}
\begin{equation}\label{A1.16}
\begin{split}
\int_{t_0}^{\frac{a+\eps}{q}} \frac{w_{\CC_0}(t)}{t^2+t+1}\, dt & =
\int_{t_0}^{\gamma^\prime} \frac{a^\prime-q^\prime t}{t^2+t+1}\, dt -
\int_{\frac{a+\eps}{q}}^{\gamma^\prime} \frac{a^\prime -q^\prime t}{t^2+t+1}\, dt
\\ & \qquad - \left( \int_{t_0}^{u_0} \frac{a+2\eps-qt}{t^2+t+1}\, dt -
\int_{\frac{a+\eps}{q}}^{u_0} \frac{a+2\eps-qt}{t^2+t+1}\, dt \right) \\
& =\frac{q_2-2Q}{4Q^2 qq^\prime (\gamma^2+\gamma+1)} \left(
\frac{Q-q}{q^\prime}+\frac{2Q-q_1}{2q}\right) +
O\left(  \frac{1}{Q^2 qq^{\prime 2}}\right).
\end{split}
\end{equation}
\begin{equation}\label{A1.17}
\begin{split}
\int_{\frac{a+\eps}{q}}^{\gamma_{k+1}} \frac{w_{\BB_k}(t)\,
dt}{t^2+t+1} & =\int_{t_k}^{\gamma_{k+1}} \frac{q_k t
-a_k+2\eps}{t^2+t+1}\, dt -\int_{t_k}^{
\frac{a+\eps}{q}} \frac{q_k t -a_k+2\eps}{t^2+t+1}\, dt \\
& =\frac{2Q-q_{k+1}}{4Q^2 qq_{k+1}(\gamma^2+\gamma+1)} \left(
\frac{q_{k+1}-Q}{q_{k+1}}+\frac{q_{k+2}-2Q}{2q}\right)+O\left(
\frac{1}{Q^2 q q_k q_{k+1}}\right).
\end{split}
\end{equation}
\begin{equation}\label{A1.18}
\int_{\frac{a+\eps}{q}}^{\gamma_{k+1}}
\frac{a_{k+1}-q_{k+1}t}{t^2+t+1} \, dt =\frac{(2Q-q_{k+1})^2}{8Q^2
q^2 q_{k+1} (\gamma^2+\gamma+1)} +O\left( \frac{1}{q^3
q_{k+1}^2}\right).
\end{equation}
\begin{equation}\label{A1.19}
\int_{\frac{a+\eps}{q}}^{\gamma_{k+1}} \frac{2(qt -a-\eps)\,
dt}{t^2+t+1} =\frac{(2Q-q_{k+1})^2}{4Q^2 qq_{k+1}^2
(\gamma^2+\gamma+1)} +O\left( \frac{1}{Qq^2 q_{k+1}^2}\right).
\end{equation}
\begin{equation}\label{A1.20}
\int_{\frac{a+\eps}{q}}^{\gamma_{k+1}}
\frac{w_{\CC_k}(t)\, dt}{t^2+t+1} =\frac{(2Q-q_{k+1})^2 q_{k+3}}{8Q^2 q^2 q_{k+1}^2
(\gamma^2+\gamma+1)} + \begin{cases} \vspace{.2cm} O\Big( \frac{1}{q^3 q_{k-1}^2}\Big)
& \mbox{\rm if $k\geqslant 1,$} \\ O\Big( \frac{1}{q^3 q^{\prime 2}}\Big)
& \mbox{\rm if $k=0.$}
\end{cases}
\end{equation}

\section{Appendix 2}

\subsection{Estimates for the $\CC_O$ contribution when $(r,r_k)=(0,1)$}
Here $W^{(1)}_{\gamma,k,n}$ and $W^{(2)}_{\gamma,k}$ are as in \eqref{7.3} and \eqref{7.4}.
\subsubsection{}\label{SubSub11.1.1}
For $k\geqslant 1$ and $q_{k+1}\leqslant 2Q$:
\begin{equation}\label{A2.1}
\begin{split}
\int_{\gamma_{k+1}}^{t_{k-1}} & \frac{w_{\AA_0}(t)\, dt}{t^2+t+1}
=\int_{\gamma_{k+1}}^{u_0} \frac{a+2\eps -qt}{t^2+t+1}\, dt -
\int_{t_{k-1}}^{u_0} \frac{a+2\eps -qt}{t^2+t+1}\, dt \\
& =\frac{2Q-q_{k+1}}{2Q^2 q_{k-1}q_{k+1}(\gamma^2+\gamma+1)}
\left( \frac{q_{k+1}-Q}{q_{k+1}} +\frac{q_k -Q}{q_{k-1}} \right)
+O\left( \frac{1}{Q q^2 q_{k+1}^2}\right) .
\end{split}
\end{equation}
Summing as in \eqref{5.2} and employing Lemma \ref{L3.4} we infer
\begin{equation*}
\begin{split}
& \G_{I,Q}^{(2.1)}(\xi) =\sum_{\beta\in\{ -1,1\}}
\sum\limits_{\substack{1\leqslant q\leqslant Q
\\ q\equiv -\beta \hspace{-6pt}\pmod{3}}} \sum_{k=1}^\infty
\sum\limits_{\substack{x=3\tilde{x}\in q(1-I),\, (\tilde{x},q)=1 \\
y=q_k \in\II_{q,k},\, y\leqslant 2Q-q  \\ \tilde{x}y\equiv
\frac{\beta q+1}{3}\hspace{-6pt}\pmod{q} }} (y+q-\xi Q)_+
\wedge q \int_{\gamma_{k+1}}^{t_{k-1}}
\frac{w_{\AA_0}(t)}{t^2+t+1}\, dt
\\ & \approxeq \frac{c_I C(3)}{3} \hspace{-3pt} \sum_{\beta\in \{ -1,1\}}
\hspace{-3pt} \frac{\varphi(q)}{q} \int_Q^{2Q-q} \hspace{-5pt} \frac{2Q-q-y}{2Q^2 (y-q)(y+q)}
\, \left( \frac{y+q-Q}{y+q}+\frac{y-Q}{y-q}\right) \cdot
(y+q-\xi Q)_+ \wedge q \, dy.
\end{split}
\end{equation*}
Estimate \eqref{7.1} follows applying Lemma \ref{L3.2} and making
the change of variable $(q,y)=(Qu,Qw)$.

\subsubsection{}\label{SubSub11.1.2}
For $k\geqslant 1$ and $q_{k+1}>2Q$:
\begin{equation}\label{A2.2}
\begin{split}
\int_{t_k}^{t_{k-1}} & \frac{w_{\AA_0}(t) -w_{\BB_k}(t)}{t^2+t+1}\,
dt = \int_{t_k}^{\gamma_{k+1}} \frac{a_{k+1}-q_{k+1}t}{t^2+t+1}\,
dt - \int_{t_{k-1}}^{\gamma_{k+1}}
\frac{a_{k+1}-q_{k+1}t}{t^2+t+1} \, dt \\
& =\frac{Q-q}{2Q^2 q_{k-1}q_k (\gamma^2+\gamma +1)} \left(
\frac{q_{k+1}-Q}{q_k} +\frac{q_{k+1}-2Q}{q_{k-1}} \right) +O\left(
\frac{1}{Q^2 qq_k q_{k+1}}\right) .
\end{split}
\end{equation}

\subsubsection{}\label{SubSub11.1.3}
When $n\geqslant 1$ and $n(Q-q)\leqslant q_k \leqslant
(n+1)(Q-q)$:
\begin{equation}\label{A2.3}
\begin{split}
\int_{t_k}^{\gamma_{k+n}} \frac{W_{\gamma,k,n}^{(1)}(t)\,
dt}{t^2+t+1} & =(q_{k+n+1}-\xi Q) \int_{t_k}^{\gamma_{k+n}}
\frac{a_{k+n}-q_{k+n}t}{t^2+t+1}\, dt \\ & =\frac{(q_{k+n}-nQ)^2 (q_{k+n+1}-\xi Q)}{2Q^2 q_k^2
q_{k+n}(\gamma^2+\gamma+1)} + O_\xi \left( \frac{1}{qq_k^2
q_{k+n}}\right) .
\end{split}
\end{equation}
Here and below $\ll_\xi$ means uniformly for $\xi$ in compact subsets of $[0,\infty)$.

\subsubsection{}\label{SubSub11.1.4}
When $n\geqslant 1$ and $(n+1)(Q-q)\leqslant q_{k}
\leqslant Q+n(Q-q)$:
\begin{equation}\label{A2.4}
\begin{split}
\int_{\gamma_{k+n+1}}^{\gamma_{k+n}}
\frac{W^{(1)}_{\gamma,k,n}(t)\, dt}{t^2+t+1} & = (q_{k+n+1}-\xi
Q) \int_{\gamma_{k+n+1}}^{\gamma_{k+n}}
\frac{a_{k+n}-q_{k+n}t}{t^2+t+1}\, dt \\ &
=\frac{q_{k+n+1}-\xi Q}{2q_{k+n} q_{k+n+1}^2
(\gamma^2+\gamma+1)} +O_\xi \left( \frac{1}{qq_{k+n}
q_{k+n+1}^2}\right).
\end{split}
\end{equation}

\subsubsection{}\label{SubSub11.1.5}
When $k\geqslant 1$ and $Q+n(Q-q) \leqslant q_k
\leqslant Q+(n+1)(Q-q)$:
{\small \begin{equation}\label{A2.5}
\begin{split}
& \int_{\gamma_{k+n+1}}^{t_{k-1}} \frac{W^{(1)}_{\gamma,k,n}(t)\,
dt}{t^2+t+1} = (q_{k+n+1}-\xi Q) \left(
\int_{\gamma_{k+n+1}}^{\gamma_{k+n}}
\frac{a_{k+n}-q_{k+n}t}{t^2+t+1}\, dt
-\int_{t_{k-1}}^{\gamma_{k+n}} \frac{a_{k+n}-q_{k+n}t}{t^2+t+1}\,
dt \right) \\
& =\frac{\big( (n+2)Q-q_{k+n+1}\big) (q_{k+n+1}-\xi Q)}{2Q^2
q_{k-1} q_{k+n+1}(\gamma^2+\gamma+1)} \left(
\frac{Q}{q_{k+n+1}}+\frac{q_{k+n}-(n+1)Q}{q_{k-1}}\right)
+O_\xi \left( \frac{1}{qq_{k-1} q_{k+n}q_{k+n+1}}\right).
\end{split}
\end{equation}}

\subsubsection{}\label{SubSub11.1.6}
When $n\geqslant 1$ and $(n+1)(Q-q)\leqslant q_{k}
\leqslant Q+(n+1)(Q-q)$:
\begin{equation*} \begin{split}
\int_{t_k}^{\gamma_{k+n+1}} \frac{dt}{t^2+t+1} & = \frac{q_k
-(n+1)(Q-q)}{Qq_k q_{k+n+1}(\gamma^2+\gamma+1)} +O\left(
\frac{1}{Qq q_{k+n+1}^2}\right) ,\\
\int_{t_k}^{\gamma_{n+k+1}} \frac{w_{\BB_k}(t)
+w_{\CC_k}(t)}{t^2+t+1}\, dt & = \int_\gamma^{\gamma_{k+n+1}}
\frac{qt -a}{t^2+t+1}\, dt -
\int_\gamma^{t_k} \frac{qt -a}{t^2+t+1}\, dt \\
& =\frac{q_k -(n+1)(Q-q)}{2Q^2 q_k q_{k+n+1} (\gamma^2+\gamma+1)}
\left( \frac{Q}{q_{n+k+1}} + \frac{Q-q}{q_k}\right)+O\left(
\frac{1}{Q^2 q_{n+k+1}^3}\right),
\end{split}
\end{equation*}
yielding
{\small \begin{equation}\label{A2.6}
\hspace{-5pt} \int_{t_k}^{\gamma_{k+n+1}}
\hspace{-2pt}\frac{W_{\gamma,k}^{(2)}(t)\, dt}{t^2+t+1}
=\frac{q_{k+n+1}-(n+1)Q}{2Qq_k q_{k+n+1} (\gamma^2+\gamma+1)}
\left( \frac{q_{k+n+1}-\xi Q}{q_{k+n+1}}+\frac{q_k -\xi
(Q-q)}{q_k}\right) +O_\xi \left( \frac{1}{Qq q_{k+n+1}^2}\right).
\end{equation}}

\subsubsection{}\label{SubSub11.1.7}
When $q_k \geqslant Q+(n+1)(Q-q)$ formulas \eqref{5.8}
and \eqref{5.9} yield (here $k\geqslant 1$)
\begin{equation}\label{A2.7}
\int_{t_k}^{t_{k-1}} \frac{W_{\gamma,k}^{(2)} (t)\, dt}{t^2+t+1}
=\frac{Q-q}{2Qq_{k-1}q_k(\gamma^2+\gamma +1)} \left(
2-\xi\bigg( \frac{Q-q}{q_{k-1}}+\frac{Q-q}{q_k}\bigg) \right)
+O_\xi \left( \frac{1}{q q_{k-1}^2 q_k}\right).
\end{equation}

\subsection{Estimates for the $\CC_\leftarrow$ contribution when $(r,r_k)=(0,1)$}
Here $\lambda_{k,n}$ is as in \eqref{7.7}.

\subsubsection{}\label{SubSub11.2.1} $(n-1)(Q-q)\leqslant q_{k}\leqslant n(Q-q)$. In this
case $n\geqslant 2$. When $k\geqslant 1$ we have
\begin{equation}\label{A2.8}
\begin{split}
& \int_{t_k}^{\lambda_{k,n}} \frac{w_{\CC_k}(t)\, dt}{t^2+t+1} =
\int_{t_k}^{t_{k-1}} \frac{a_{k-1}-2\eps-q_{k-1}t}{t^2+t+1}\, dt
-\int_{\lambda_{k,n}}^{t_{k-1}} \frac{a_{k-1}-2\eps
- q_{k-1}t}{t^2+t+1}\, dt \\
& =\frac{q_{k}- (n-1)(Q-q)}{2Q^2 q_k (q_{k}+q_{k+n-1})
(\gamma^2+\gamma+1)} \left( \frac{Q-q}{q_k} +
\frac{(n+1)Q-q_{k+n}}{q_{k}+q_{k+n-1}} \right)
+O\left(\frac{1}{qq_{k-1}^2 q_k^2}\right).
\end{split}
\end{equation}
\begin{equation}\label{A2.9}
\begin{split}
\int_{t_k}^{\lambda_{k,n}} \frac{w_{\BB_k} (t)\, dt}{t^2+t+1} & =
\int_{t_k}^{\lambda_{k,n}} \frac{q_k t -a_k+2\eps}{t^2+t+1}\, dt
\\ & =\frac{\big( q_{k}-(n-1)(Q-q)\big)^2}{2Q^2 q_k
(q_{k}+q_{k+n-1})^2 (\gamma^2+\gamma+1)} +O\left( \frac{1}{q q_k^2
q_{k+n}^2}\right).
\end{split}
\end{equation}
\begin{equation}\label{A2.10}
\begin{split}
& \int_{t_k}^{\lambda_{k,n}} \frac{q_{k+n} t -a_{k+n}}{t^2+t+1}\,
dt =\int_{\gamma_{k+n}}^{\lambda_{k,n}} \frac{q_{k+n}t
-a_{k+n}}{t^2+t+1}\, dt -\int_{\gamma_{k+n}}^{t_k}
\frac{q_{k+n}t -a_{k+n}}{t^2+t+1}\, dt \\
&  =\frac{q_{k}-(n-1)(Q-q)}{2Q^2 q_k
(q_{k}+q_{k+n-1})(\gamma^2+\gamma+1)} \left(
\frac{(n+1)Q-q_{k+n}}{q_{k} + q_{k+n-1}} +
\frac{n(Q-q)-q_{k}}{q_k} \right) +O \left( \frac{1}{qq_k^3
q_{k+n}} \right) .
\end{split}
\end{equation}
When $k=0$ formulas \eqref{A2.9} and \eqref{A2.10} still hold. The
main term in \eqref{A2.8} remains the same but we need a more
careful estimate to control the error term, getting
\begin{equation}\label{A2.11}
\begin{split}
\int_{t_0}^{\lambda_{0,n}} & \frac{w_{\CC_0}(t)\, dt}{t^2+t+1} =
\int_{\gamma}^{\lambda_{0,n}} \frac{qt -a}{t^2+t+1}\, dt -
\int_{\gamma}^{t_0} \frac{qt -a}{t^2+t+1}\, dt
-\int_{t_0}^{\lambda_{0,n}} \frac{q^\prime t -a^\prime+2\eps}{t^2+t+1}\, dt \\
& =\frac{q_{n-1}-(n-1)Q}{2Q^2 q^\prime (q^\prime
+q_{n-1})(\gamma^2+\gamma+1)} \left( \frac{(n+1)Q-q_n}{q^\prime
+q_{n-1}} + \frac{Q-q}{q^\prime} \right) +O\left( \frac{1}{Qq^2
q_{n-1}^2} \right).
\end{split}
\end{equation}
Since $n$ only takes one value $n=\left\lfloor
\frac{Q+q_{k-1}}{Q-q}\right\rfloor$ for fixed $k$, the error terms
in \eqref{A2.8}-\eqref{A2.11} do not play a role in the final
asymptotic formula.

\subsubsection{}\label{SubSub11.2.2} When $k\geqslant 1$ and $n(Q-q)\leqslant q_k \leqslant
Q+n(Q-q)$:
\begin{equation}\label{A2.12}
\begin{split}
& \int_{\lambda_{k,n+1}}^{\lambda_{k,n}} \frac{w_{\CC_k} (t)\
dt}{t^2+t+1} = \int_{\lambda_{k,n+1}}^{t_{k-1}}
\frac{a_{k-1}-2\eps-q_{k-1} t}{t^2+t+1}\ dt
-\int_{\lambda_{k,n}}^{t_{k-1}}
\frac{a_{k-1}-2\eps-q_{k-1}t}{t^2+t+1}\ dt  \\
& =\frac{2Q-q}{2Q^2
(q_{k}+q_{k+n-1})(q_k+q_{k+n})(\gamma^2+\gamma +1)} \left(
\frac{(n+2)Q-q_{k+n+1}}{q_k+q_{k+n}} +
\frac{(n+1)Q-q_{k+n}}{q_{k}+q_{k+n-1}} \right) \\ & \qquad
+ O \left( \frac{1}{(n-1)^3 q^2 q_k^3}\right) .
\end{split}
\end{equation}
\begin{equation}\label{A2.13}
\begin{split}
& \int_{\lambda_{k,n+1}}^{\lambda_{k,n}} \frac{w_{\BB_k}(t)\,
dt}{t^2+t+1} = \int_{t_k}^{\lambda_{k,n}} \frac{q_k
t-a_k+2\eps}{t^2+t+1}\, dt - \int_{t_k}^{\lambda_{k,n+1}} \frac{q_k
t - a_k+2\eps}{t^2+t+1} \, dt
\\ & \quad =\frac{2Q-q}{2Q^2 (q_{k}+q_{k+n-1})(q_k+q_{k+n})(\gamma^2+\gamma+1)}
\left( \frac{q_{k+n-1}-(n-1)Q}{q_{k} +q_{k+n-1}} +
\frac{q_{k+n}-nQ}{q_k+q_{k+n}} \right) \\ & \qquad + O \left( \frac{1}{q q_k
q_{k+n}^3}\right) .
\end{split}
\end{equation}
\begin{equation}\label{A2.14}
\int_{\lambda_{k,n+1}}^{\gamma_{k+n}}
\frac{a_{k+n}-q_{k+n}t}{t^2+t+1}\, dt =\frac{(q_{k+n}-nQ)^2}{2Q^2
q_{k+n} (q_k+q_{k+n})^2(\gamma^2+\gamma+1)} + O\left(
\frac{1}{qq_{k+n}^4} \right) .
\end{equation}
\begin{equation}\label{A2.15}
\int_{\gamma_{k+n}}^{\lambda_{k,n}} \frac{q_{k+n}t -
a_{k+n}}{t^2+t+1}\, dt =\frac{\big( (n+1)Q-q_{k+n}\big)^2}{2Q^2
q_{k+n} (q_{k-1}+q_{k+n})^2 (\gamma^2+\gamma+1)} +O \left(
\frac{1}{qq_{k+n}^4} \right) .
\end{equation}

\subsubsection{}\label{SubSub11.2.3} When $k=0$ and $n\geqslant 2$ we have
$t_0<\lambda_{0,n+1}<\gamma_n <\lambda_{0,n}\leqslant t_{k-1}$.
Analog formulas as \eqref{A2.12}-\eqref{A2.15} hold, with same
main terms and
\begin{equation}\label{A2.16}
\begin{split}
& \int_{\lambda_{0,n+1}}^{\lambda_{0,n}} \frac{w_{\CC_0}(t)\,
dt}{t^2+t+1} =\int_{\gamma}^{\lambda_{0,n}} \frac{qt -a}{t^2+t+1}\,
dt - \int_{\gamma}^{\lambda_{0,n+1}} \frac{qt -a}{t^2+t+1}\, dt \\
& \hspace{3cm} -\left( \int_{t_0}^{\lambda_{0,n}} \frac{q^\prime t
-a^\prime+2\eps}{t^2+t+1}\, dt - \int_{t_0}^{\lambda_{0,n+1}}
\frac{q^\prime t -a^\prime +2\eps}{t^2+t+1}\, dt \right) \\
& =\frac{2Q-q}{2Q^2 (q^\prime +q_{n-1}) (q^\prime+q_n) (\gamma^2
+\gamma +1)} \left( \frac{(n+2)Q-q_{n+1}}{q^\prime +q_n} +
\frac{(n+1)Q-q_n}{q^\prime+q_{n-1}} \right)+O \left( \frac{1}{q
q^{\prime } q_n^3}\right) .
\end{split}
\end{equation}

\subsubsection{}\label{SubSub11.2.4} When $k=0$ and $n=1$ we have $t_0 \leqslant
\lambda_{0,2} <\gamma_1 \leqslant \lambda_{0,1}$. We have
\begin{equation}\label{A2.17}
\begin{split}
\int_{\lambda_{0,2}}^{\gamma_1} \frac{w_{\BB_0}(t)\, dt}{t^2+t+1}
 & = \int_{t_0}^{\gamma_1} \frac{q^\prime t -a^\prime
+2\eps}{t^2+t+1}\, dt -\int_{t_0}^{\lambda_{0,2}} \frac{q^\prime
t -a^\prime +2\eps}{t^2+t+1}\, dt \\
& =\frac{(q_1-Q)^2 (q^\prime +2q_1)}{2Q^2 q_1^2 (q^\prime +q_1)^2
(\gamma^2+\gamma+1)} +O \left( \frac{1}{Q^3 qq^\prime}\right) .
\end{split}
\end{equation}
\begin{equation}\label{A2.18}
\int_{\lambda_{0,2}}^{\gamma_1} \frac{a_1-q_1 t}{t^2+t+1}\, dt
=\frac{(q_1-Q)^2}{2Q^2 q_1 (q^\prime +q_1)^2 (\gamma^2+\gamma +1)}
+O\left( \frac{1}{Q^4 q}\right) .
\end{equation}
\begin{equation}\label{A2.19}
\begin{split}
\int_{\lambda_{0,2}}^{\gamma_1} & \frac{w_{\CC_0}(t)\, dt}{t^2+t+1}
= \int_{\lambda_{0,2}}^{\gamma^\prime}
\frac{a^\prime-q^\prime t}{t^2 +t+1}\, dt
-\int_{\gamma_1}^{\gamma^\prime} \frac{a^\prime -q^\prime t}{t^2+t+1}\, dt  -
\int_{\lambda_{0,2}}^{\frac{a+2\eps}{q}}
\frac{a+2\eps-qt}{t^2+t+1}\, dt \\ & \qquad +
\int_{\gamma_1}^{\frac{a+2\eps}{q}} \frac{a+2\eps-qt}{t^2+t+1}\, dt  \\
& =\frac{q_1-Q}{2Q^2 q_1 (q^\prime +q_1)(\gamma^2+\gamma+1)}
\left( \frac{3Q-q_2}{q^\prime+q_1} +\frac{2Q-q_1}{q_1} \right)
+O\left( \frac{1}{Q^3 qq^\prime}\right) .
\end{split}
\end{equation}

\subsubsection{}\label{SubSub11.2.5} When $Q+n(Q-q) \leqslant q_k \leqslant Q+(n+1)(Q-q)$:
\begin{equation}\label{A2.20}
\int_{\lambda_{k,n+1}}^{t_{k-1}} \frac{w_{\CC_k} (t)\, dt}{t^2+t+1}
= \frac{\big( (n+2)Q-q_{k+n+1}\big)^2}{2Q^2 q_{k-1}
(q_k+q_{k+n})^2 (\gamma^2+\gamma+1)} +O\left( \frac{1}{Qqq_{k-1}
q_{k+n}^2}\right) .
\end{equation}
\begin{equation}\label{A2.21}
\begin{split}
\int_{\lambda_{k,n+1}}^{t_{k-1}} & \frac{w_{\BB_k}(t)\,
dt}{t^2+t+1} = \int_{t_k}^{t_{k-1}} \frac{q_k
t-a_k+2\eps}{t^2+t+1}\, dt - \int_{t_k}^{\lambda_{k,n+1}} \frac{q_k
t - a_k +2\eps}{t^2+t+1} \, dt \\
& =\frac{(n+2)Q-q_{k+n+1}}{2Q^2 q_{k-1} (q_k+q_{k+n})
(\gamma^2+\gamma+1)} \left( \frac{Q-q}{q_{k-1}}
+\frac{q_{k+n}-nQ}{q_k+q_{k+n}} \right) + O \left( \frac{1}{Q^2 q
q_k^2}\right) .
\end{split}
\end{equation}
\begin{equation}\label{A2.22}
\begin{split}
& \int_{\lambda_{k,n+1}}^{t_{k-1}}
\frac{a_{k+n}-q_{k+n}t}{t^2+t+1}\, dt =
\int_{\lambda_{k,n+1}}^{\gamma_{k+n}}
\frac{a_{k+n}-q_{k+n}t}{t^2+t+1}\, dt -
\int_{t_{k-1}}^{\gamma_{k+n}} \frac{a_{k+n}-q_{k+n}t}{t^2+t+1}\, dt
\\ & =\frac{(n+2)Q-q_{k+n+1}}{2Q^2 q_{k-1}
(q_k+q_{k+n}) (\gamma^2+\gamma+1)} \left(
\frac{q_{k+n}-nQ}{q_k+q_{k+n}} +\frac{q_{k+n}-(n+1)Q}{q_{k-1}}
\right) + O \left( \frac{1}{Q^2 q q_{k+n}^2} \right).
\end{split}
\end{equation}

\section{Appendix 3}

\subsection{Estimates for the $\CC_O$ contribution when $(r,r_k)=(1,0)$}\label{SubA3.1}
Here $\lambda_{k,n}$ is as in \eqref{8.3} and $W_{\gamma,k,n}^{(1)}$ and $W_{\gamma,k}^{(2)}$
as in \eqref{8.4}.

\subsubsection{}\label{SubSub12.1.1} For $k\geqslant 1$ and $q_{k+1}\leqslant 2Q$:
\begin{equation}\label{A3.1}
\int_{\gamma_{k+1}}^{t_{k-1}} \frac{w_{\CC_k}(t)\, dt}{t^2+t+1} =
\frac{(2Q-q_{k+1})^2}{2Q^ 2 q_{k-1}q_{k+1}^2 (\gamma^2+\gamma+1)}
+O\left( \frac{1}{qq_{k-1}^2 q_k^2}\right) .
\end{equation}
\begin{equation}\label{A3.2}
\begin{split}
\int_{\gamma_{k+1}}^{t_{k-1}} \frac{w_{\BB_k}(t)
-w_{\AA_0}(t)}{t^2+t+1}\, dt & =\int_{\gamma_{k+1}}^{t_{k-1}}
\frac{q_{k+1}t -a_{k+1}}{t^2+t+1}\, dt \\ & =\frac{(2Q-q_{k+1})^2}{2Q^2 q_{k-1}^2
q_{k+1}(\gamma^2+\gamma+1)} +O \left( \frac{1}{qq_{k-1}^3
q_{k+1}}\right).
\end{split}
\end{equation}

\subsubsection{}\label{SubSub12.1.2} For $k\geqslant 1$ and $n(q_k-Q)\leqslant Q-q\leqslant
(N+1)(q_k-Q)$:
\begin{equation}\label{A3.3}
\int_{\lambda_{k,n}}^{t_{k-1}} \frac{(q+nq_k)t
-(a+na_k)}{t^2+t+1}\, dt =\frac{\big( (n+1)Q-(q+nq_k)\big)^2}{2Q^2
q_{k-1}^2 (q+nq_k)(\gamma^2+\gamma+1)} +O\left(
\frac{1}{qq_{k-1}^3 (q+nq_k)}\right) .
\end{equation}

\subsubsection{}\label{SubSub12.1.3} For $k\geqslant 1$ and $(n+1)(q_k-Q)\leqslant Q-q$:
\begin{equation}\label{A3.4}
\int_{\lambda_{k,n}}^{\lambda_{k,n+1}}
\frac{W^{(1)}_{\gamma,k,n}(t)\ dt}{t^2+t+1}
=\frac{q_{k+1}+nq_k-\xi Q}{2(q+nq_k)(q_{k+1}+nq_k)^2 (\gamma^2
+\gamma+1)} +O_\xi \left( \frac{1}{qq_k (q+nq_k) (q_{k+1}+nq_k)}
\right) .
\end{equation}
\begin{equation}\label{A3.5}
\int_{\lambda_{k,n+1}}^{t_{k-1}} \frac{dt}{t^2+t+1} =
\frac{(n+2)Q-q_{k+1}-nq_k}{Qq_{k-1} (q_{k+1}+nq_k) (\gamma^2
+\gamma+1)} +O\left( \frac{1}{qq_{k-1}^2 q_k} \right) .
\end{equation}
\begin{equation}\label{A3.6}
\begin{split}
\int_{\lambda_{k,n+1}}^{t_{k-1}} & \frac{w_{\AA_0}(t)
+w_{\CC_k}(t)}{t^2+t+1}\, dt  = \int_{\lambda_{k,n+1}}^{\gamma_k}
\frac{a_k-q_k t}{t^2+t+1}\,
dt - \int_{t_{k-1}}^{\gamma_k} \frac{a_k-q_k t}{t^2+t+1}\, dt \\
& =\frac{(n+2)Q-q_{k+1}-nq_k}{2Q^2 q_{k-1} (q_{k+1}+nq_k)(\gamma^2
+\gamma +1)} \left( \frac{Q}{q_{k+1}+nq_k}
+\frac{q_k-Q}{q_{k-1}}\right) +O \left( \frac{1}{q_{k-1}^3
q_{k+1}^2} \right) .
\end{split}
\end{equation}
\begin{equation}\label{A3.7}
\begin{split}
\int_{\lambda_{k,n+1}}^{t_{k-1}}
\frac{W^{(2)}_{\gamma,k}(t)}{t^2+t+1}\, dt & =\frac{(n+2)Q-q_{k+1}-nq_k}{2Q
q_{k-1}(q_{k+1}+nq_k)(\gamma^2 +\gamma +1)} \\ & \qquad \cdot \left(
\frac{q_{k+1}+nq_k -\xi Q}{q_{k+1}+nq_k}
+\frac{q_{k-1}-\xi (q_k-Q)}{q_{k-1}} \right) +O_\xi \left(
\frac{1}{qq_{k-1}^3} \right) .
\end{split}
\end{equation}

\subsection{Estimates for the $\CC_\leftarrow$ contribution when $(r,r_k)=(1,0)$}\label{AubA5.2}
Here $\mu_{k,n}$ and $\nu_{k,n}$ are as in \eqref{8.7}.

\subsubsection{}\label{SubSub12.2.1} For $q_2 >2Q$:
\begin{equation}\label{A3.8}
\begin{split}
\int_{\frac{a+\eps}{q}}^{\gamma_1} \frac{w_{\CC_0}(t)\,
dt}{t^2+t+1} &  = - \left( \int_{t_0}^{\gamma_1} \frac{q^\prime t-
a^\prime + 2\eps}{t^2+t+1}\, dt -\int_{t_0}^{\frac{a+\eps}{q}}
\frac{q^\prime t -a^\prime+2\eps}{t^2+t+1}\, dt \right)
\\ & \qquad \qquad +\left( \int_{\gamma}^{\gamma_1} \frac{qt -a}{t^2+t+1}\,
dt -\int_{\gamma}^{\frac{a+\eps}{q}} \frac{qt -a}{t^2+t+1}\, dt
\right) \\ & =\frac{(2Q-q_1)^2 q_3}{8Q^2 q^2 q_1^2 (\gamma^2
+\gamma+1)} +O\left( \frac{1}{Q^2 q^2 q^\prime}\right) .
\end{split}
\end{equation}
\begin{equation}\label{A3.9}
\begin{split}
\int_{\frac{a+\eps}{q}}^{\gamma_1} \frac{w_{\AA_0}(t)\,
dt}{t^2+t+1} & = \int_{\frac{a+\eps}{q}}^{u_0} \frac{w_{\AA_0}(t)\,
dt}{t^2+t+1}\, dt - \int_{\gamma}^{u_0} \frac{w_{\AA_0}(t)\,
dt}{t^2+t+1} \\ & =\frac{(2Q-q_1)(3q_1-2Q)}{8Q^2 qq_1^2 (\gamma^2
+\gamma +1)} +O\left( \frac{1}{Q^3 q^2}\right) .
\end{split}
\end{equation}
\begin{equation}\label{A3.10}
\int_{\frac{a+\eps}{q}}^{\gamma_1} \frac{a_1-q_1 t}{t^2+t+1}\, dt
=\frac{(2Q-q_1)^2}{8Q^2 q^2 q_1 (\gamma^2 +\gamma +1)} +O\left(
\frac{1}{q^3 Q^2}\right) .
\end{equation}

\subsubsection{}\label{SubSub12.2.2} When $n=1$ and $\frac{2(Q-q)}{n+1} \leqslant q_k
\leqslant \frac{2(Q-q)}{n} \wedge \Big( Q+\frac{Q-q}{n+1}\Big)$:
\begin{equation}\label{A3.11}
\begin{split}
\int_{t_k}^{\mu_{k,n}} & \frac{w_{\AA_0} (t)\, dt}{t^2+t+1} =
\int_{t_k}^{u_0} \frac{a+2\eps-qt}{t^2+t+1}\, dt
- \int_{\mu_{k,n}}^{u_0} \frac{a+2\eps -qt}{t^2+t+1}\, dt \\
& =\frac{(n+1)q_k -2(Q-q)}{2Q^2 q_k (2q+nq_k)(\gamma^2 +\gamma
+1)} \left( \frac{q_{k+1}-Q}{q_k} +\frac{q+nq_k -nQ}{2q+nq_k}
\right) + O\left( \frac{1}{Qq^2 q_k^2}\right) .
\end{split}
\end{equation}
\begin{equation}\label{A3.12}
\begin{split}
& \int_{t_k}^{\mu_{k,n}} \frac{w_{\CC_k}(t)\, dt}{t^2+t+1}
=\int_{t_k}^{t_{k-1}} \frac{a_{k-1}-2\eps-q_{k-1}t}{t^2+t+1}\, dt -
\int_{\mu_{k,n}}^{t_{k-1}} \frac{a_{k-1}-2\eps
- q_{k-1}t}{t^2+t+1}\, dt \\
& =\frac{(n+1)q_k -2(Q-q)}{2Q^2 q_k (2q+nq_k)(\gamma^2 +\gamma+1)}
\left( \frac{Q-q}{q_k} + \frac{(n+2)Q-q-(n+1)q_k}{2q+nq_k} \right)
+ O\left( \frac{1}{Qqq_{k-1} q_k^2}\right) .
\end{split}
\end{equation}
\begin{equation}\label{A3.13}
\int_{t_k}^{\nu_{k,n}} \frac{a+na_k -(q+nq_k)t}{t^2+t+1}\, dt
=\frac{(q+nq_k-Q)^2}{2Q^2 q_k^2 (q+nq_k) (\gamma^2 +\gamma +1)}
+O\left( \frac{1}{qq_k^4}\right) .
\end{equation}
\begin{equation}\label{A3.14}
\int_{\nu_{k,n}}^{\mu_{k,n}} \frac{(q+nq_k)t -(a+na_k)}{t^2+t+1}\,
dt =\frac{(q+nq_k -Q)^2}{2Q^2 (q+nq_k) (2q+nq_k)^2 (\gamma^2
+\gamma +1)} + O\left( \frac{1}{qq_k^4} \right) .
\end{equation}

\subsubsection{}\label{SubSub12.2.3} When $\frac{2(Q-q)}{n}\leqslant q_k \leqslant
Q+\frac{Q-q}{n+1}$, $n,k\geqslant 1$:

For $n\geqslant 2$ we have
\begin{equation}\label{A3.15}
\begin{split}
& \int_{\mu_{k,n-1}}^{\mu_{k,n}} \frac{w_{\AA_0}(t)\, dt}{t^2+t+1}
 = \int_{\mu_{k,n-1}}^{u_0} \frac{a+2\eps-qt}{t^2+t+1}\, dt -
\int_{\mu_{k,n}}^{u_0} \frac{a+2\eps -qt}{t^2+t+1}\, dt \\
& =\frac{2Q-q_k}{2Q^2 (2q+nq_k)\big( 2q+(n-1)q_k\big)(\gamma^2
+\gamma +1)} \left( \frac{q+(n-1)(q_k-Q)}{2q+(n-1)q_k}
+\frac{q+n(q_k-Q)}{2q+nq_k} \right)  \\ & \qquad  +O \left(
\frac{1}{(n-1)^3 q^2 q_k^3} \right) .
\end{split}
\end{equation}
For $n=1$ the error can be improved (since $2Q-q_k \leqslant 2q$ and $\mu_{k,1}
-\gamma \leqslant \frac{1}{qq_k}$) as follows:
\begin{equation}\label{A3.16}
\begin{split}
\int_{\mu_{k,0}}^{\mu_{k,1}} & \frac{w_{\AA_0}(t)\, dt}{t^2+t+1} =
- \int_{\mu_{k,0}}^{\mu_{k,1}} \frac{qt -a-\eps}{t^2+t+1}\, dt
+ \int_{\mu_{k,0}}^{\mu_{k,1}} \frac{\eps \, dt}{t^2+t+1} \\
& =\frac{2Q-q_k}{2Q^2 (2q+q_k)(2q)(\gamma^2 +\gamma +1)} \left(
\frac{q}{2q} + \frac{q+q_k-Q}{2q+q_k}\right) +O \left(
\frac{1}{Q^2 q q_k^2}\right) .
\end{split}
\end{equation}
Since $0\leqslant (n+1)Q-q-nq_k \leqslant Q-q <q^\prime$ and
$0 \leqslant nq_k -(n+1)Q \leqslant 2Q$ we find:
\begin{equation}\label{A3.17}
\begin{split}
& \int_{\mu_{k,n-1}}^{\mu_{k,n}} \frac{w_{\CC_k}(t)\ dt}{t^2+t+1}
 = \int_{\mu_{k,n-1}}^{t_{k-1}}
\frac{a_{k-1}-2\eps-q_{k-1}t}{t^2+t+1}\, dt -
\int_{\mu_{k,n}}^{t_{k-1}} \frac{a_{k-1}-2\eps -q_{k-1}t
}{t^2+t+1}\, dt
\\ & \qquad \qquad \qquad \quad =\frac{2Q-q_k}{2Q^2 (2q+nq_k) \big( 2q+(n-1)q_k\big)
(\gamma^2 +\gamma +1)} \\ & \cdot \left( \frac{Q-q-n(q_k-Q)}{2q+(n-1)q_k}
+\frac{Q-q-(n+1)(q_k-Q)}{2q+nq_k}\right) +O\left( \frac{1}{qq_{k-1}^2 \big( q+(n-1)q_k\big)^2} \right) .
\end{split}
\end{equation}
\begin{equation}\label{A3.18}
\begin{split}
\int_{\mu_{k,n-1}}^{\nu_{k,n}} \frac{a+na_k -(q+nq_k)t}{t^2+t+1} \,
dt & = \frac{\big( Q-q-n(q_k-Q)\big)^2}{2Q^2 (q+nq_k) \big(
2q+(n-1)q_k\big)^2 (\gamma^2 +\gamma+1)} \\ & \qquad + O\left(
\frac{1}{n qq_k^2 \big( q+(n-1)q_k\big)^2} \right) .
\end{split}
\end{equation}
\begin{equation}\label{A3.19}
\int_{\nu_{k,n}}^{\mu_{k,n}} \frac{(q+nq_k)t -(a+na_k)}{t^2+t+1}\,
dt = \frac{\big(q+n(q_k -Q)\big)^2}{2Q^2 (q+nq_k) (2q+nq_k)^2
(\gamma^2 +\gamma +1)} +O \left( \frac{1}{n^3 q q_k^4} \right) .
\end{equation}
We check that the contribution of error terms is negligible.
Note that when $n=1$ we must have $2q\geqslant 2Q-q_k \geqslant
2Q-\frac{3Q-q}{2} =\frac{Q+q}{2}$, so $q\geqslant \frac{Q}{3}$.
The errors in \eqref{A3.15}-\eqref{A3.19} add up to
\begin{equation*}
\begin{split}
& \ll \sum_{n=2}^\infty \sum_{k=1}^\infty \sum_{\gamma\in \FF(Q)}
\frac{q}{(n-1)^3 q^2 q_k^3} +\sum_{k=1}^\infty
\sum_{\gamma\in\FF(Q)} \frac{1}{Q^2 q_k^2} + \sum_{n=2}^\infty
\sum_{k=1}^\infty \sum_{\gamma\in\FF(Q)} \frac{q_{k+1}}{qq_{k-1}^2
q_k^2} \\ & + \hspace{-1pt} \sum_{k=1}^\infty \hspace{-1pt} \sum_{\gamma\in\FF(Q)}
\frac{q_{k+1}}{q^3 q_{k-1}^2}
+ \hspace{-1pt} \sum_{n=2}^\infty \hspace{-1pt} \sum_{k=1}^\infty \sum_{\gamma\in\FF(Q)}
\frac{q_{k+1}}{n(n-1)^2 qq_k^4} +\hspace{-1pt} \sum_{k=1}^\infty \hspace{-1pt}
\sum_{\gamma\in\FF(Q)} \frac{q_{k+1}}{q^3 q_k^2}
+ \hspace{-1pt} \sum_{n=1}^\infty \hspace{-1pt} \sum_{k=1}^\infty \frac{1}{n^3 q_k^4} +
\frac{1}{\vert I\vert} \cdot\frac{1}{Q} \\ & \ll
\sum_{\gamma\in\FF(Q)} \frac{Q}{q^3 q^{\prime 2}} +
\sum_{\frac{Q}{3} \leqslant q\leqslant Q} \frac{\varphi(q)}{q^3} +
Q^{c-1} \ll Q^{c-1}.
\end{split}
\end{equation*}

\subsubsection{}\label{SubSub12.2.4} For $k\geqslant 1$ and $q_{k+2}\leqslant 2Q$:
\begin{equation}\label{A3.20}
\int_{\gamma_{k+1}}^{t_{k-1}} \frac{q_{k+1}t -a_{k+1}}{t^2+t+1}\,
dt =\frac{(2Q-q_{k+1})^2}{2Q^2 q_{k-1}^2 q_{k+1}(\gamma^2
+\gamma+1)} + O \left( \frac{1}{qq_{k-1}^3 q_{k+1}} \right) .
\end{equation}

\subsubsection{}\label{SubSub13.2.5} For $n,k\geqslant 1$ and $ \frac{2(Q-q)}{n} \vee \Big(
Q+\frac{Q-q}{n+1}\Big) \leqslant q_k \leqslant Q+\frac{Q-q}{n}$:
\begin{equation}\label{A3.21}
\begin{split}
\int_{\mu_{k,n-1}}^{t_{k-1}} & \frac{w_{\AA_0}(t)\, dt}{t^2+t+1} =
\int_{\mu_{k,n-1}}^{u_0} \frac{a+2\eps-qt}{t^2+t+1}\, dt -
\int_{t_{k-1}}^{u_0} \frac{a+2\eps -qt}{t^2+t+1}\, dt \\
&=\frac{(n+1)Q-q-nq_k}{2Q^2 q_{k-1} \big( 2q+(n-1)q_k\big)
(\gamma^2 +\gamma +1)} \left( \frac{q+(n-1)(q_k -Q)}{2q+(n-1)q_k}
+ \frac{q_k-Q}{q_{k-1}} \right) \\ & \hspace{2cm} + O \Bigg( \frac{1}{q
\big( 2q+(n-1)q_k\big)^2} \bigg( \frac{Q-q}{Q^2
q_{k-1}}+\frac{1}{q\big( 2q+(n-1)q_k\big)} \bigg) \Bigg).
\end{split}
\end{equation}
When $n\geqslant 2$ the error is $\leqslant \frac{2}{q^2 q_{k-1}
q_k^2}$. When $n=1$ we can improve on the error term:
\begin{equation}\label{A3.22}
\begin{split}
\int_{\mu_{k,0}}^{t_{k-1}} & \frac{a+2\eps -qt}{t^2+t+1}\, dt = -
\int_{\frac{a+\eps}{q}}^{t_{k-1}} \frac{a+\eps-qt}{t^2+t+1}\, dt +
\int_{\frac{a+\eps}{q}}^{t_{k-1}} \frac{\eps \, dt}{t^2+t+1} \\
& =\frac{2Q-q_{k+1}}{4Q^2 qq_{k-1} (\gamma^2 +\gamma +1)} \left(
1-\frac{2Q-q_{k+1}}{2q_{k-1}} \right)+ O \left( \frac{1}{Qq^2
q_{k-1}^2} +\frac{1}{q^2 q_{k-1}^3} +\frac{1}{Q^2 q^2 q_{k-1}}
\right) .
\end{split}
\end{equation}
\begin{equation}\label{A3.23}
\int_{\mu_{k,n-1}}^{t_{k-1}} \frac{w_{\CC_k}(t)\, dt}{t^2+t+1} =
\frac{\big( (n+1)Q-q-nq_k\big)^2}{2Q^2 q_{k-1} \big(
2q+(n-1)q_k\big)^2 (\gamma^2 +\gamma +1)} +O \left(
\frac{1}{qq_{k-1}^2 \big( 2q+(n-1)q_k\big)^2} \right) .
\end{equation}
\begin{equation}\label{A3.24}
\int_{\mu_{k,n-1}}^{\nu_{k,n}} \frac{a+na_k -(q+nq_k)t}{t^2+t+1}\,
dt = \frac{\big( (n+1)Q-q-nq_k\big)^2}{2Q^2 (q+nq_k) \big( 2q+
(n-1) q_k\big)^2 (\gamma^2 +\gamma +1)} +O\left( \frac{1}{q^3
q_k^2} \right).
\end{equation}
\begin{equation}\label{A3.25}
\int_{\nu_{k,n}}^{t_{k-1}} \frac{(q+nq_k)t -(a+na_k)}{t^2+t+1}\, dt
=\frac{\big( (n+1)Q-q-nq_k\big)^2}{2Q^2 q_{k-1}^2 ( q+nq_k )
(\gamma^2 +\gamma +1)} + O \left( \frac{1}{Q^2 qq_k^2}\right) .
\end{equation}

\section{Appendix 4}

When $k\geqslant 1$ and $q_{k+1}\leqslant 2Q$:
\begin{equation}\label{A4.1}
\begin{split}
\int_{t_k}^{\gamma_{k+1}} \frac{w_{\CC_k}(t)\, dt}{t^2+t+1} &
=\int_{t_k}^{t_{k-1}} \frac{a_{k-1}-2\eps -q_{k-1}t}{t^2+t+1}\, dt
- \int_{\gamma_{k+1}}^{t_{k-1}} \frac{a_{k-1}-2\eps
- q_{k-1} t}{t^2+t+1}\, dt \\
& =\frac{q_{k+1}-Q}{2Q^2 q_k q_{k+1} (\gamma^2 +\gamma +1)} \left(
\frac{Q-q}{q_k} +\frac{2Q-q_{k+1}}{q_{k+1}} \right) + O \left(
\frac{1}{q q_{k-1} q_k^3} \right) .
\end{split}
\end{equation}
\begin{equation}\label{A4.2}
\begin{split}
\int_{\gamma_{k+1}}^{t_{k-1}} \frac{w_{\AA_0} (t)\, dt}{t^2+t+1} &
= \int_{\gamma_{k+1}}^{u_0} \frac{a+2\eps -qt}{t^2+t+1}\, dt -
\int_{t_{k-1}}^{u_0} \frac{a+2\eps -qt}{t^2+t+1} \, dt \\ &
 =\frac{2Q-q_{k+1}}{2Q^2 q_{k-1} q_{k+1}} \left(
\frac{q_{k+1} -Q}{q_{k+1}} +\frac{q_k -Q}{q_{k-1}} \right)
+O\left( \frac{1}{q^3 q_{k+1}^2} \right) .
\end{split}
\end{equation}

When $q_{k+1} >2Q$:
\begin{equation}\label{A4.3}
\begin{split}
\int_{t_k}^{\gamma_{k+1}} & \frac{2(a+\eps -qt)\, dt}{t^2+t+1}
=-\left( \int_{\gamma}^{\gamma_{k+1}} \frac{qt -a}{t^2+t+1}\, dt
-\int_{\gamma}^{t_k} \frac{qt -a}{t^2+t+1}\ dt\right)  +
\int_{t_k}^{\gamma_{k+1}}\frac{2\eps\, dt}{t^2+t+1} \\
& =\frac{q_{k+1}-Q}{2Q^2 q_k q_{k+1}(\gamma^2 +\gamma +1)} \left(
\frac{q_{k+2}-2Q}{q_k} + \frac{q_{k+1}-2Q}{q_{k+1}}\right) +O
\left( \frac{1}{q^3 q_k^2} \right) .
\end{split}
\end{equation}

\subsection{Estimates for the $\CC_O$ and the $\CC_\leftarrow$ contributions
when $(r,r_k)=(1,-1)$ and $w_{\AA_0} >w_{\BB_k}$}\label{SubA4.1}
Here $\lambda_{k,n}$ is as in \eqref{9.3}.

\subsubsection{}\label{SubSub14.1.1} If $Q+\frac{Q-q}{n+1}\leqslant q_{k+1}\leqslant 2Q$ and
$k\geqslant 1$, then
\begin{equation}\label{A4.4}
\begin{split}
& \int_{\lambda_{k,n+1}}^{\gamma_{k+1}} \frac{w_{\CC_k}(t)\,
dt}{t^2+t+1}  = \int_{\lambda_{k,n+1}}^{t_{k-1}}
\frac{a_{k-1}-2\eps -q_{k-1}t}{t^2+t+1}\, dt
-\int_{\gamma_{k+1}}^{t_{k-1}} \frac{a_{k-1}-2\eps
-q_{k-1}t}{t^2+t+1}\, dt \\ &  = \frac{2Q-q_{k+1}}{2Q^2
q_{k+1}(nq_{k+1}+2q)(\gamma^2 +\gamma +1)} \left(
\frac{(n+1)(2Q-q_{k+1})}{nq_{k+1}+2q} +\frac{2Q-q_{k+1}}{q_{k+1}}
\right) + O\left( \frac{1}{qq_{k-1}^2 q_{k+1}^2} \right) .
\end{split}
\end{equation}
When $k=0$ one can improve on the error as follows:
\begin{equation}\label{A4.5}
\begin{split}
& \int_{\lambda_{0,n+1}}^{\gamma_1} \frac{w_{\CC_0}(t)\,
dt}{t^2+t+1} = -\left( \int_{t_0}^{\gamma_1} \frac{q^\prime t
-a^\prime +2\eps}{t^2+t+1}\, dt -\int_{t_0}^{\lambda_{0,n+1}}
\frac{q^\prime t - a^\prime + 2 \eps}{t^2+t+1}\, dt \right) \\ &
\qquad \qquad \qquad \qquad +\left( \int_{\gamma}^{\gamma_1} \frac{qt -a}{t^2+t+1}\, dt
-\int_{\gamma}^{\lambda_{0,n+1}} \frac{qt -a}{t^2+t+1}\, dt \right)  \\
& =\frac{2Q-q_{1}}{2Q^2 q_{1}(nq_{1}+2q)(\gamma^2 +\gamma
+1)} \left( \frac{(n+1)(2Q-q_{1})}{nq_{1}+2q}
+\frac{2Q-q_{1}}{q_{1}} \right) + O\left( \frac{1}{Q^2 q^2
q^\prime} \right) .
\end{split}
\end{equation}

\subsubsection{}\label{SubSub13.1.2} When $n\geqslant 1$ and $Q+\frac{Q-q}{n+1} \leqslant
q_{k+1} \leqslant 2Q$:
\begin{equation}\label{A4.6}
\int_{\lambda_{k,n+1}}^{\gamma_{k+1}} \frac{w_{\AA_0} (t) -
w_{\BB_k}(t)}{t^2+t+1}\, dt = \frac{(2Q-q_{k+1})^2}{2Q^2
q_{k+1}(nq_{k+1}+2q)^2 (\gamma^2 +\gamma +1)} +O\left(
\frac{1}{qq_{k+1}^4}\right) .
\end{equation}
\begin{equation}\label{A4.7}
\begin{split}
\int_{\lambda_{k,n}}^{\lambda_{k,n+1}} & \frac{
( n q_{k+1}-q_{k-1})t -na_{k+1} +a_{k-1} -2\eps}{t^2+t+1}\, dt
\\ & = \frac{(2Q-q_{k+1})^2}{2Q^2 (nq_{k+1}+2q)^2 \big( (n-1) q_{k+1}
+ 2q \big)(\gamma^2 +\gamma +1)} + O \left( \frac{1}{Qq^2
q_{k+1}^2} \right).
\end{split}
\end{equation}

\subsubsection{}\label{A13.1.3} When $n\geqslant 1$ and $Q+\frac{Q-q}{n+1} \leqslant
q_{k+1} \leqslant Q+\frac{Q-q}{n}$ we have $0\leqslant
(n+1)(q_{k+1}-Q) +q-Q \leqslant \frac{Q-q}{n}$ and
\begin{equation}\label{A4.8}
\begin{split}
\int_{t_k}^{\lambda_{k,n+1}} & \frac{(n+1)a_{k+1}-a_{k-1}+2\eps
-\big( (n+1)q_{k+1}-q_{k-1} \big)t}{t^2+t+1}\, dt \\ & \qquad
\qquad = \frac{\big( (n+1)(q_{k+1}-Q)+q-Q\big)^2}{2Q^2 q_k^2 (
nq_{k+1}+2q) (\gamma^2 +\gamma +1)} +O\left( \frac{1}{n^4 qq_k^2
q_{k+1}^2}\right) .
\end{split}
\end{equation}
\begin{equation}\label{A4.9}
\begin{split}
& \int_{t_k}^{\lambda_{k,n+1}} \frac{(nq_{k+1}-q_{k-1})t -na_{k+1}
+a_{k-1} -2\eps}{t^2+t+1}\, dt =
\int_{\lambda_{k,n}}^{\lambda_{k,n+1}} \cdots
- \int_{\lambda_{k,n}}^{t_k} \cdots \\
&  =\frac{(n+1)(q_{k+1}-Q)+q-Q}{2Q^2 q_k ( nq_{k+1}+2q) (\gamma^2
+\gamma +1)} \left( \frac{2Q-q_{k+1}}{nq_{k+1}+2q}
+\frac{Q-q-n(q_{k+1}-Q)}{q_k}\right)  +O \left( \frac{1}{n^2 Q q
q_k q_{k+1}^2}\right) .
\end{split}
\end{equation}

\subsubsection{}\label{SubSub13.1.4} When $n\geqslant 1$ and $Q+\frac{Q-q}{n} \leqslant
q_{k+1} \leqslant 2Q$:
\begin{equation}\label{A4.10}
\begin{split}
& \int_{\lambda_{k,n}}^{\lambda_{k,n+1}} \frac{(n+1)a_{k+1}
-a_{k-1}+2\eps -\big( (n+1)q_{k+1}-q_{k-1}\big)t}{t^2+t+1}\, dt
\\ &  = \frac{(2Q-q_{k+1})^2}{2Q^2 (nq_{k+1}+2q) \big(
(n-1)q_{k+1}+2q\big)^2 (\gamma^2 +\gamma +1)} +O \left(
\frac{1}{(n^2-n+1)q^3 q_{k+1}^2} \right) .
\end{split}
\end{equation}
\begin{equation}\label{A4.11}
\begin{split}
& \int_{\lambda_{k,n}}^{\lambda_{k,n+1}} \frac{(nq_{k+1}-q_{k-1})t
-na_{k+1} +a_{k-1}-2\eps}{t^2+t+1}\, dt \\
& \qquad \qquad =\frac{(2Q-q_{k+1})^2}{2Q^2 (nq_{k+1}+2q)^2 \big(
(n-1)q_{k+1}+2q\big) (\gamma^2+\gamma +1)} +O \left( \frac{1}{n^2
Q q^2 q_{k+1}^2} \right) .
\end{split}
\end{equation}

\subsection {Estimates for the $\CC_O$ and the $\CC_\downarrow$ contributions when $(r,r_k)=(1,-1)$
and $w_{\AA_0} <w_{\BB_k}$}\label{SubA4.2} Here we take $\lambda_{k,n}$ as in
\eqref{9.8}, $k\geqslant 1$ and $q_{k+1} \leqslant 2Q$.

\begin{equation}\label{A4.12}
\begin{split}
\int_{\gamma_{k+1}}^{\lambda_{k,n+1}} \frac{w_{\CC_k} (t)\,
dt}{t^2+t+1} & = \int_{\gamma_{k+1}}^{t_{k-1}}
\frac{a_{k-1}-2\eps-q_{k-1}t}{t^2+t+1}\, dt -
\int_{\lambda_{k,n+1}}^{t_{k-1}} \frac{a_{k-1}-2\eps-q_{k-1}t}{t^2+t+1} \, dt \\
& = \frac{(2Q-q_{k+1})^2 \big( (2n+2)q_{k+1}+q_{k-1}\big)}{2Q^2
q_{k+1}^2 \big( (n+1)q_{k+1}+q_{k-1}\big)^2
(\gamma^2 +\gamma +1)} +O\left( \frac{1}{qq_{k-1}^2 q_{k+1}^2}
\right).
\end{split}
\end{equation}
\begin{equation}\label{A4.13}
\begin{split}
\int_{\gamma_{k+1}}^{\lambda_{k,n+1}} & \frac{w_{\BB_k} (t) -
w_{\AA_0}(t)}{t^2+t+1}\, dt  =
\int_{\gamma_{k+1}}^{\lambda_{k,n+1}} \frac{q_{k+1}t -a_{k+1}}{t^2+t+1}\, dt \\
& = \frac{(2Q-q_{k+1})^2}{2Q^2 q_{k+1} \big( (n+1)q_{k+1} +
q_{k-1}\big)^2 (\gamma^2 +\gamma +1)} + O \left(
\frac{1}{qq_{k+1}^4} \right) .
\end{split}
\end{equation}
\begin{equation}\label{A4.14}
\begin{split}
& \int_{\lambda_{k,n+1}}^{\lambda_{k,n}}
\frac{na_{k+1}+a_{k-1}-2\eps - (nq_{k+1}+q_{k-1})t}{t^2+t+1} \, dt
\\ & =\frac{(2Q-q_{k+1})^2}{2Q^2 ( nq_{k+1}+q_{k-1})\big(
(n+1)q_{k+1} +q_{k-1}\big)^2 (\gamma^2 +\gamma +1)} + O \left(
\frac{1}{(n^3+1)q q_{k-1}^2 q_{k+1}^2} \right) .
\end{split}
\end{equation}
\begin{equation}\label{A4.15}
\begin{split}
& \int_{\lambda_{k,n+1}}^{\lambda_{k,n}} \frac{\big(
(n+1)q_{k+1}+q_{k-1}\big)t - (n+1) a_{k+1}
-a_{k-1}+2\eps}{t^2+t+1}\, dt \\
& =\frac{(2Q-q_{k+1})^2}{2Q^2 ( nq_{k+1} +q_{k-1})^2 \big(
(n+1)q_{k+1}+q_{k-1} \big)(\gamma^2 +\gamma +1)} +O \left(
\frac{1}{(n^3+1)q q_{k-1}^2 q_{k+1}^2}\right) .
\end{split}
\end{equation}


\begin{thebibliography}{19}
\bibitem{BK} S. Blank and N. Krikorian, \emph{Thom's problem on irrational flows}, Internat.
J. Math. \textbf{4} (1993), 721--726.

\bibitem{BCZ} F. P. Boca, C. Cobeli and A. Zaharescu,
\emph{Distribution of lattice points visible from the origin},
Comm. Math. Phys. \textbf{213} (2000), 433--470.

\bibitem{BG} F. P. Boca and R. N. Gologan, \emph{On the distribution of the free path
length of the linear flow in a honeycomb}, Ann. Inst. Fourier
\textbf{59} (2009), 1043--1075.

\bibitem{BGZ0} F. P. Boca, R. N. Gologan and A. Zaharescu,
\emph{The average length of a trajectory in a certain billiard in
a flat two-torus}, New York J. Math. \textbf{9} (2003), 303--330.

\bibitem{BGZ} F. P. Boca, R. N. Gologan and A. Zaharescu,
\emph{The statistics of the trajectory of a billiard in a flat
two-torus}, Comm. Math. Phys. \textbf{240} (2003), 53--73.

\bibitem{BZ1} F. P. Boca and A. Zaharescu, \emph{On the
correlations of directions in the Euclidean plane}, Trans. Amer.
Math. Soc. \textbf{358} (2006), 1797--1825.

\bibitem{BZ2} F. P. Boca and A. Zaharescu, \emph{The distribution of the free path lengths in the
periodic two-dimensional Lorentz gas in the small-scatterer
limit}, Comm. Math. Phys. \textbf{269} (2007), 425--471.

\bibitem{BGW}
J. Bourgain, F. Golse and B. Wennberg, \emph{On the distribution of free path
lengths for the periodic Lorentz gas}, Comm. Math. Phys. \textbf{190} (1998), 491--508.

\bibitem{CG1} E. Caglioti and F. Golse, \emph{On the distribution of free path
lengths for the periodic Lorentz gas. III}, Comm. Math. Phys.
\textbf{236} (2003), 199--221.

\bibitem{CG2} E. Caglioti and F. Golse, \emph{The Boltzmann-Grad limit
of the periodic Lorentz gas in two space dimensions}, C. R. Math.
Acad. Sci. Paris \textbf{346} (2008), 477--482.

\bibitem{Dah} P. Dahlqvist, \emph{The Lyapunov exponent in the
Sinai billiard in the small scatterer limit}, Nonlinearity
\textbf{10} (1997), 159--173.

\bibitem{Go1} F. Golse, \emph{The periodic Lorentz gas in the
Boltzmann-Grad limit}, Proc. ICM (Madrid, 2006), vol. 3, EMS, Z\"
urich 2006, pp. 183--201.

\bibitem{Go2} F. Golse, \emph{Recent results on the periodic Lorentz
gas}, preprint arXiv:0906.0191.

\bibitem{Lor}
H. A. Lorentz, \emph{Le mouvement des \' electrons dans les m\'
etaux}. Arch. N\' eerl. \textbf{10} (1905), p. 336. Reprinted in
\emph{Collected papers}, Vol. \textbf{3}, The Hague: Martinus
Nijhoff, 1936.

\bibitem{Mar}
J. Marklof, \emph{Kinetic transport in crystals},
preprint math-ph/0909.3463, Proceedings of the XVIth International
Congress of Mathematical Physics, Prague 2009.

\bibitem{MS1} J. Marklof and A. Str\" ombergsson, \emph{The
distribution of free path lengths in the periodic Lorentz gas and
related lattice point problems}, preprint arXiv:0706.4395, to
appear in Ann. Math.

\bibitem{MS2} J. Marklof and A. Str\" ombergsson, \emph{The Boltzmann-Grad limit
of the periodic Lorentz gas}, preprint arXiv:0801.0612, to appear in Ann. Math.

\bibitem{MS3} J. Marklof and A. Str\" ombergsson, \emph{Kinetic
transport in the two-dimensional periodic Lorentz gas},
Nonlinearity \textbf{21} (2008), 1413--1422.

\bibitem{Po}
G. P\'olya,  \emph{Zahlentheoretisches und
wahrscheinlichkeitstheoretisches \" uber die sichtweite im walde},
Arch. Math. Phys. \textbf{27} (1918), 135--142.

\bibitem{SU} J. Smillie, C. Ulcigrai, \emph{Symbolic coding for linear trajectories
in the regular octogon}, preprint math.DS/0905.0871.

\end{thebibliography}
\end{document}